\documentclass[11pt,a4paper,reqno]{amsart} 
\usepackage{amsthm,amsmath,amsfonts,amssymb,bbm,cancel,tikz}
\usepackage{cases,color,xfrac}

\usepackage{vmargin}
\setmarginsrb{2cm}{1cm}{2cm}{3cm}{1cm}{1cm}{2cm}{2cm}

\usepackage{mathtools}
\usepackage{mathrsfs}
\usepackage{amsaddr}

\usepackage{hyperref}
\hypersetup{linkcolor=blue, colorlinks=true,citecolor = red}
\usepackage[capitalise]{cleveref}
\usepackage{upref}
\usepackage{url}

\usepackage{lineno}

\numberwithin{equation}{section}
\newtheorem{Theorem}{Theorem}[section]
\newtheorem{Lemma}[Theorem]{Lemma}
\newtheorem{Proposition}[Theorem]{Proposition}
\newtheorem{Corollary}[Theorem]{Corollary}
\newtheorem{Definition}[Theorem]{Definition}
\newtheorem{Remark}[Theorem]{Remark}
\renewcommand{\leq}{\leqslant}

\renewcommand{\geq}{\geqslant}
\renewcommand{\ge}{\geqslant}
\renewcommand{\Re}{\operatorname{Re}}

\DeclareMathOperator{\supp}{supp}
\DeclareMathOperator{\Cof}{Cof}
\renewcommand{\div}{\operatorname{div}}
\newcommand{\Id}{\operatorname{Id}}
\newcommand{\ds}{\displaystyle}
\newcommand{\mr}{\mathcal{R}}
\newcommand{\mf}{\mathcal{F}}
\newcommand{\orh}{\overline{\rho}}
\newcommand{\ot}{\overline{\vartheta}}
\newcommand{\ms}{\mathcal{S}}
\newcommand{\mx}{\mathcal{X}}
\newcommand{\rhob}{\rho_{\flat}}

\renewcommand{\mid}{\ ; \ }

\newcommand{\norm}[1]{\left\Vert#1\right\Vert}
\newcommand{\vect}[1]{\left[#1\right]^{\top}}
\newcommand{\rd}{{\rm d}}

\title[]{Existence and uniqueness of strong solutions for the system of interaction between a compressible Navier-Stokes-Fourier fluid
and a damped plate equation}

\author{Debayan Maity}
\address{Centre for Applicable Mathematics, TIFR, \\
Post Bag No. 6503, GKVK Post Office, Bangalore-560065, India.
\\
(\href{mailto:Debayan.Maity@math.cnrs.fr}{debayan.maity@math.cnrs.fr})}

\author{Tak\'eo Takahashi}
\address{Universit\'e de Lorraine, CNRS, Inria, IECL, F-54000 Nancy, France.\\
(\href{mailto:takeo.takahashi@inria.fr}{takeo.takahashi@inria.fr})}

\date{\today}
\thanks{Debayan Maity was partially supported by INSPIRE faculty fellowship (IFA18-MA128) and by Department of Atomic Energy, 
Government of India, under project no. 12-R \& D-TFR-5.01-0520. Tak\'eo Takahashi was partially supported by the ANR research project IFSMACS (ANR-15-CE40-0010).}

\begin{document}                  

\maketitle

\begin{abstract}
The article is devoted to the mathematical analysis of  a fluid-structure interaction system where 
the fluid is compressible and heat conducting and where the structure is deformable and located on a part of the boundary of the fluid domain. 
The fluid motion is modeled by the compressible Navier-Stokes-Fourier system and the structure displacement is described by a structurally damped plate equation.
Our main results are the existence of strong solutions in an $L^p-L^q$ setting for small time or for small data. Through a change of variables and a fixed point argument, the proof of the main results is mainly based on the maximal regularity property of the corresponding linear systems. For small time existence, this property is obtained by
decoupling the linear system into several standard linear systems whereas for global existence and for small data, the maximal regularity property is proved by showing that the corresponding linear coupled {\em fluid-structure} operator is $\mathcal{R}-$sectorial.
\end{abstract}

\bigskip


\tableofcontents

\section{Introduction} 
In this work, we study the interaction between a viscous compressible heat conducting fluid and a viscoelastic structure located on a part of the fluid domain boundary. 
More precisely, we consider a smooth bounded domain $\mathcal{F} \subset \mathbb{R}^{3}$ such that its boundary $\partial \mathcal{F}$ contains a flat part 
$\Gamma_S:=\ms \times \{0\},$ where $\ms$ is a smooth bounded domain of $\mathbb{R}^{2}.$ We also set
$$
\Gamma_0=\partial \mathcal{F} \setminus \overline{\Gamma_S}.
$$ 
The set $\Gamma_{0}$ is rigid and remains unchanged whereas 
on the flat part, 
we assume that there is a plate that can deform only in the transversal direction, and 
if we denote by $\eta$ the corresponding displacement, then $\Gamma_S$ is transformed into
\begin{equation*}
\Gamma_{S}(\eta) :=  \left\{ [x_1,x_2, \eta(x_1,x_2)]^\top \ ; \  [x_1,x_2]^\top \in \ms \right\}.
\end{equation*}
In our study, we consider only displacements $\eta$ regular enough and satisfying the boundary conditions (the plate is clamped):
\begin{equation}\label{rev0.0}
\eta= \nabla_s \eta \cdot n_{S} = 0 \quad \text{on} \ \partial \mathcal{S}
\end{equation}
and a condition insuring that the deformed plate does not have any contact with the other part of the boundary of the fluid domain:
\begin{equation}\label{rev0.1}
\Gamma_{0} \cap \Gamma_{S}(\eta) = \emptyset.
\end{equation}
We have denoted by $n_{S}$ the unitary exterior normal to $\partial \mathcal{S}$ and in the whole article we add the index $s$ in the gradient and in the Laplace operators if they apply to functions defined on $\mathcal{S}\subset \mathbb{R}^2$ (and we keep the usual notation for functions defined on a domain of $\mathbb{R}^3$).

With the above notations and hypotheses, $\Gamma_{0} \cup \overline{\Gamma_{S}(\eta)}$ corresponds to a closed simple and regular surface whose interior is the fluid domain
$\mathcal{F}(\eta)$. In what follows, we consider that $\eta$ is also a function of time and its evolution is governed by a damped plate equation.

\medskip

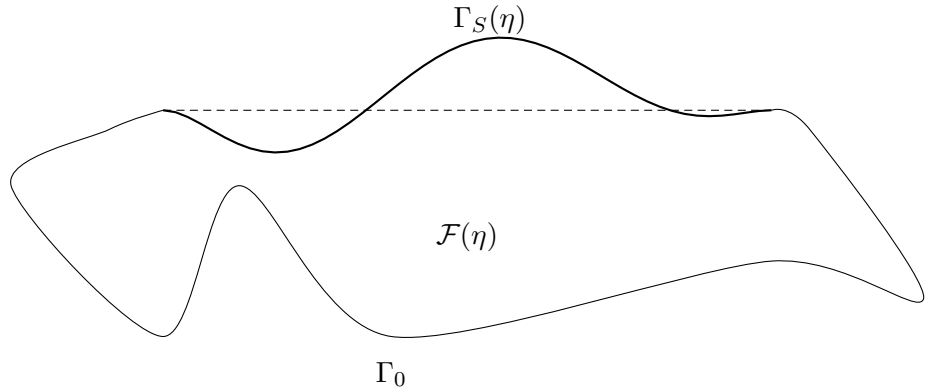
\begin{center}
\begin{tikzpicture}
\draw [densely dashed] (2,2)--(10,2);
\draw[draw=black]  plot [smooth] coordinates {(2,2) (1.3,1.75) (0,1) (2,-1) (3,1) (5,-1) (10,0) (12,-0.5) (10.5, 1.75) (10,2)};
\draw [domain=2:10,samples=50,thick] plot (\x,{2-4*((\x-2)/8)^2*((\x-10)/8)^2*(\x-(2*2+10)/3)*(\x-(2+5*10)/6)});
\draw (6,0) node[above] {$\mathcal{F}(\eta)$};
\draw (6.3,3.2) node {$\Gamma_S(\eta)$};
\draw (5,-1.5) node {$\Gamma_0$};
\end{tikzpicture}
\end{center}

In $\mathcal{F}(\eta(t))$, we assume that there is a viscous compressible heat conducting fluid and we denote by
$\widetilde \rho$, $\widetilde v$, and $\widetilde \vartheta$ respectively its density, velocity and temperature.
The equations modeling the evolution of these quantities can be written as follows:
\begin{equation} \label{eq:main01}
\begin{dcases}
\partial_{t} \widetilde \rho + \div (\widetilde \rho \widetilde v ) = 0, &  t> 0,\ x \in  \mathcal{F}(\eta(t)), 
\\ 
\widetilde \rho \left( \partial_{t} \widetilde v + (\widetilde v \cdot \nabla) \widetilde v \right)  
	- \div \mathbb{T}(\widetilde v, \widetilde \pi) = 0, &  t> 0, \ x \in \mathcal{F}(\eta(t)), 
\\
c_{v} \widetilde \rho \left(\partial_{t} \widetilde \vartheta + \widetilde v \cdot \nabla \widetilde \vartheta \right) 
	+ \widetilde \pi \div \widetilde v  - \kappa \Delta \widetilde \vartheta  
	= \alpha (\div \widetilde v)^{2} + 2\mu \left|\mathbb{D}\widetilde v\right|^2 &    t> 0,\  x \in  \mathcal{F}(\eta(t)), \\
\partial_{tt} \eta + \Delta_{s}^{2} \eta - \Delta_{s} \partial_{t}\eta = \mathbb{H}_{\eta}(\widetilde v,\widetilde \pi) 
	& t > 0,\  s \in \ms,
\end{dcases}
\end{equation}
with the boundary conditions
\begin{equation} \label{eq:bc}
\begin{dcases}
\widetilde  v(t, s, \eta(t,s)) = \partial_{t} \eta(t,s) e_{3} & t > 0,\  s \in \ms,  \\
\widetilde  v = 0 & t > 0,\  x \in \Gamma_{0}, \\
\frac{\partial \widetilde \vartheta}{\partial \widetilde n} (t,x) = 0  & t > 0, \  x \in \partial  \mf({\eta(t)}),\\
\eta  = \nabla_s \eta \cdot n_{S}= 0 & t > 0,\   s \in \partial \ms,
\end{dcases}
\end{equation}
and the initial conditions
\begin{equation} \label{eq:ini}
\begin{dcases}
 \eta(0,\cdot) = \eta_{1}^{0}, \quad \partial_{t}\eta(0,\cdot) = \eta_{2}^{0}  \quad \text{in} \ \ms,  \\
 \widetilde \rho(0,\cdot) = \widetilde \rho^{0}, \quad  \widetilde  v(0,\cdot) = \widetilde  v^{0}, \quad  
 \widetilde \vartheta(0,\cdot) = \widetilde \vartheta^{0} \quad \text{in}\  \mf(\eta_{1}^{0}).
 \end{dcases}
\end{equation}
In the above system $(e_1,e_2,e_3)$ is the canonical basis of $\mathbb{R}^3$, the fluid stress tensor is defined by
$$
\mathbb{T}(\widetilde v,\widetilde \pi) = 2\mu \mathbb{D}(\widetilde v) + ( \alpha \div \widetilde v - \widetilde \pi)  I_3, \quad \mathbb{D}(\widetilde v) = \frac12 \left( \nabla\widetilde  v + \nabla \widetilde v^{\top}\right),
$$
and the pressure law is given by 
\begin{equation} \label{eq:pressure-law}
\widetilde \pi = R_0 \widetilde \rho \widetilde \vartheta+ \pi_0.
\end{equation}
The above physical constants satisfy 
\begin{equation} \label{visco-relation}
R_0>0,
\quad \mu > 0 \mbox{ (viscosity)}, \quad  \alpha + \frac{2}{3} \mu > 0, \quad \kappa>0, \quad c_v>0, \quad \pi_0\in \mathbb{R}.
\end{equation}
For any matrix $A,B\in \mathcal{M}_d(\mathbb{R})$, we use the canonical scalar product and norm:
$$
A:B = \sum_{i,j} a_{ij} b_{ij}, \quad |A|=\sqrt{A:A}.
$$
We have set  
$$
\nabla_s = [\partial_{y_1},\partial_{y_2}]^\top, 
\quad 
\Delta_{s}=\partial_{y_1}^2+\partial_{y_2}^2.
$$
The function $\mathbb{H}$ is defined by 
\begin{equation} \label{eq:force-fsi}
\mathbb{H}_{\eta}(\widetilde v,\widetilde \pi) =  - \sqrt{1 + |\nabla_{s} \eta|^{2}} \left( \mathbb{T}(\widetilde v,\widetilde \pi) \widetilde{n} \right)|_{\Gamma_S(\eta(t))} \cdot e_{3},
\end{equation}
where 
\begin{equation*}
\widetilde{n} =  \frac{1}{\sqrt{1 + |\nabla_{s} \eta|^{2}}} \left[ -\nabla_{s} \eta , 1 \right]^{\top},
\end{equation*}
is the unit normal to $\Gamma_{S}(\eta(t))$ outward $\mf(\eta(t)).$ 
Let us mention that the boundary conditions \eqref{eq:bc} are obtained by assuming that the fluid does not slip on the boundaries and that the plate is thermally insulated.

Fluid-structure interaction problems have been an active area of research among the engineers, physicist and mathematicians over the last few decades
due to the numerous practical applications and the corresponding scientific challenges. 
The type of model considered in this article  appears in the design of many engineering structures, e.g aircraft and bridges etc., (\cite{AGW})  as well as in biomechanics (\cite{bio-book}).

Let us mention some related works from the literature. In the last two decades, there has been considerable number of works on similar fluid-structure systems where the fluid is modelled by incompressible flows. We refer to, for instance \cite{biomedical} and references therein for a concise description of recent progress regarding incompressible flows interacting with deformable structure (beam or plate) located on a part of the fluid domain boundary. Moreover, in some recent articles (\cite{GHL, BT19, BT19-bis}) existence and uniqueness of strong solutions (either local in time or for small initial data) were proved without the additional damping term (i.e., without the term $-\Delta_{s} \partial_{t} \eta$) in the beam/plate equation. 

Concerning compressible fluids interacting with plate/beam equations through boundary of the fluid domain, there are only few results available in the literature. Global existence of weak solutions until the structure touches the boundary of the  fluid domain were proved in \cite{FO, BS18}.  Local in time existence of strong solutions in the corresponding  $2D/1D$ case was recently obtained  in \cite{SM18}. Well-posedness and stability of linear compressible fluid-structure systems were studied in \cite{IC, AGW}.

Let us mention that all the above mentioned works correspond to a ``Hilbert'' space framework. In this article, we are interested in studying existence and uniqueness of strong solutions, local in time or global in time for small initial data, within an ``$L^{p}-L^{q}$'' framework. More precisely, we look for solutions in the spaces of functions which are $L^{p}$ with respect to time and $L^{q}$ with respect to space variable, with arbitrary $p, q > 1.$  In the context of fluid-solid interaction problems, there are only few articles available in the literature that studies well-posedness in an  $L^{p}-L^{q}$ framework. Let us mention
 \cite{GGH, MT18} (viscous incompressible fluid and rigid bodies), \cite{HiebMur, Tucsnak17a, HMTT} (viscous compressible fluid and rigid bodies) and \cite{MT20, DS19} (viscous incompressible fluid interacting with viscoelastic structure located at the boundary of the fluid domain). In fact,  this article is a compressible counterpart of our previous work  \cite{MT20}.

The main novelties that we bring in this article are :
\begin{itemize}
\item The full nonlinear free boundary system coupling viscous compressible Navier-Stokes-Fourier system and a viscoelastic structure located on a part of the fluid domain has not, at the best of our knowledge,  been studied in the literature. 
\item The existence and uniqueness results are proved in $L^{p}-L^{q}$ setting. 
\item Global in time existence for small initial data seems to be a new result for such coupled systems. 
\end{itemize}
Let us emphasize that using the $L^{p}-L^{q}$ setting allows us to weaken the regularity on the initial conditions (see for instance \cite{SM18}).
Moreover, this ``$L^{p}-{L^{q}}$'' framework is interesting even for studies in fluid-structure interaction problems done in the ``$L^2-L^2$'' framework: let us quote for instance
the uniqueness of weak solutions (\cite{GlassSueur, Bravin}), the asymptotic behavior for large time (\cite{EHL, EMT}), and the
asymptotic behavior for small structures (\cite{LacTak}). 

\subsection{Notation} 
To state our main results, we need to introduce some notations for the functional spaces.  For $\Omega\subset \mathbb{R}^{n}$ is an open set,  $q > 1$ and $ k \in \mathbb{N},$ we denote by $L^{q}(\Omega)$ and $W^{k,q}(\Omega)$ the standard Lebesgue and Sobolev spaces respectively. $W^{s,q}(\Omega)$, with $q>1$ and 
$s\in \mathbb{R}_+^*$, denotes the usual Sobolev-Slobodeckij space.  Moreover, $W^{k,q}_{0}(\Omega)$ is the completion of $C_{c}^{\infty}(\Omega)$ with respect to the $W^{k,q}(\Omega)$ norm. 
Let $k,m \in \mathbb{N}$, $k<m$.
For $1\leqslant p < \infty$, $1\leqslant q < \infty$, we consider the standard definition of the Besov spaces by real interpolation of Sobolev spaces
$$
B^{s}_{q,p}(\mathcal{F}) =  \left( W^{k,q}(\mathcal{F}), W^{m,q}(\mathcal{F})\right)_{\theta, p} \mbox{ where }  
s = (1 - \theta) k + \theta m, \quad \theta\in (0,1). 
$$
We refer to \cite{AF} and \cite{Triebel} for a detailed presentation of the Besov spaces. 
We denote by $C^k_b$ is the set of continuous and bounded functions with derivatives continuous and bounded up to the order $k$. 
For $s \in (0,1)$ and a Banach space $U,$ $F^{s}_{p,q}(0,T,U)$
stands for $U$ valued Lizorkin-Triebel space. For precise definition
of such spaces we refer to \cite{Triebel}. 
If $T\in (0,\infty]$, we set 
\begin{gather*}
W^{1,2}_{p,q}((0,T) ; \mathcal{F}) = L^{p}(0,T;W^{2,q}(\mathcal{F})) \cap W^{1,p}(0,T;L^{q}(\mathcal{F})), \\
W^{2,4}_{p,q}((0,T) ; \mathcal{S}) = L^{p}(0,T;W^{4,q}(\mathcal{S})) \cap W^{1,p}(0,T;W^{2,q}(\mathcal{S})) \cap W^{2,p}(0,T;L^{q}(\mathcal{S})), \\
W^{1,2}_{p,q}((0,T) ; \mathcal{S}) = L^{p}(0,T;W^{2,q}(\mathcal{S})) \cap W^{1,p}(0,T;L^{q}(\mathcal{S})). 
\end{gather*}
We have the following embeddings (see, for instance, \cite[Theorem 4.10.2, p.180]{Amann}),
\begin{equation}\label{embama}
W^{1,2}_{p,q}((0,T) ; \mf) \hookrightarrow C_{b}^0([0,T);B^{2(1-1/p)}_{q,p}(\mf)),
\end{equation}
\begin{equation}\label{embama2}
W^{2,4}_{p,q}((0,T) ; \mathcal{S}) \hookrightarrow C_{b}^0([0,T);B^{2(2-1/p)}_{q,p}(\mathcal{S}))
\cap C_{b}^1([0,T);B^{2(1-1/p)}_{q,p}(\mathcal{S})).
\end{equation}
In particular, in what follows, we use the following norm for $W^{1,2}_{p,q}((0,T) ; \mathcal{F})$:
$$
\|f\|_{W^{1,2}_{p,q}((0,T) ; \mathcal{F})}:=
\|f\|_{L^{p}(0,T;W^{2,q}(\mathcal{F}))}
+
\|f\|_{W^{1,p}(0,T;L^{q}(\mathcal{F}))}
+
\|f\|_{C_{b}^0([0,T);B^{2(1-1/p)}_{q,p}(\mf))}
$$
and we proceed similarly for the two other spaces. 

We also introduce functional spaces with time decay. We write for any $\beta\in \mathbb{R}$ 
$$
\mathbb{E}_\beta : \mathbb{R} \to \mathbb{R}, t\mapsto e^{\beta t}.
$$
We denote by $L^p_{\beta}(0,\infty)$ the space $\mathbb{E}_{-\beta}L^p(0,\infty)$, that is the set of functions $f$ such that $t\mapsto e^{\beta t}f(t)$ is in $L^p(0,\infty)$.
The corresponding norm is
$$
\|f\|_{L^p_{\beta}(0,\infty)}:=\|\mathbb{E}_{\beta} f\|_{L^p(0,\infty)}.
$$
We proceed similarly for all spaces on $(0,\infty)$ or on $[0,\infty)$. 

Finally, we also need to introduce functional spaces for the fluid density, velocity and temperature depending on the displacement $\eta$ of the structure. 
Assume $T \in (0, \infty]$ and that $\eta\in W^{2,4}_{p,q}((0,T) ; \mathcal{S})$ satisfies \eqref{rev0.0} and  \eqref{rev0.1}.
We show in \cref{sec:local} that there exists a mapping $X=X_\eta$ such that $X(t,\cdot)$ is a $C^{1}$-diffeomorphism from $\mf$ onto $\mf(\eta(t))$
and for any function $\widetilde f$ defined for $t\in (0,T)$ and $x\in \mf(\eta(t))$, we then define
$$
f(t,y)  := \widetilde f(t,X(t,y)) \quad (t\in (0,T), \ y\in \mf).
$$
Then we define the following sets as follows
\begin{align*}
\widetilde f\in W^{r,p}(0,T;W^{s,q}(\mf(\eta(\cdot) ))) &\quad \text{if} \quad f \in W^{r,p}(0,T;W^{s,q}(\mf)), \\
\widetilde f\in W^{1,2}_{p,q}((0,T) ; \mf(\eta(\cdot) )) &\quad \text{if} \quad f \in W^{1,2}_{p,q}((0,T) ; \mf), \\
\widetilde f\in C^0([0,T];B^{2(1-1/p)}_{q,p}(\mf(\eta(\cdot)))) &\quad \text{if} \quad  f \in C^0([0,T];B^{2(1-1/p)}_{q,p}(\mf)
\end{align*}
and a similar definition for all the other spaces.

\subsection{Statement of the main results}
Let us give the conditions we require on $(p,q)$ and on the initial data for the system \eqref{eq:main01}--\eqref{eq:force-fsi}: 
\begin{equation}\label{pqpqpq}
2 < p < \infty, \quad 3 < q< \infty, \quad \frac{1}{p} + \frac{1}{2q} \neq \frac12,
\end{equation}
\begin{gather}
\eta_{1}^{0} \in B^{2(2-1/p)}_{q,p}(\ms), \quad \eta_{2}^{0} \in B^{2(1-1/p)}_{q,p}(\ms), \quad  
\widetilde{\rho}^0 \in W^{1,q}({\mf({\eta_{1}^{0}})}), \quad
\min_{\overline{\mf(\eta_{1}^{0})}}  \widetilde{\rho}^{0} > 0,  \label{cc1}\\
\widetilde v^{0} \in B^{2(1-1/p)}_{q,p}(\mathcal{F}({\eta_{1}^{0}}))^{3}, \qquad  \widetilde \vartheta^{0} \in B^{2(1-1/p)}_{q,p}(\mathcal{F}({\eta_{1}^{0}})),   \label{cc2}
\end{gather} 
with the compatibility conditions 
\begin{gather}
\eta_{1}^{0} = \ds \nabla \eta_{1}^{0} \cdot n_{S} = \eta_{2}^{0}  = 0 \mbox{ on } \partial \ms,  \quad \widetilde v^{0} = 0 \mbox{ on } \Gamma_{0}, \quad \widetilde v^{0} = \eta_{2}^{0} e_{3} \mbox{ on } \Gamma_{S}(\eta_{1}^{0}),  \label{cc3} \\
\ds \nabla \eta_{2}^{0} \cdot n_{S} \mbox{ on } \partial \ms \quad \text{and} \quad \frac{\partial \widetilde \vartheta^{0}}{\partial n} = 0 \mbox{ on } \partial \mf(\eta_{1}^{0}), \quad  \mbox{ if } \quad  \frac{1}{p} + \frac{1}{2q} < \frac{1}{2}. \label{cc4}
\end{gather}
Note that, all the traces in the above relation makes sense for our choice of $p$ and $q$ (see for instance, \cite[p. 200]{Triebel}). 

We also need a geometrical condition on the initial deformation.
Using that $\mathcal{F}$ is a smooth domain, there exist 
two smooth surfaces $\eta_- : \mathcal{S} \to \mathbb{R}_-^*$,  $\eta_+ : {\mathcal{S}} \to \mathbb{R}_+^*$ such that
\begin{equation}\label{geo00-bis}
\left\{[y_1,y_2,y_3]^\top\in {\mathcal{S}}\times \mathbb{R} \ ; \ y_3\in (\eta_-(y_1,y_2),0)\right\} \subset \mathcal{F},
\end{equation}
\begin{equation}\label{geo01-bis}
\left\{[y_1,y_2,y_3]^\top\in {\mathcal{S}}\times \mathbb{R} \ ; \ y_3\in (0,\eta_+(y_1,y_2))\right\} \subset \mathbb{R}^3\setminus \overline{\mathcal{F}}.
\end{equation}
Then our geometrical condition on the initial deformation writes
\begin{equation}\label{geo02}
\eta_- < \eta_1^0 <\eta_+ \quad \text{in} \ {\mathcal{S}}.
\end{equation}
This yields in particular that $\Gamma_{0} \cap \Gamma_S(\eta_{1}^{0}) = \emptyset$.
According to the geometry, we can in some situation remove the condition $\eta_1^0<\eta_+$.
Note that this condition is not a smallness condition, $\eta_+$ and $\eta_-$ do not need to be small.

Our main results are the following two theorems. The first one is the local in time existence and uniqueness : 
\begin{Theorem} \label{th:local-main}
Assume $(p,q)$ satisfies \eqref{pqpqpq} and that $[\widetilde \rho^{0},  \widetilde v^{0}, \widetilde \vartheta^{0}, \eta_{1}^{0}, \eta_{2}^{0}]^{\top}$ 
satisfies \eqref{cc1}--\eqref{cc4} and \eqref{geo02}.  
Then there exists  $T > 0,$ depending only on initial data, such that the 
system \eqref{eq:main01}--\eqref{eq:force-fsi} admits a unique strong solution 
$[\widetilde \rho, \widetilde v, \widetilde \vartheta, \eta]^{\top}$ satisfying 
\begin{gather*}
\widetilde \rho \in W^{1,p}(0,T; W^{1,q}(\mf(\eta(\cdot)))), \quad 
\widetilde v \in W^{1,2}_{p,q}((0,T) ; \mf(\eta(\cdot) ))^{3}, \\
\vartheta \in W^{1,2}_{p,q}((0,T) ; \mf(\eta(\cdot) )), \quad
\eta \in W^{2,4}_{p,q}((0,T) ; \mathcal{S}), \\
\Gamma_{0} \cap \Gamma_S(\eta(t)) = \emptyset \quad (t\in [0,T]), \quad
\widetilde \rho(t,x)>0 \quad (t\in [0,T], \ x\in \overline{\mf(\eta(t))}).
\end{gather*}
\end{Theorem}

Our second main result states the global existence and uniqueness under a smallness condition on the initial data. 
Let $\orh$ and $\ot$  be two given positive  constants. Let us take in the pressure law \eqref{eq:pressure-law}
\begin{equation} \label{pressure-new}
{\pi}_{0} =  - R_0 \orh \ot. 
\end{equation}
With the above choice of $\pi_{0}$, $\vect{\widetilde \rho, \widetilde v, \widetilde \vartheta, \eta} = \vect{\orh, 0, \ot, 0}$  is a steady state solution to the system \eqref{eq:main01}--\eqref{eq:force-fsi}. 

Then our result states as follows:
\begin{Theorem} \label{th:global-main}
Assume $(p,q)$ satisfies \eqref{pqpqpq}
and assume that $\orh$ and $\ot$ are two given positive constants such that \eqref{pressure-new} holds. 
Then there exist $\beta>0$  and $R > 0$ such that 
for any  $[\widetilde \rho^{0},  \widetilde v^{0}, \widetilde \vartheta^{0}, \eta_{1}^{0}, \eta_{2}^{0}]^{\top}$ satisfying \eqref{cc1}--\eqref{cc4}, and 
\begin{multline}\label{mainsmall}
\norm{\widetilde \rho^{0} - \orh}_{W^{1,q}(\mf(\eta_{1}^{0}))} + \norm{\widetilde v^{0}}_{B^{2(1-1/p)}_{q,p}(\mf(\eta_{1}^{0}))^{3}} +  \norm{\widetilde \vartheta^{0} - \ot}_{B^{2(1-1/p)}_{q,p}(\mf(\eta_{1}^{0}))} \\
+  \norm{\eta_{1}^{0}}_{B^{2(2-1/p)}_{q,p}(\ms)} + 
 \norm{\eta_{2}^{0}}_{B^{2(1-1/p)}_{q,p}(\ms)} \leqslant R,
\end{multline}
the system \eqref{eq:main01}--\eqref{eq:force-fsi} admits a unique strong solution $\vect{\widetilde \rho, \widetilde v, \widetilde \vartheta, \eta}$
satisfying 
\begin{gather*}
\widetilde\rho \in C^0_{b}([0,\infty);W^{1,q}(\mf(\eta(\cdot)))), \; \nabla \widetilde \rho \in {W^{1,p}_{\beta}(0,\infty;L^{q}(\mf(\eta(\cdot))))}^{3}, \; \partial_t \widetilde \rho \in {L^{p}_{\beta}(0,\infty;W^{1,q}(\mf(\eta(\cdot))))},  
\\
\widetilde v \in   W^{1,2}_{p,q,\beta}((0,T) ; \mf(\eta(\cdot)))^3,
\\
\widetilde \vartheta \in C^{0}_{b}([0,\infty);B^{2(1-1/p)}_{q,p}(\mf(\eta(\cdot)))), \quad \nabla \widetilde \vartheta \in L^{p}_{\beta}(0,\infty ; W^{1,q}(\mf(\eta(\cdot))))^{3}, \quad  
	\partial_t \widetilde \vartheta \in {L^{p}_{\beta}(0,\infty ; L^q(\mf(\eta(\cdot))))},  
\\
\eta \in C^{0}_{b}([0,\infty);B^{2(2-1/p)}_{q,p}(\ms)), \quad \eta \in L^{p}_{\beta}(0,\infty;W^{4,q}(\ms)) + L^{\infty}(0,\infty;W^{4,q}(\ms)), \\
  \partial_{t}\eta \in W^{1,2}_{p,q,\beta}((0,\infty) ; \mathcal{S}),
\end{gather*}
and
$$
\Gamma_{0} \cap \Gamma_S(\eta(t)) = \emptyset \quad (t\in [0,\infty)), \quad
\widetilde \rho(t,x)>0 \quad (t\in [0,\infty), \ x\in \overline{\mf(\eta(t))}).
$$
\end{Theorem}

\begin{Remark}
Let us make the following remarks on the above results:
\begin{enumerate}
\item Note that, in \cref{th:local-main} we do not need initial displacement of the plate $\eta_1^0$ to be zero. 
This is a difference with respect to previous works, for instance \cite{SM18} or our previous work \cite{MT20} (with an incompressible fluid). Here we manage to handle this case by modifying our change of variables (see \cref{sec:cov-local}). 
\item In \cref{th:local-main} and \cref{th:global-main}, we do not have any ``loss of regularity'' at initial time. More precisely, we obtain the continuity of the solution with respect to time in the same space where the initial data belong. Due to the coupling between the fluid system and the structure equations, some results in the literature 
are stated with this loss of regularity: for instance in \cite[Theorem 1.7]{SM18}, there is a loss of order $1/2$ in the space regularity for the fluid velocity at initial time. 
\item As explained above since we work in the ``$L^{p}-{L^{q}}$'' framework, we need less regularity on the initial conditions that in the Hilbert case done by \cite{SM18}. 
More precisely, in  \cite{SM18} the author assumes that the initial conditions satisfy 
\begin{gather*}
\eta_{1}^{0} =0, \quad \eta_{2}^{0} \in W^{3,2}(\ms), \quad  
\widetilde{\rho}^0 \in W^{2,2}({\mf({\eta_{1}^{0}})}), \quad
\widetilde v^{0} \in W^{3,2}(\mathcal{F}({\eta_{1}^{0}}))^{3},
\end{gather*} 
with the corresponding compatibility conditions.
\item \cref{th:local-main} and \cref{th:global-main} can be adapted to the $2D/1D$ case, that is where $\mf$ is a regular bounded domain in $\mathbb{R}^{2}$ such that $\partial{\mf}$ contains a flat part $\Gamma_{S} = \ms \times \{0 \},$ where $\ms$ is an open bounded interval of $\mathbb{R}.$ In that case we can  take $p, q \in (2, \infty)$ such that  $\ds \frac{1}{p} + \frac{1}{2q} \neq \frac{1}{2}.$
\item Instead of taking heat conducting fluid, we can also consider barotropic fluid model, i.e., the system \eqref{eq:main01} without the temperature equation and with the pressure law  $ \widetilde{\pi} = \widetilde{\rho}^{\gamma},$ for some constant $\gamma > 1.$ In that case, we can take $1 < p < \infty$ and $n < q < \infty$ ($n=2$ or $3,$ the dimension of the fluid domain)
such that $\ds \frac{1}{p} + \frac{1}{2q} \neq 1.$ 
\end{enumerate}
\end{Remark}

The proofs of \cref{th:local-main} and \cref{th:global-main} follow a standard approach in the literature on well-posedness for fluid-solid interaction systems. 
One of the main difficulties in studying fluid-structure models is that the fluid system is written in the deformed configuration (in Eulerian variables) whereas the structure equations are written in the reference configuration (in Lagrangian variables). 
Since the fluid domain $\mf(\eta(t))$ depends on the structure displacement, which is one of unknowns, we first reformulate the problem in a fixed domain. 
This is achieved thanks to a combination of a \textit{geometric} change of variables (defined through the initial displacement of the structure) and a 
\textit{Lagrangian} change of coordinates. With this combined change of variables, we reformulate the problem in the reference domain $\mf.$ 
In most of the existing literature, a geometric change of variables via the displacement of the fluid-structure interface is used to rewrite the problem in a fixed domain (\cite{JL, GHL, BT19, MT20}). However, in the context of compressible fluid-structure systems, it is more convenient to use a Lagrangian (see for instance \cite{HMTT}) or a combination of geometric and Lagrangian change of coordinates (\cite{HiebMur}). In fact, such transformations allow us to use basic contraction mapping theorem. More precisely, this transformation eliminates the difficult term $\widetilde v \cdot \nabla \widetilde\rho$ from the density equation. 

Next, we associate the original nonlinear problem to a linear one involving the non-homogeneous terms. In the case of the local in time existence, 
this linear system can be partially decoupled (see system \eqref{NL1.0}-\eqref{NL1.3}). The $L^{p}-L^{q}$ regularity of such linear system over  finite time interval  is obtained by combining various existing maximal $L^{p}-L^{q}$ results for parabolic systems. One of the difficulties is that 
due to the non-zero initial displacement of the beam, we are dealing with linear operators involving variable coefficients.  
For the global existence part, we use a ``monolithic'' type approach, which means that the linearized system in consideration is still a coupled system of fluid and structure equations (see system \eqref{eq:linear-FSa}-\eqref{eq:linear-FSc}). A crucial step is to show the maximal $L^{p}-L^{q}$ property of the associated 
\textit{fluid-structure} linear operator in the infinite time horizon.  This is achieved by showing that this operator is $\mr$-sectorial and generates an exponentially stable semigroup in a suitable function space. Finally, for both the existence for small time and the existence for small initial conditions, 
we end the proof by using the Banach fixed point theorem. 

The plan of the paper is as follows. 
In \cref{sec_back}, we recall some results concerning $\mr$-sectorial operators that are used both for the proofs of \cref{th:local-main} and \cref{th:global-main}.
Then, we prove \cref{th:local-main} in \cref{sec:local}. In \cref{sec:cov-local}, we introduce the combination of Lagrangian and geometric change of coordinates to reformulate the original problem in the reference configuration. Local in time existence for the system written in reference configuration is stated in \cref{th:local-2}. 
In \cref{sec:local-max-lp}, we prove the maximal $L^{p}-L^{q}$ regularity of a linearized system, whereas in \cref{sec:23}, we derive estimates for the nonlinear terms 
in order to prove \cref{th:local-2} by using the Banach fixed point theorem. 
\cref{sec_glo} is devoted to the proof of \cref{th:global-main}.  In \cref{sec_chgvar2} we apply the same change of variables 
than in \cref{sec:cov-local} with some slight modifications and then linearize the system around a constant steady state.
The global in time existence for small initial data for the system written in the reference configuration is stated in \cref{th:global-new}.
In \cref{sec:op-fsi}, we introduce the so-called \textit{fluid-structure} operator and we show that it is an $\mr$-sectorial operator and in 
\cref{sec:exp-fsi} that is generates an exponentially stable semigroup in a suitable function space. 
The maximal $L^{p}-L^{q}$ regularity of the linearized system is proved in \cref{sec:max-lin-g}. Finally, in \cref{sec:36} we show \cref{th:global-new}.
by using the Banach fixed point theorem.

\section{Some Background on $\mathcal{R}$-sectorial Operators}\label{sec_back}
We recall here some definitions and properties related to $\mathcal{R}$-sectorial operators. 
First, let us give the definition of 
$\mathcal{R}$-boundedness ($\mathcal{R}$ for Randomized) for a family of operators (see, for instance, \cite{Weis01,DenkHieberPruss,KunstmannWeis:Levico}): 
\begin{Definition} 
Assume $\mathcal{X}$ and $\mathcal{Y}$ are Banach spaces and $\mathcal{E} \subset \mathcal{L}(\mathcal{X},\mathcal{Y})$.
We say that $\mathcal{E}$ is $\mr-$bounded if there exist
  $p\in [1,\infty)$ and a constant $C>0$, such that for any integer $N \ge 1$,
  any $T_1, \ldots T_N \in  \mathcal{E}$,
  any independent Rademacher random variables $r_1, \ldots, r_N$,
  and any $x_1, \ldots, x_N \in \mathcal{X}$,
\[
   \left(\mathbb{E} \left\| \sum_{j=1}^{N} r_{j} T_{j} x_{j}\right\|_\mathcal{Y}^p\right)^{1/p} \leq C
   \left(\mathbb{E} \left\| \sum_{j=1}^{N} r_{j} x_{j}\right\|_\mathcal{X}^p\right)^{1/p}.
\]
The $\mr_p$-bound of $\mathcal{E}$ on $\mathcal{L}(\mathcal{X},\mathcal{Y})$, denoted by $\mr_p(\mathcal{E})$,
is the smallest constant $C$ in the above inequality.
\end{Definition}

Let us recall that a Rademacher random variable is a symmetric random variables with value in $\{-1,1\}$
and that $\mathbb{E}$ denotes the expectation of a random variable.
Note that the above definition is independent of $p\in [1,\infty)$ (see, for instance, \cite[p.26]{DenkHieberPruss}).
The $\mr_p$-bound has the following properties (see, for instance, Proposition 3.4 in \cite{DenkHieberPruss}):
\begin{equation}  \label{eq:Rbdd-1}
    \mr_p(\mathcal{E}_1 + \mathcal{E}_2) \leqslant \mr_p(\mathcal{E}_1 ) +  \mr_p( \mathcal{E}_2),
  \quad
  \mr_p(\mathcal{E}_1 \mathcal{E}_2 ) \leqslant \mr_p(\mathcal{E}_1 )   \mr_p( \mathcal{E}_2).
\end{equation}
For any $\beta \in (0,\pi)$, we consider the sector $\mr$-sectorial operators:
\begin{equation}\label{sigma}
 \Sigma_{\beta} = \{ \lambda \in \mathbb{C} \setminus \{0\} \ ; \    |\arg(\lambda)| < \beta \}.
\end{equation}
We can introduce the definition of :
\begin{Definition}[sectorial and $\mr$-sectorial operators]\label{defmr}
Let $A : \mathcal{D}(A) \to \mathcal{X}$ be a densely defined closed linear operator on the Banach space $\mathcal{X}$.
The operator $A$ is ($\mathcal{R}$)-sectorial of angle $\beta \in (0, \pi)$
if 
$$
\Sigma_{\beta} \subset \rho(A)
$$
and if the set
\[
  R_\beta =  \left\{ \lambda(\lambda - A)^{-1} \ ; \  \lambda   \in \Sigma_{\beta} \right\}
\]
is ($\mathcal{R}$)-bounded  in $\mathcal{L}(\mathcal{X})$. 
\end{Definition}
We denote by $M_\beta(A)$ (respectively $\mr_{\beta}(A)$) the bound (respectively the $\mr$-bound) of $R_\beta$.
One can replace in the above definitions $R_\beta$ by the set
$$
\widetilde{R_\beta}= \left\{ A(\lambda - A)^{-1} \ ; \  \lambda   \in \Sigma_{\beta} \right\}.
$$
In that case, we denote the uniform bound and the $\mr$-bound by
$\widetilde{M}_\beta(A)$ and $\widetilde{\mr_{\beta}}(A)$.

The following result, due to \cite{Weis01} (see also \cite[p.45]{DenkHieberPruss}), shows the important relation between the notion of $\mr$-sectoriality 
and the maximal regularity of type $L^p$:
\begin{Theorem}\label{thm:weis-Lp-maxreg-char}
Assume $\mathcal{X}$ is a UMD Banach space and that $A : \mathcal{D}(A) \to \mathcal{X}$ is a densely defined, closed linear operator on $\mathcal{X}$.
Then the following assertions are equivalent:
\begin{enumerate}
\item For any $T\in (0,\infty]$ and for any $f\in L^p(0,T;\mathcal{X})$, the Cauchy problem
\begin{equation} \label{eq:max0}
u' = A u + f \quad \text{in} \quad (0,T), \quad u(0) = 0
\end{equation}
admits a unique solution $u$ with $u', Au\in L^{p}(0,T;\mathcal{X})$ and there exists a constant $C>0$ such that
$$
\|u'\|_{L^{p}(0,T;\mathcal{X})}+\|Au\|_{L^{p}(0,T;\mathcal{X})}\leq C\|f\|_{L^{p}(0,T;\mathcal{X})}.
$$
\item $A$ is $\mr$-sectorial of angle $> \frac{\pi}{2}$.
\end{enumerate}
\end{Theorem}
In the above definition, we recall that $\mathcal{X}$ is a UMD Banach space if the Hilbert transform is bounded in $L^p(\mathbb{R};\mathcal{X})$ for $p\in (1,\infty)$. 
In particular, the closed subspaces of $L^q(\Omega)$ for $q\in (1,\infty)$ are UMD Banach spaces.  We refer the reader to \cite[pp.141--147]{Amann} for more information on UMD spaces.

Combining the above theorem with \cite[Theorem 2.4]{Dor91} and  \cite[Theorem 1.8.2]{Tri95}, we can consider the following Cauchy problem
\begin{equation} \label{eq:max-reg-g}
u' = A u + f \quad \text{in}\quad  (0,\infty), \quad u(0) = u_{0}.
\end{equation}
\begin{Corollary} \label{thm:max-reg-g}
Assume $\mathcal{X}$ is a UMD Banach space, $1 < p < \infty$ and $A$ is a closed, densely defined operator in $\mathcal{X}$ with domain $\mathcal{D}(A).$
Let us suppose also that $A$ is a $\mr$-sectorial operator of angle $ > \frac{\pi}{2}$ and that the semigroup generated by $A$ has negative exponential type.
Then for any
 $u_{0} \in (\mathcal{X}, \mathcal{D}(A))_{1-1/p,p}$ and for any $f \in L^{p}(0,\infty;\mathcal{X}),$
the system  \eqref{eq:max-reg-g} admits a unique solution in $L^{p}(0,\infty;\mathcal{D}(A)) \cap W^{1,p}(0,\infty;\mathcal{X}).$
 \end{Corollary}
Finally, we will need the following result (\cite[Corollary~2]{KunstmannWeis:Pisa2001}) on the perturbation theory of $\mr$-sectoriality.
\begin{Proposition} \label{pr:perturb}
Suppose $A$ is a $\mr$-sectorial operator of angle $\beta$ on a Banach space $\mathcal{X}$.
Assume that $B : \mathcal{D}(B) \to \mathcal{X}$ is a linear  operator such that $\mathcal{D}(A) \subset \mathcal{D}(B)$ and
such that there exist $a,b\geq 0$ satisfying
\begin{equation} \label{c:small}
    \| B x  \|_{\mathcal{X}} \leq a  \|  Ax \|_{\mathcal{X}} + b \| x \|_{\mathcal{X}} \quad (x\in \mathcal{D}(A)).
\end{equation}
If 
$$
a < \frac{1}{\widetilde M_\beta(A) \widetilde{\mr_{\beta}}(A)} \quad \text{and} \quad
\lambda >   \frac{b M_\beta(A)  \widetilde{\mr_{\beta}}(A)} {1-a \widetilde{M}_\beta(A) \widetilde{\mr_{\beta}}(A)},
$$ 
then $A+B -\lambda$ is $\mr$-sectorial of angle $\beta$.
\end{Proposition} 

\section{Local in time existence} \label{sec:local}
The aim of this section is to prove  \cref{th:local-main}. 
\subsection{Change of variables and Linearization}  \label{sec:cov-local}
In this subsection, we consider a change of variables to transform the moving domain $\mathcal{F}(\eta(t))$ into the fixed domain 
$\mathcal{F}$. For this we use the Lagrangian change of variables to write everything in $\mathcal{F}(\eta_1^0)$ and a geometric change of variables to
transform $\mathcal{F}(\eta_1^0)$ into $\mathcal{F}$. Let us start with the second one.

First using that $\mf$ is smooth, there exist an open bounded neighborhood $\widetilde{\mathcal{S}}$ of $\overline{\mathcal{S}}$ in $\mathbb{R}^2$,
$\widetilde\varepsilon>0$ and $\widetilde{\eta}: \widetilde{\mathcal{S}} \to \mathbb{R}$ smooth such that
\begin{equation*}  %
\left(\widetilde{\mathcal{S}} \times [-\widetilde\varepsilon,\widetilde\varepsilon]\right)
\cap
\partial \mathcal{F}
=\left\{(s,\widetilde{\eta}(s)), \ s\in \widetilde{\mathcal{S}}
\right\}.
\end{equation*}
We have in particular that $\widetilde{\eta}\equiv 0$ in $\mathcal{S}$. 
From \eqref{geo00-bis}, \eqref{geo01-bis}, we can extend $\eta_-$ and $\eta_+$ with
\begin{equation*} %
\left\{[y_1,y_2,y_3]^\top\in \widetilde{\mathcal{S}}\times \mathbb{R} \ ; \ y_3\in (\eta_-(y_1,y_2),\widetilde{\eta}(y_1,y_2))\right\} \subset \mathcal{F},
\end{equation*}
\begin{equation*} %
\left\{[y_1,y_2,y_3]^\top\in \widetilde{\mathcal{S}}\times \mathbb{R} \ ; \ y_3\in (\widetilde{\eta}(y_1,y_2),\eta_+(y_1,y_2))\right\} \subset \mathbb{R}^3\setminus \overline{\mathcal{F}}.
\end{equation*}
Using \eqref{cc2}--\eqref{cc3} and that $q>3$, we can extend $\eta_{1}^{0}$ by $0$ in $\mathbb{R}^2\setminus \overline{\ms}$ with
$\eta_{1}^{0} \in W^{2,q}(\mathbb{R}^2).$ 
Then \eqref{geo02} yields the existence of $\varepsilon\in (0,1)$ such that 
\begin{equation*}%
\eta_- (1-\varepsilon) < \eta_1^0 <\eta_+ (1-\varepsilon) \quad \text{in} \ \widetilde{\mathcal{S}}.
\end{equation*}
We consider $\chi \in C_{c}^{\infty}(\mathbb{R}^3)$ such that
$$
\supp \chi \subset \left\{[y_1,y_2,y_3]^\top\in \widetilde{\mathcal{S}} \times \mathbb{R} \ ; \ y_3\in (\eta_-(y_1,y_2),\eta_+(y_1,y_2))\right\},
$$
$$
\chi \equiv 1 \quad \text{in} \ \left\{[y_1,y_2,y_3]^\top\in \mathcal{S}\times \mathbb{R} \ ; \ y_3\in ((1-\varepsilon)\eta_-(y_1,y_2),(1-\varepsilon)\eta_+(y_1,y_2))\right\}.
$$
We also define 
$$
\Lambda(y_1,y_2,y_3)=\eta_1^0(y_1,y_2) \chi(y_1,y_2,y_3)e_3 \quad [y_1,y_2,y_3]^\top\in \mathbb{R}^3
$$
and we consider 
\begin{equation}\label{geo5}
\left\{
\begin{array}{l}
\zeta'(t,y)=\Lambda(\zeta(t,y)), \\
\zeta(0,y)=y\in \mathbb{R}^3.
\end{array}
\right.
\end{equation}
Then 
\begin{equation}\label{geo6}
X^0:=\zeta(1,\cdot)
\end{equation}
is a $C^1$-diffeomorphism such that 
$$
X^0 \equiv \Id \quad \text{in} \ \mathbb{R}^3\setminus \left\{[y_1,y_2,y_3]^\top\in {\mathcal{S}}\times \mathbb{R} \ ; \ y_3\in (\eta_-(y_1,y_2),\eta_+(y_1,y_2))\right\}
$$
$$
X^0(\Gamma_S(0))=\Gamma_S(\eta_1^0),
$$
\begin{multline*}
X^0\left(\left\{[y_1,y_2,y_3]^\top\in \mathcal{S}\times \mathbb{R} \ ; \ y_3\in (\eta_-(y_1,y_2),0)\right\}\right)
\\
=\left\{[y_1,y_2,y_3]^\top\in \mathcal{S}\times \mathbb{R} \ ; \ y_3\in (\eta_-(y_1,y_2),\eta_1^0(y_1,y_2))\right\}.
\end{multline*}
In particular, $X^0$ is a $C^1$-diffeomorphism such that $X^0(\mathcal{F})=\mathcal{F}(\eta_1^0)$ and such that $X^0=\Id$ on $\Gamma_0$.

We consider the characteristics $X$ associated with the fluid velocity $\widetilde v$:
\begin{equation} \label{X-def}
\begin{dcases}
\partial_{t} X(t,y) = \widetilde v(t, X(t,y)) \qquad (t  > 0), \\
X(0,y) = X^0(y), \quad y   \in \overline{\mathcal{F}}.
\end{dcases}
\end{equation}
Assume that $X$ is a $C^{1}$-diffeomorphism from $\overline{\mathcal{F}}$ onto $\overline{\mf(\eta(t))}$ for all $t \in (0,T).$ 
For each $t \in (0,T),$   we denote by $Y(t,\cdot) = [X(t,\cdot)]^{-1}$ the inverse of $X(t,\cdot)$.  We consider the following change of variables
\begin{equation}
\begin{array}{c}
  \rho(t,y)  = \widetilde \rho(t,X(t,y)) , \qquad  v (t,y)   = \widetilde  v(t,X(t,y)), \\
  \vartheta(t,y)  = \widetilde  \vartheta(t,X(t,y)), \qquad    \pi  = R_0  \rho   \vartheta +\pi_0, \label{lpq01}
\end{array}
\end{equation}
for $(t,y) \in (0,T) \times \mf.$  In particular,
\begin{equation*}
\widetilde \rho(t,x)   =   \rho(t,Y(t,x)), \quad  \widetilde  v(t,x)  = v(t,Y(t,x)),\quad
\widetilde  \vartheta(t,x)   =   \vartheta(t,Y(t,x)),
\end{equation*}
for $(t,x) \in (0,T) \times \mf(\eta(t))$. We introduce the notation
\begin{equation} \label{mat-Azv}
\mathbb{B}_X:=\Cof \nabla X, \quad \delta_X:=\det \nabla X, \quad
\mathbb{A}_{X} :=   \frac{1}{\delta_X} \mathbb{B}_X^{\top} \mathbb{B}_X,
\end{equation}
\begin{equation} \label{mat-Azv0}
\mathbb{B}^0:=\mathbb{B}_{X^0}, \quad 
\delta^0:=\delta_{X^0}, \quad
\mathbb{A}^0 :=   \mathbb{A}_{X^0}.
\end{equation}

This change of variables transforms \eqref{eq:main01}--\eqref{eq:force-fsi} into the following system for
$\vect{\rho, v, \vartheta, \eta}$:
\begin{equation} \label{NL0.0} 
\left\{
\begin{array}{ll}
     \displaystyle \partial_{t}   \rho + \frac{\rho^{0}}{\delta^0}  \nabla v :  \mathbb{B}^0  = {F}_{1} & \mbox{ in } (0,T) \times \mf, \\
        \rho(0,\cdot) = \rho^{0} & \mbox{ in } \mf, \\
\end{array}   
\right.
\end{equation}
\begin{equation} \label{NL0.1} 
\left\{
\begin{array}{ll}     
\partial_{t}  v - \mathcal{L}v = {F}_{2}  & \mbox{ in } (0,T) \times \mf,\\
        v = 0 &\mbox{ on } (0,T) \times \Gamma_{0},\\
        v =  \partial_{t} \eta e_{3} & \mbox{ on } (0,T) \times \Gamma_{S}, \\
      v(0,\cdot) = v^{0}  & \mbox{ in } \mf,
\end{array}   
\right.
\end{equation}
\begin{equation} \label{NL0.2} 
\left\{
\begin{array}{ll}     
\partial_{t} \vartheta - \dfrac{\kappa}{c_{v} \rho^{0} \delta^0} \div \left( \mathbb{A}^0 \nabla \vartheta \right) 
= {F}_{3} & \mbox{ in } (0,T) \times \mf, \\
     \mathbb{A}^0 \nabla \vartheta \cdot n  = {G} &  \mbox{ on } (0,T) \times \partial \mf,  \\
        \vartheta(0,\cdot) = \vartheta^{0}  & \mbox{ in } \mf,
\end{array}   
\right.
\end{equation}
\begin{equation} \label{NL0.3} 
\left\{
\begin{array}{ll}     
\partial_{tt} \eta  + \Delta_{s}^{2} \eta - \Delta_{s} \partial_{t}\eta =  {H}   & \mbox{ in } (0,T) \times \ms, \\
\eta  = \nabla \eta \cdot n_{S}= 0   & \mbox{ on } (0,T) \times \partial \ms, \\
\eta(0,\cdot) = \eta_{1}^{0}, \quad \partial_{t} \eta(0,\cdot) = \eta_{2}^{0}  & \mbox{ in } \ms, 
\end{array}   
\right.
\end{equation}
where we have used the following notation
\begin{equation}
\rho^{0} : = \widetilde \rho^{0} \circ X^0,
\quad 
v^{0} := \widetilde v^{0}\circ X^0,
\quad
 \vartheta^{0} := \widetilde \vartheta^{0} \circ X^0,
\end{equation}
\begin{equation}\label{defcalL}
\mathcal{L}v =   \frac{1}{\rho^{0} \delta^0} \div \mathbb{T}^0(v),
\quad \mathbb{T}^0(v):= \mu \nabla v\mathbb{A}^0 + \frac{\mu+\alpha}{\delta^0} \mathbb{B}^0 (\nabla v)^\top \mathbb{B}^0
\end{equation}
\begin{equation}\label{defF1}
F_1( \rho,  v, \vartheta, \eta):=\frac{\rho^{0}}{\delta^0}  \nabla v :  \mathbb{B}^0-\frac{\rho}{\delta_X}  \nabla v :  \mathbb{B}_X
\end{equation}
\begin{multline}\label{defF2}
F_2( \rho,  v, \vartheta, \eta):=\frac{1}{\rho^{0}\delta^0}
\Bigg[
\left( \rho^0 \delta^0-\rho \delta_X \right)\partial_t v
+\mu \div \left( \nabla v\left(\mathbb{A}_X-\mathbb{A}^0\right) \right)
\\
+(\mu+\alpha) \div \left[ \frac{1}{\delta_X} \mathbb{B}_X (\nabla v)^\top \mathbb{B}_X-\frac{1}{\delta^0} \mathbb{B}^0 (\nabla v)^\top \mathbb{B}^0 \right]
+R_0\mathbb{B}_X \nabla (\rho\vartheta)
\Bigg]
\end{multline}
\begin{multline}\label{defF3}
F_3( \rho,  v, \vartheta, \eta):=\frac{1}{c_{v} \rho^{0}\delta^0}
\Bigg[
c_{v}\left( \rho^0 \delta^0 -\rho \delta_X\right)\partial_t \vartheta
+\kappa \div \left( \left(\mathbb{A}_X-\mathbb{A}^0\right) \nabla \vartheta\right)
\\
+\frac{\alpha}{\delta_X} \left(\mathbb{B}_X :\nabla v\right)^2
+\frac{\mu}{2\delta_X} \left|\nabla v \mathbb{B}_X^\top+\mathbb{B}_X \nabla v ^\top \right|^2
-(R_0\rho \vartheta+\pi_0) \nabla v:\mathbb{B}_X
\Bigg]
\end{multline}
\begin{equation}\label{G}
{G}( \rho,  v, \vartheta, \eta)  = \left(\mathbb{A}^0-\mathbb{A}_X\right) \nabla \vartheta\cdot n
\end{equation}
\begin{equation}\label{H}
{H}( \rho,  v, \vartheta, \eta)  = 
-\frac{\mu}{\delta_X} \left(\nabla v \mathbb{B}_X^\top+\mathbb{B}_X \nabla v ^\top\right) \begin{bmatrix}
-\nabla_s \eta \\ 1
\end{bmatrix}\cdot e_3
-\frac{\alpha}{\delta_X} \nabla v:\mathbb{B}_X +R_0\rho\vartheta+\pi_0.
\end{equation}

The characteristics $X$ defined in \eqref{X-def} can now be written as 
\begin{align} \label{Jacobi}
X(t,y)  = X^0(y) + \int_{0}^{t}  v(r,y) \ {\rm d}r, \quad 
\end{align}
for every $y\in \mf$ and $t\geqslant 0$. 

The hypotheses \eqref{cc1}--\eqref{cc4} on the initial conditions are transformed into the following conditions 
\begin{gather}
\rho^{0} \in W^{1,q}(\mf), \quad \min_{\overline{\mf}} \rho^{0}  > 0, \label{rho0pos} \\
\eta_{1}^{0} \in B^{2(2-1/p)}_{q,p}(\ms), \quad \eta_{1}^{0} = \ds \nabla_{s} \eta_{1}^{0} \cdot n_{S} = 0 \mbox{ on } \ms,  \label{eta10pos}\\
v^{0} \in B^{2(1-1/p)}_{q,p}(\mf)^{3}, \quad
	\vartheta^{0} \in B^{2(1-1/p)}_{q,p}(\mf), \quad 
	\quad \eta_{2}^{0} \in B^{2(1-1/p)}_{q,p}(\ms), \label{cc5} \\
 v^{0} = 0 \mbox{ on } \Gamma_{0}, \quad v^{0} = \eta_{2}^{0} e_{3} \mbox{ on }  \Gamma_{S}, 
 	\quad   \eta_{2}^{0} = 0 \mbox { on } \partial \ms,  \label{ci02} \\
\ds \nabla \eta_{2}^{0} \cdot n_{S} = 0 \mbox{ on } \partial \ms \quad \text{and} \quad \mathbb{A}^0 \nabla \vartheta^{0} \cdot n   = 0 \quad \mbox{on} \ \partial\mf  \quad \mbox{if} \quad  \displaystyle \frac{1}{p} + \frac{1}{2q} < \frac{1}{2}. \label{ci03}
\end{gather}
Here $n$ is the unit normal to $\partial \mf$ outward to $\mf.$ 
The regularity properties in \eqref{rho0pos} and \eqref{cc5} can be obtained from \eqref{cc1}, \eqref{cc2} by applying \cite[Lemma 2.1]{MT20}.
Using the above change of variables, our main result in \cref{th:local-main} can be rephrased as 
\begin{Theorem}  \label{th:local-2}
Assume $(p,q)$ satisfies \eqref{pqpqpq} and that
$[\rho^{0}, v^{0}, \vartheta^{0}, \eta_{1}^{0}, \eta_{2}^{0}]^{\top}$ 
satisfies \eqref{rho0pos}--\eqref{ci03} and \eqref{geo02}.
Then there exists $T > 0$ such that the system \eqref{NL0.0}--\eqref{Jacobi}  
admits a unique strong solution 
$$
\vect{\rho, v, \vartheta, \eta} \in W^{1,p}(0,T;W^{1,q}(\mf))\times \left( W^{1,2}_{p,q} ((0,T) ; \mf\right)^{3}\times W^{1,2}_{p,q} ((0,T) ; \mf)
\times W^{2,4}_{p,q}((0,T) ; \ms )
$$
Moreover, 
$$
\min_{[0,T]\times \overline{\mf}} \rho  > 0, \quad \Gamma_{0} \cap \Gamma_S(\eta(t)) = \emptyset \quad (t\in [0,T]),
$$
and for all $t \in [0,T]$, $X(t,\cdot): \mf \to \mf(\eta(t))$ is a $C^{1}$-diffeomorphism.
\end{Theorem}

\subsection{Maximal $L^{p}$-$L^{q}$ regularity of a linear system.}  \label{sec:local-max-lp}
The proof of \cref{th:local-2} relies on the Banach fixed point theorem and on maximal $L^{p}$-$L^{q}$ estimates of a linearized system. By replacing the nonlinear terms ${F}_{1}, {F}_{2}, {F}_{3}, {G}$ and ${H}$ in \eqref{NL0.0}--\eqref{NL0.3} by given source terms $f_{1}, f_{2}, f_{3}, g$ and $h$ we obtain the following linear system 
\begin{equation} \label{NL1.0} 
\left\{
\begin{array}{ll}
     \displaystyle \partial_{t}   \rho + \frac{\rho^{0}}{\delta^0}  \nabla v :  \mathbb{B}^0  = {f}_{1} & \mbox{ in } (0,T) \times \mf, \\
        \rho(0,\cdot) = \widehat\rho^{0} & \mbox{ in } \mf, \\
\end{array}   
\right.
\end{equation}
\begin{equation} \label{NL1.1} 
\left\{
\begin{array}{ll}     
\partial_{t}  v - \mathcal{L}v = {f}_{2}  & \mbox{ in } (0,T) \times \mf,\\
        v = 0 &\mbox{ on } (0,T) \times \Gamma_{F},\\
        v =  \partial_{t} \eta e_{3} & \mbox{ on } (0,T) \times \Gamma_{S}, \\
      v(0,\cdot) = v^{0}  & \mbox{ in } \mf,
\end{array}   
\right.
\end{equation}
\begin{equation} \label{NL1.2} 
\left\{
\begin{array}{ll}     
\partial_{t} \vartheta - \dfrac{\kappa}{c_{v} \rho^{0} \delta^0} \div \left( \mathbb{A}^0 \nabla \vartheta \right) 
= {f}_{3} & \mbox{ in } (0,T) \times \mf, \\
     \mathbb{A}^0 \nabla \vartheta \cdot n  = {g} &  \mbox{ on } (0,T) \times \partial \mf,  \\
        \vartheta(0,\cdot) = \vartheta^{0}  & \mbox{ in } \mf,
\end{array}   
\right.
\end{equation}
\begin{equation} \label{NL1.3} 
\left\{
\begin{array}{ll}     
\partial_{tt} \eta  + \Delta_{s}^{2} \eta - \Delta_{s} \partial_{t}\eta =  {h}   & \mbox{ in } (0,T) \times \ms, \\
\eta  = \nabla \eta \cdot n_{S}= 0   & \mbox{ on } (0,T) \times \partial \ms, \\
\eta(0,\cdot) = \widehat\eta_{1}^{0}, \quad \partial_{t} \eta(0,\cdot) = \eta_{2}^{0}  & \mbox{ in } \ms, 
\end{array}   
\right.
\end{equation}
where $\mathbb{A}^0, \mathbb{B}^0, \delta^0$  are defined in \eqref{mat-Azv0} and where $\mathcal{L}$ is defined by \eqref{defcalL}. 
Note that we also modify the initial conditions in the above system with respect to \eqref{NL0.0}--\eqref{NL0.3}  
since $\rho^0$ and $\eta_1^0$ already appear in the coefficients of \eqref{NL1.0}--\eqref{NL1.3}.
In the next section, we will take
$$
\widehat\rho^{0}=\rho^{0}, \quad \widehat\eta_{1}^{0}=\eta_{1}^{0}
$$
but here we do not assume the above relation. In particular, we assume that $\rho^0$ satisfies the second condition of \eqref{rho0pos}
and that $\eta_1^0$ satisfies \eqref{geo02} but we do not impose these hypotheses on $\widehat\rho^{0}$ and on $\widehat\eta_{1}^{0}$.

We recall that $(p,q)$ satisfies \eqref{pqpqpq} and to simplify, we assume
throughout this section that  
\begin{equation*} %
T\in (0,1].
\end{equation*} 
This condition is only used to avoid the dependence in time of the constants in the estimates of this section.

We consider the subset of initial conditions 
\begin{multline}
\mathcal{I}_{p,q} = 
\Bigg\{ \vect{\widehat\rho^{0},v^{0},\vartheta^{0},\widehat\eta_{1}^{0},\eta_{2}^{0} } 
		\in W^{1,q}(\mf)\times B^{2(1-1/p)}_{q,p}(\mf)^{3}\times B^{2(1-1/p)}_{q,p}(\mf)\times B^{2(2-1/p)}_{q,p}(\ms)\times B^{2(1-1/p)}_{q,p}(\ms), \\
 v^{0} = 0 \mbox{ on } \Gamma_{0}, \quad v^{0} = \eta_{2}^{0} e_{3} \mbox{ on }  \Gamma_{S}, 
 \quad \widehat \eta_{1}^{0} = \ds \frac{\partial \widehat \eta_{1}^{0}}{\partial n_S} = \eta_{2}^{0} = 0 \mbox { on } \partial \ms,  \\
\ds \frac{\partial \eta_{2}^{0}}{\partial n_S} = 0 \mbox { on } \partial \ms
\quad \text{and} \quad
\mathbb{A}^0 \nabla \vartheta^{0} \cdot n   = 0 \mbox{ on } \partial\mf   \quad \mbox{if} \quad  \displaystyle \frac{1}{p} + \frac{1}{2q} < \frac{1}{2}
\Bigg\},
\end{multline}
endowed with the norm
\begin{multline*}
\left\| \vect{\widehat\rho^{0},v^{0},\vartheta^{0},\widehat\eta_{1}^{0},\eta_{2}^{0} }  \right\|_{\mathcal{I}_{p,q}} 
	:=  \|\widehat\rho^{0} \|_{W^{1,q}(\mf)} + \| v^{0} \|_{B^{2(1-1/p)}_{q,p}(\mf)^{3}} + \|\vartheta^{0} \|_{B^{2(1-1/p)}_{q,p}(\mf)} \\ 
+ \|\widehat\eta_{1}^{0}\|_{ B^{2(2-1/p)}_{q,p}(\ms)} + \|\eta_{2}^{0} \|_{ B^{2(1-1/p)}_{q,p}(\ms)}.
\end{multline*}
We also consider the space  ${\mathcal R}_{T,p,q}$ of the source terms in \eqref{NL1.0}--\eqref{NL1.3}:
\begin{multline}\label{com0.5}
\mathcal{R}_{T,p,q} = \Big\{ \vect{f_{1}, f_{2}, f_{3}, g, h} \mid f_{1} \in  L^{p}(0,T,W^{1,q}(\mf)), f_{2} \in  L^{p}(0,T;L^{q}(\mf))^{3}, \\ f_{3} \in  L^{p}(0,T;L^{q}(\mf)),
g \in F^{(1-1/q)/2}_{p,q}(0,T;L^{q}(\partial \mf)) \cap L^{p}(0,T;W^{1-1/q,q}(\partial \mf)), \\
h \in L^{p}(0,T;L^{q}(\ms)),   \mbox{  with } g(0,\cdot) = 0  \mbox{ if } \displaystyle \frac{1}{p} + \frac{1}{2q} < \frac{1}{2} \Big\},
\end{multline}
with
\begin{multline*}
\|\vect{f_{1},f_{2}, f_{3}, g, h}\|_{\mathcal{R}_{T,p,q}} 
=  
\|f_{1}\|_{L^{p}(0,T;W^{1,q}(\mf))} + \|f_{2}\|_{L^{p}(0,T;L^{q}(\mf))^{3}}   
+ \|f_{3}\|_{L^{p}(0,T;L^{q}(\mf))} \\ + \|g\|_{F^{(1-1/q)/2}_{p,q}(0,T;L^{q}(\partial \mf)) \cap L^{p}(0,T;W^{1-1/q,q}(\partial \mf))} 
+ \|h\|_{L^{p}(0,T;L^{q}(\ms))}. 
\end{multline*}
Finally, the space  ${\mathcal W}_{T,p,q}$ of the solutions 
$[\rho,u,\vartheta, \eta]^\top$ of \eqref{NL1.0}--\eqref{NL1.3} is the Cartesian product:
\begin{equation} \label{solspace}
\mathcal{W}_{T,p,q} = W^{1,p}(0,T;W^{1,q}(\mf))\times  W^{1,2}_{p,q} ((0,T) ; \mf)^{3}\times W^{1,2}_{p,q} ((0,T) ; \mf)
\times W^{2,4}_{p,q}((0,T) ; \ms ),
\end{equation}
with the norm
$$
\|\vect{\rho,u,\vartheta,\eta}\|_{\mathcal{W}_{T,p,q}}  := \|\rho\|_{W^{1,p}(0,T;W^{1,q}(\mf))} +  \|u\|_{W^{1,2}_{p,q} ((0,T) ; \mf)^{3}}    
	+ \|\vartheta\|_{W^{1,2}_{p,q}((0,T) ; \mf)}   + \|\eta\|_{W^{2,4}_{p,q}((0,T) ; \ms)}.
$$

With the above notation, we can state the main result of this section:
\begin{Theorem} \label{thm-linear}
Assume \eqref{pqpqpq} \eqref{rho0pos},  \eqref{eta10pos} and \eqref{geo02}.
Then for any 
\begin{equation}\label{lininisour}
\vect{\widehat \rho^{0}, v^{0}, \vartheta^{0}, \widehat \eta_{1}^{0}, \eta_{2}^{0}} \in \mathcal{I}_{p,q},
\quad
\vect{f_{1},f_{2},f_{3}, g, h} \in \mathcal{R}_{T,p,q},
\end{equation}
the system \eqref{NL1.0}--\eqref{NL1.3} admits a unique solution $[\rho, v,  \vartheta,\eta]^\top \in \mathcal{W}_{T,p,q}$
and there exists a constant $C > 0$ depending  on $p,q$  and independent of $T$ such that
\begin{equation} \label{est:thm1}
\left\| \vect{ \rho,v, \vartheta,\eta} \right\|_{\mathcal{W}_{T,p,q}}
  \leqslant C  \Big(  \norm{ \vect{\widehat \rho^{0}, v^{0}, \vartheta^{0},\widehat \eta_{1}^{0}, \eta_{2}^{0}} }_{\mathcal{I}_{p,q}}  +  \norm{ \vect{f_{1},f_{2}, f_{3}, g, h} }_{\mathcal{R}_{T,p,q}} \Big).
\end{equation}
\end{Theorem}
In order to prove the above result, we notice that the system \eqref{NL1.0}--\eqref{NL1.3}
can be solved in ``cascades''.  Systems \eqref{NL1.2} and \eqref{NL1.3} can be solved independently.  With the solution of system \eqref{NL1.3} we can solve the system \eqref{NL1.1} and then  \eqref{NL1.0}.  

We first need the following result on the coefficients appearing in the system \eqref{NL1.0}--\eqref{NL1.3}:
\begin{Lemma} \label{lem:zeta-reg}
Assume \eqref{pqpqpq} \eqref{rho0pos},  \eqref{eta10pos} and \eqref{geo02}. Then $\mathbb{A}^0, \mathbb{B}^0, \delta^0$ defined in \eqref{mat-Azv0} satisfy
$$
\delta^0>0, \quad \mathbb{A}^0=(\mathbb{A}^0)^\top,
\quad
\frac{1}{\delta^0} \in W^{1,q}({\mf}),  \quad \mathbb{B}^0, \mathbb{A}^0\in W^{1,q}({\mf})^{9},
$$
and there exists $c^0>0$ such that
$$
\mathbb{A}^0 \geq c^0 \mathbb{I}_{3} \quad \text{in} \ \overline{\mf}.
$$
\end{Lemma}
\begin{proof}
The proof relies on the dependence of the solutions of \eqref{geo5} with respect to the initial conditions. 
Using that $\eta_{1}^{0} \in W^{2,q}(\mathbb{R}^{2})$ for $q>3$ and Sobolev embedding, we have that $\Lambda\in C^1_b(\mathbb{R}^3)$.
In particular, from standard results (see, for instance, \cite[p.116]{AmannODE}), we have that $\zeta\in C^1(\mathbb{R}\times \mathbb{R}^3)$
and by using the ordinary differential equation satisfied by the derivatives of $\zeta$ in space, we find that
$X^0 \in W^{2,q}({\mathcal{F}})^{3}$ and $\nabla X^0$ is invertible. This yields the result.
\end{proof}

We are now in a position to prove \cref{thm-linear}:
\begin{proof}[Proof of \cref{thm-linear}]
The proof is divided in several steps devoted to the resolution of each system.

\underline{\bf Step 1:} we show here that 
\eqref{NL1.3} admits a unique solution $\eta \in W^{2,4}_{p,q}((0,T) \times \ms)$ 
and that there exists a constant $C$ independent of $T$ such that 
\begin{equation} \label{est:lin-eta}
\|\eta\|_{ W^{2,4}_{p,q}((0,T) ; \ms)} +  \|\partial_{t}\eta\|_{ W^{1,2}_{p,q}((0,T) ; \ms)} 
\leqslant C \left(  \|\widehat\eta_{1}^{0}\|_{B^{2(2-1/p)}_{q,p}(\ms)} +  \|\eta_{2}^{0}\|_{B^{2(1-1/p)}_{q,p}(\ms)}  +  \|h\|_{L^{p}(0,T;L^{q}(\ms))}\right). 
\end{equation}

To prove this, we combine \cite[Theorem 5.1]{DenkSchnaubelt} and \cite[Theorem 4.2]{Weis01}.  For the sake of clarity, we provide brief details about the proof. 
We first consider 
\begin{equation} \label{eq:xs}
\mathcal{X}_{S} := W^{2,q}_{0}(\ms)  \times L^{q}(\ms),
\end{equation}
and the operator $A_{S}$ defined by
\begin{equation} \label{op:As}
\mathcal{D}(A_{S}) = \left(W^{4,q}(\ms) \cap W^{2,q}_{0}(\ms) \right) \times W^{2,q}_{0}(\ms),
\quad
A_{S} = \begin{bmatrix}
0 & \Id \\ -  \Delta^{2} & \Delta
\end{bmatrix}.
\end{equation}
With the above notation, the system \eqref{NL1.3} can be written as 
\begin{equation*}
\frac{d}{dt} \begin{bmatrix}
\eta \\ \partial_{t} \eta
\end{bmatrix}
= A_{S} \begin{bmatrix}
\eta \\ \partial_{t} \eta
\end{bmatrix} + \begin{bmatrix}
0 \\ h
\end{bmatrix}, \qquad 
\begin{bmatrix}
\eta \\ \partial_{t} \eta
\end{bmatrix}(0)  = 
\begin{bmatrix}
\widehat \eta_{1}^{0} \\ \eta_{2}^{0}
\end{bmatrix}.
\end{equation*}
Applying Theorem 5.1 in \cite{DenkSchnaubelt}, we have that $A_{S}$ is $\mr$-sectorial in $\mathcal{X}_{S}$ of angle $\beta_{0} > \pi/2$ (see \cref{sec_back}).
Thus the operator $A_{S}$ has maximal regularity $L^{p}$-regularity in $\mathcal{X}_{S}$ (\cite[Theorem 4.2]{Weis01}
or \cref{thm:max-reg-g}). More precisely, for every $h \in L^{p}(0,T;L^{q}(\mf))$ and for every 
$(\widehat \eta_{1}^{0}, \eta_{2}^{0}) \in (\mathcal{X}_{S}, \mathcal{D}(A_{S}))_{1-1/p,p},$  the system  \eqref{NL1.3} admits a unique strong solution with 
\begin{equation*}
\eta \in L^{p}(0,T;W^{4,q}(\ms)) \cap W^{2,p}(0,T;L^{q}(\ms)). 
\end{equation*}
In order to obtain the estimate \eqref{est:lin-eta} independent of $T$, we proceed as \cite[Proposition 2.2]{HMTT}.

\underline{\bf Step 2:} we show now that
the system \eqref{NL1.1} admits a unique solution $v \in W^{1,2}_{p,q} ((0,T) ; \mf)^{3}$
and that there exists a constant $C > 0$ depending only on the geometry
such that
\begin{multline} \label{est:lin-v}
\left\| v \right\|_{W^{1,2}_{p,q} ((0,T) ; \mf)^{3}}
\leqslant C \Big(
 \|\widehat \eta_{1}^{0}\|_{B^{2(2-1/p)}_{q,p}(\ms)} +  \|\eta_{2}^{0}\|_{B^{2(1-1/p)}_{q,p}(\ms)}  + \left\| v^{0} \right\|_{B^{2(1-1/p)}_{q,p}(\mf)^{3}} 
\\
+  \|h\|_{L^{p}(0,T;L^{q}(\ms))}
+  \| f_{2} \|_{L^{p}(0,T;L^{q}(\mf))^{3}}
\Big).
\end{multline}
To do this, we are going to apply \cite[Theorem 2.3]{DenkHieberPruss07} and for this, we first reduce the problem to the case of homogeneous boundary conditions. 

Using that $\mathcal{F}$ is a smooth domain, there exists an open bounded neighborhood $\widetilde{\mathcal{S}}$ of $\overline{\mathcal{S}}$ in $\mathbb{R}^2$,
$\widetilde\varepsilon>0$ and $\widetilde{\eta}: \widetilde{\mathcal{S}} \to \mathbb{R}$ smooth such that
\begin{equation}\label{geo10-bis}
\left(\widetilde{\mathcal{S}} \times [-\widetilde\varepsilon,\widetilde\varepsilon]\right)
\cap
\partial \mathcal{F}
=\left\{(s,\widetilde{\eta}(s)), \ s\in \widetilde{\mathcal{S}}
\right\}.
\end{equation}
We consider $\chi \in C_{c}^{\infty}(\mathbb{R}^3)$ such that
$$
\supp \chi \subset \widetilde{\mathcal{S}} \times [-\widetilde\varepsilon,\widetilde\varepsilon],
\quad
\chi \equiv 1 \quad \text{in} \ {\mathcal{S}} \times [-\widetilde\varepsilon/2,\widetilde\varepsilon/2].
$$
Then we define
\begin{equation} \label{chi-cut-off}
w(t,y_1,y_2,y_3) := \chi(y_1,y_2,y_3) \partial_{t} \eta(t,y_1,y_2) e_{3} \quad \left((t, y_1,y_2,y_3) \in (0,T) \times \mathbb{R}^3\right)
\end{equation}  
and we set $u=v-w$ so that $u$ is the solution of  
\begin{equation} \label{sys:u}
   \begin{dcases}
       \partial_{t}  u - \mathcal{L} u
       = \widehat{f}_{2}:= f_{2} - \partial_{t} w - \mathcal{L} w  & \mbox{ in } (0,T) \times \mf,\\
        u = 0 &\mbox{ on } (0,T) \times \partial \mf,\\
      u(0,\cdot) = u^{0} :=v^{0} - w(0,\cdot)   & \mbox{ in } \mf,
     \end{dcases}  
\end{equation}
From \cref{lem:zeta-reg} and \eqref{est:lin-eta}, there exists a positive constant $C$ independent of $T$ such that 
$$
\|\widehat{f}_{2}\|_{L^{p}(0,T;L^{q}(\mf))^{3}} \leqslant C \Big( 
\|\widehat\eta_{1}^{0}\|_{B^{2(2-1/p)}_{q,p}(\ms)} +  \|\eta_{2}^{0}\|_{B^{2(1-1/p)}_{q,p}(\ms)}  +  \|h\|_{L^{p}(0,T;L^{q}(\ms))}
+    \| f_{2} \|_{L^{p}(0,T;L^{q}(\mf))^{3}} \Big),
$$
$$
\|u^{0}\|_{B^{2(1-1/p)}_{q,p}(\mf)^{3}}\leq \|v^{0}\|_{B^{2(1-1/p)}_{q,p}(\mf)^{3}}+\|\eta_{2}^{0}\|_{B^{2(1-1/p)}_{q,p}(\ms)}.
$$
Moreover, $u^{0} = 0$ on $\partial \mf.$  To obtain the result it remains to show that for $u^{0}\in B^{2(1-1/p)}_{q,p}(\mf)^{3}$ with 
$u^{0} = 0$ on $\partial \mf$ and for $\widehat{f}_{2}\in L^{p}(0,T;L^{q}(\mf))^{3},$ system \eqref{sys:u} admits a unique strong solution in 
$W^{1,2}_{p,q}((0,T);\mf)^{3}$ with an estimate independent of $T$. In order to do this, we are going to apply \cite[Theorem 2.3]{DenkHieberPruss07}.

Let us denote by $\mathcal{L}_{0}(y, \xi)$ the principal symbol of the operator $\mathcal{L}$ defined by \eqref{defcalL}.
Then we have
$$
\mathcal{L}_{0}(\cdot, \xi)=\frac{\mu}{\rho^{0} \delta^0} (\mathbb{A}^0\xi\cdot \xi) I_3 + 
\frac{\mu+\alpha}{\rho^{0} (\delta^0)^2} (\mathbb{B}^0\xi) \otimes (\mathbb{B}^0\xi).
$$
In particular, $\mathcal{L}_{0}(\cdot, \xi)$ is symmetric and using \eqref{mat-Azv0} and \eqref{visco-relation},
there exists $c^0$ such that
\begin{equation} \label{eq:strong-ellipticity}
\mathcal{L}_{0}(y, \xi)a\cdot a\geq c^0 |a|^2 \quad (y\in \mathcal{F}, a,\xi\in \mathbb{R}^3, \ |\xi|=1).
\end{equation}
This shows condition $\bf{(E)}$ (ellipticity of the interior symbol) of  \cite{DenkHieberPruss07}.

Since we are in the case of the Dirichlet boundary conditions, \eqref{eq:strong-ellipticity} yields the Lopatinskii–Shapiro condition $\bf{(LS)}$, see for instance,
\cite[Proposition 6.2.13 and Remark (i), p.270]{Prussbook}.

Finally, applying again \cref{lem:zeta-reg} and using that $q>3$, we can verify that $\bf{(SD1)}$ and $\bf{(SB1)}$ hold true. 
We can thus apply \cite[Theorem 2.3]{DenkHieberPruss07} and deduce that the system \eqref{sys:u} admits a unique solution $u \in W^{1,2}_{p,q}((0,T) ; \mf)^{3}.$ 
This yields that the system \eqref{NL1.1}  admits a unique solution $v \in W^{1,2}_{p,q}((0,T) ; \mf)^{3}.$  
In order to show that the estimate \eqref{est:lin-v} holds with a constant independent of $T,$ we can proceed as \cite[Proposition 2.2]{HMTT}. 

\underline{\bf Step 3:} next we prove that the system \eqref{NL1.2} admits a unique strong solution $ \vartheta\in W^{1,2}_{p,q} ((0,T) ; \mf)$ 
and that there exists a constant $C > 0$, depending only on the geometry such that
\begin{multline} \label{est:heat}
\| \vartheta \|_{W^{1,2}_{p,q} ((0,T) ; \mf)}  \leqslant C \Big( \| \vartheta^{0} \|_{B^{2(1-1/p)}_{q,p}(\mf)} + \| f_{3} \|_{L^{p}(0,T;L^{q}(\mf))} +  \|g \|_{F^{(1-1/q)/2}_{p,q}(0,T;L^{q}(\partial \mf))} \\ + \| g \|_{L^{p}(0,T;W^{1-1/q,q}(\partial \mf))} \Big). 
\end{multline}
 
As for the previous step, we are going to apply \cite[Theorem 2.3]{DenkHieberPruss07}.
The principal symbol associated with the operator $\vartheta \mapsto - \dfrac{\kappa}{c_{v} \rho^{0} \delta^0} \div \left( \mathbb{A}^0 \nabla \vartheta \right)$
is 
\begin{equation*}
a_{0}(\cdot, \xi) =  \dfrac{\kappa}{c_{v} \rho^{0} \delta^0} \mathbb{A}^0 \xi \cdot \xi
\end{equation*}
and from \cref{lem:zeta-reg} it satisfies $a_{0}(\cdot,\xi)\geq c^1>0$ for $\xi$ such that $|\xi|=1$.
This shows condition $\bf{(E)}$ (ellipticity of the interior symbol) of  \cite{DenkHieberPruss07}.

Due to Theorem 10.4 in \cite[p.145]{Wloka}, the above operator is properly elliptic and following Example 11.6
in \cite[pp.160-161]{Wloka}), we see that the Lopatinskii–Shapiro condition $\bf{(LS)}$ holds true.

Finally, applying again \cref{lem:zeta-reg} and using that $q>3$, we can verify that $\bf{(SD1)}$ and $\bf{(SB1)}$ hold true. 

Thus all the conditions of \cite[Theorem 2.3]{DenkHieberPruss07} are satisfied.  Finally, to obtain the estimate \eqref{est:heat} with constant independent of $T$ we can proceed as  \cite[Proposition 2.2]{HMTT}. 

\underline{\bf Step 4:} 
it only remains to prove the estimate for $\rho.$  It follows from  $v \in W^{1,2}_{p,q}((0,T) ; \mf)^{3}$ and \cref{lem:zeta-reg} that
the system \eqref{NL1.0} admits a unique solution $\rho \in W^{1,p}(0,T;W^{1,q}(\mf))$ and there exists a constant $C$ independent of $T$ such  that 
\begin{equation}
\|\rho\|_{W^{1,p}(0,T;W^{1,q}(\mf))} \leqslant C \left( \|v\|_{W^{1,2}_{p,q}((0,T) ; \mf)^{3}} + \|\widehat{\rho}^{0}\|_{W^{1,q}(\mf)} + \|f_{1}\|_{L^{p}(0,T;W^{1,q}(\mf))} \right).
\end{equation}
Combining {\bf Step 1} to {\bf Step 4}, we deduce the result.
\end{proof}

\subsection{Proof of \cref{th:local-2}} \label{sec:23}
Here, we show the local in time existence of solutions for \eqref{NL0.0}--\eqref{Jacobi}. For this, we notice that 
a solution of  \eqref{NL0.0}--\eqref{H} is a solution of \eqref{NL1.0}--\eqref{NL1.3} such that the source terms satisfy
$$
\vect{f_{1},f_{2}, f_{3}, g, h} = \vect{F_{1},F_{2}, F_{3}, G, H},
$$
where $F_{1}, F_{2}, F_{3}, G$ and $H$ are given by \eqref{defF1}-\eqref{H}. This suggests to prove \cref{th:local-2} by showing that the following mapping admits a 
fixed point:
\begin{equation}\label{Xi}
\Xi_{T,R} : \mathcal{B}_{T,R} \longrightarrow \mathcal{B}_{T,R}, 
	\quad \vect{f_{1}, f_{2}, f_{3}, g, h} \longmapsto \vect{F_{1}, F_{2}, F_{2},  G,  H}, 
\end{equation}
where
 $$
\mathcal{B}_{T,R}=\left\{  \vect{f_{1}, f_{2}, f_{3}, g, h} \in \mathcal{R}_{T,p,q}  \ ; \  \|\vect{f_{1}, f_{2}, f_{3},g,h} \|_{\mathcal{R}_{T,p,q}} \leq R \right\}
$$
(recall that $\mathcal{R}_{T,p,q}$ is defined by \eqref{com0.5}) and where 
$\vect{\rho, v, \vartheta, \eta}$ is the solution of \eqref{NL1.0}---\eqref{NL1.3}  associated with $\vect{f_{1}, f_{2}, f_{3}, g, h}$ and with initial conditions 
$\vect{\rho^{0}, v^{0}, \vartheta^{0},\eta_{1}^{0}, \eta_{2}^{0}} \in \mathcal{I}_{p,q}$. More precisely, we take $R$ large enough so that
\begin{equation}\label{lininisour-est}
\| \vect{\rho^{0}, v^{0}, \vartheta^{0},\eta_{1}^{0}, \eta_{2}^{0}} \|_{\mathcal{I}_{p,q}} \leq R,
\end{equation}
and we assume  \eqref{pqpqpq} \eqref{rho0pos},  \eqref{eta10pos} and \eqref{geo02} so that we can apply 
\cref{thm-linear}:
the system \eqref{NL1.0}--\eqref{NL1.3} admits a unique solution $(\rho, v,  \vartheta,\eta) \in \mathcal{W}_{T,p,q}$
and 
$$
\left\| \vect{ \rho,v, \vartheta,\eta} \right\|_{\mathcal{W}_{T,p,q}}
\leq
C  \Big(  \| \vect{ \rho^{0}, v^{0}, \vartheta^{0}, \eta_{1}^{0}, \eta_{2}^{0}} \|_{\mathcal{I}_{p,q}}  +  \| \vect{f_{1},f_{2}, f_{3}, g, h} \|_{\mathcal{R}_{T,p,q}} \Big).
$$
To prove \cref{th:local-2}, we need to show that, for $T$ small enough, the mapping $\Xi_{T,R}$ is well-defined, that 
$\Xi_{T,R}(\mathcal{B}_{T,R}) \subset \mathcal{B}_{T,R}$ and ${{\Xi}_{T,R}}|_{\mathcal{B}_{T,R}}$ is a strict contraction.

In this proof, we write $C_R$ for any positive constant of the form $C(1+R^N)$ for $N\in \mathbb{N}$, with $C$ a constant that only depends on the geometry and on the physical parameters, and in particular independent of $T$.
In particular the above inequality can be written as
\begin{equation}\label{noli01}
\|\rho\|_{W^{1,p}(0,T;W^{1,q}(\mf))} +  \|v\|_{W^{1,2}_{p,q} ((0,T) ; \mf)^{3}}    
	+ \|\vartheta\|_{W^{1,2}_{p,q}((0,T) ; \mf)}   + \|\eta\|_{W^{2,4}_{p,q}((0,T) ; \ms)}
	\leq C_R.
\end{equation}
We are going to use several times that since $q>3$, $W^{1,q}(\mf)$ is an algebra and $W^{1,q}(\mf)\subset L^{\infty}(\mf)$.
We also have that $W^{1-\frac{1}{q},q}(\partial \mf)\subset L^{\infty}(\partial \mf)$.

We also recall the following elementary inequalities:
\begin{equation}
\| f \|_{L^{p}(0,T)} \leqslant T^{\frac{1}{p} - \frac{1}{r}} \| f \|_{L^{r}(0,T)} \quad (f \in L^{r}(0,T))
\quad \text{if } \ r > p,  \label{est:lplr} 
\end{equation}
\begin{equation}
\| f -f(0)\|_{L^{\infty}(0,T)} \leqslant T^{\frac{1}{p'}} \| f \|_{W^{1,p}(0,T)} \quad (f \in W^{1,p}(0,T)) \quad \text{if } \ \frac{1}{p} + \frac{1}{p'} = 1. 
\label{est:liw1p}
\end{equation}
In particular, we deduce from \eqref{noli01} and the above inequality
\begin{equation} \label{est:d-id}
\| \rho- \rho^{0} \|_{L^{\infty}(0,T;W^{1,q}(\mf))} \leqslant C_R T^{\frac{1}{p'}}, \quad \| \rho \|_{L^{\infty}(0,T;W^{1,q}(\mf))} \leqslant  C_R.
\end{equation}
The above estimate with \eqref{est:lplr} yields
\begin{equation}
\| \rho \|_{L^{p}(0,T;W^{1,q}(\mf))} \leqslant  C_R T^{\frac{1}{p}}. \label{est:lpdma}
\end{equation}
Since $2 < p < \infty$, one has
$B^{2(1-1/p)}_{q,p}(\mf) \hookrightarrow W^{1,q}(\mf)$. Therefore,
using  \eqref{noli01} and \eqref{embama}, we obtain
\begin{equation}
\| v \|_{L^{\infty}(0,T;W^{1,q}(\mf))^{3}}+ \|  \vartheta  \|_{L^{\infty}(0,T;W^{1,q}(\mf))} \leqslant  C_R. \label{est:livt}
\end{equation}

Using  \eqref{est:liw1p} and \eqref{Jacobi}, we deduce successively 
\begin{equation}\label{noli02}
\| X\|_{W^{1,p}(0,T;W^{2,q}(\mf))^{3}} \leqslant  C_R,
\quad 
\| X- X^{0} \|_{L^{\infty}(0,T;W^{2,q}(\mf))^{3}} \leqslant C_R T^{\frac{1}{p'}}, \quad \| X\|_{L^{\infty}(0,T;W^{2,q}(\mf))^{3}} \leqslant  C_R.
\end{equation}
Since $X^0$ is a $C^1$-diffeomorphism, we deduce from the above estimates that $X$ is a $C^1$-diffeomorphism 
for $T$ small enough. Moreover, by combining the above estimates with \cref{lem:zeta-reg} and with \eqref{mat-Azv}, we also deduce
\begin{equation}\label{noli03}
\left\| \mathbb{B}_X \right\|_{W^{1,p}(0,T;W^{1,q}(\mf))^{9}} \leqslant C_R,
\quad
\left\| \mathbb{B}_X-\mathbb{B}^0 \right\|_{L^{\infty}(0,T;W^{1,q}(\mf))^{9}} \leqslant C_R T^{\frac{1}{p'}},
\quad
\left\| \mathbb{B}_X \right\|_{L^{\infty}(0,T;W^{1,q}(\mf))^{9}} \leqslant C_R,
\end{equation}
\begin{equation}\label{noli05}
\left\| \delta_X\right\|_{W^{1,p}(0,T;W^{1,q}(\mf))} \leqslant C_R,
\quad
\left\| \delta_X-\delta^0 \right\|_{L^{\infty}(0,T;W^{1,q}(\mf))} \leqslant C_R T^{\frac{1}{p'}},
\quad 
\left\| \delta_X\right\|_{L^{\infty}(0,T;W^{1,q}(\mf))} \leqslant C_R,
\end{equation}
and in particular, there exists $c_0$ depending on $\eta_1^0$ such that for $T$ small enough,
\begin{equation}\label{noli04}
\delta_X\geq c_0>0.
\end{equation}
We thus deduce
\begin{equation}\label{noli07}
\left\| \frac 1{\delta_X}\right\|_{W^{1,p}(0,T;W^{1,q}(\mf))} \leqslant C_R,
\quad
\left\| \frac 1{\delta_X}- \frac 1{\delta^0} \right\|_{L^{\infty}(0,T;W^{1,q}(\mf))} \leqslant C_R T^{\frac{1}{p'}},
\quad 
\left\| \frac 1{\delta_X}\right\|_{L^{\infty}(0,T;W^{1,q}(\mf))} \leqslant C_R.
\end{equation}
Using the above estimates and \eqref{mat-Azv}, we also obtain
\begin{equation}\label{noli06} 
\left\| \mathbb{A}_{X}\right\|_{W^{1,p}(0,T;W^{1,q}(\mf))^{9}} \leqslant C_R,
\quad
\left\| \mathbb{A}_{X}-\mathbb{A}^0 \right\|_{L^{\infty}(0,T;W^{1,q}(\mf))^{9}} \leqslant C_R T^{\frac{1}{p'}},
\quad
\left\| \mathbb{A}_{X}\right\|_{L^{\infty}(0,T;W^{1,q}(\mf))^{9}} \leqslant C_R.
\end{equation}

We are now in position to estimate the non linear terms in \eqref{defF1}-\eqref{H}.
From the above estimates, we deduce 
\begin{equation}\label{noli99}
\left\| F_1( \rho,  v, \vartheta, \eta)\right\|_{L^{p}(0,T;W^{1,q}(\mf))} 
+
\left\| F_2( \rho,  v, \vartheta, \eta)\right\|_{L^{p}(0,T;L^{q}(\mf))^{3}} 
+
\left\| F_3( \rho,  v, \vartheta, \eta)\right\|_{L^{p}(0,T;L^{q}(\mf))} 
\leqslant C_R T^{\frac{1}{p}}.
\end{equation}
By using the trace theorems, we also have
\begin{equation}\label{noli98}
\left\| G( \rho,  v, \vartheta, \eta)\right\|_{L^{p}(0,T;W^{1-1/q,q}(\partial \mf))} 
+
\left\| H( \rho,  v, \vartheta, \eta)\right\|_{L^{p}(0,T;L^{q}(\ms))} 
\leqslant C_R T^{\frac{1}{p}}.
\end{equation}
It only remains to estimate $G$ given by \eqref{G} in $F^{(1-1/q)/2}_{p,q}(0,T;L^{q}(\partial \mf))$.
First, using \cite[Proposition 6.4]{DenkHieberPruss07}, since  $\vartheta\in W^{1,2}_{p,q}((0,T) ; \mf)$, we have that
$$
\forall i,j, \quad \frac{\partial \vartheta}{\partial y_j}n_i \in F^{(1-1/q)/2}_{p,q}(0,T;L^{q}(\partial \mf)), \quad 
\left\| \frac{\partial \vartheta}{\partial y_j}n_i  \right\|_{F^{(1-1/q)/2}_{p,q}(0,T;L^{q}(\partial \mf))}\leq C_R.
$$
Then we apply the general result \cite[Proposition 2.7]{HMTT} with $s=(1{-}\frac{1}{q})/2$, 
$U_{1} = U_{3} = L^{q}(\partial \mf)$, $U_{2} = W^{1-\frac{1}{q},q}(\partial \mf)$.
Note that since $2 < p < \infty$, we have the condition $s + \frac{1}{p} < 1$. From 
\cite[Proposition 2.7]{HMTT}, we deduce that
for some positive constant $\delta$,
\begin{multline}\label{noli97}
\left\| \left(\mathbb{A}^0-\mathbb{A}_X\right) \nabla \vartheta\cdot n\right\|_{F^{(1-1/q)/2}_{p,q}(0,T;L^{q}(\partial \mf))}
\\
\leq C T^{\delta}    \| \mathbb{A}^0-\mathbb{A}_X \|_{W^{1,p}(0,T;W^{1-\frac{1}{q},q}(\partial \mf))} 
\sum_{i,j} \left\| \frac{\partial \vartheta}{\partial y_j}n_i \right\|_{F^{(1-1/q)/2}_{p,q}(0,T;L^{q}(\partial \mf))}
\leq C_R T^{\delta}.
\end{multline}
Combining \eqref{noli99}, \eqref{noli98}, \eqref{noli97}, we deduce 
\begin{equation}\label{ptfixe1}
\left\| \Xi_{T,R}(f_{1}, f_{2}, f_{3}, g, h)\right\|_{\mathcal{R}_{T,p,q}}  \leq C_R T^{\delta}
\end{equation}
for some power $\delta>0$. Thus for $T$ small enough, $\Xi_{T,R}(\mathcal{B}_{T,R}) \subset \mathcal{B}_{T,R}$.

To show that ${{\Xi}_{T,R}}|_{\mathcal{B}_{T,R}}$ is a strict contraction, we proceed similarly: we consider 
$$
\vect{f_{1}^{(i)}, f_{2}^{(i)}, f_{3}^{(i)}, g^{(i)}, h^{(i)}} \in \mathcal{B}_{T,R}, \quad i=1,2
$$
and we denote by
$\vect{\rho^{(i)}, v^{(i)}, \vartheta^{(i)}, \eta^{(i)}}$ the solution of \eqref{NL1.0}---\eqref{NL1.3}  associated with 
$$
\vect{f_{1}^{(i)}, f_{2}^{(i)}, f_{3}^{(i)}, g^{(i)}, h^{(i)}} \in \mathcal{R}_{T, p, q}
\quad \text{and}
\quad
\vect{\rho^{0}, v^{0}, \vartheta^{0},\eta_{1}^{0}, \eta_{2}^{0}} \in \mathcal{I}_{p,q}.
$$ 
We also write
$$
\vect{f_{1}, f_{2}, f_{3}, g, h}=\vect{f_{1}^{(1)}, f_{2}^{(1)}, f_{3}^{(1)}, g^{(1)}, h^{(1)}}-\vect{f_{1}^{(2)}, f_{2}^{(2)}, f_{3}^{(2)}, g^{(2)}, h^{(2)}},
$$
$$
\vect{\rho, v, \vartheta, \eta}=\vect{\rho^{(1)}, v^{(1)}, \vartheta^{(1)}, \eta^{(1)}}-\vect{\rho^{(2)}, v^{(2)}, \vartheta^{(2)}, \eta^{(2)}}.
$$
We can apply \cref{thm-linear} and deduce that 
\begin{multline}\label{noli10}
\|\rho\|_{W^{1,p}(0,T;W^{1,q}(\mf))} +  \|v\|_{W^{1,2}_{p,q} ((0,T) ; \mf)^{3}}    
	+ \|\vartheta\|_{W^{1,2}_{p,q}((0,T) ; \mf)}   + \|\eta\|_{W^{2,4}_{p,q}((0,T) ; \ms)}
\\	
\leq C \norm{ \vect{f_{1},f_{2}, f_{3}, g, h} }_{\mathcal{R}_{T,p,q}},
\end{multline}
and since the initial conditions of $\vect{\rho, v, \vartheta, \eta}$ are null, we can apply \eqref{est:liw1p}:
\begin{equation} \label{est:d-id-lip}
\| \rho\|_{L^{\infty}(0,T;W^{1,q}(\mf))} \leqslant C T^{\frac{1}{p'}} \norm{ \vect{f_{1},f_{2}, f_{3}, g, h} }_{\mathcal{R}_{T,p,q}}.
\end{equation}
We deduce similarly that
\begin{equation}\label{noli12}
\| X^{(1)}- X^{(2)} \|_{L^{\infty}(0,T;W^{2,q}(\mf))} \leqslant C T^{\frac{1}{p'}}\| \vect{f_{1},f_{2}, f_{3}, g, h} \|_{\mathcal{R}_{T,p,q}},
\end{equation}
and we obtain similar estimates for 
$\mathbb{B}_{X^{(1)}}-\mathbb{B}_{X^{(2)}}$, $\mathbb{A}_{X^{(1)}}-\mathbb{A}_{X^{(2)}}$, $\delta_{X^{(1)}}-\delta_{X^{(2)}}$.
Proceeding as above, we deduce that $F_{1}, F_{2}, F_{3}, G$ and $H$ given by \eqref{defF1}-\eqref{H} satisfy
\begin{multline}\label{noli89}
\left\| F_1( \rho^{(1)},  v^{(1)}, \vartheta^{(1)}, \eta^{(1)})-F_1( \rho^{(2)},  v^{(2)}, \vartheta^{(2)}, \eta^{(2)})\right\|_{L^{p}(0,T;W^{1,q}(\mf))} 
\\
+\left\| F_2( \rho^{(1)},  v^{(1)}, \vartheta^{(1)}, \eta^{(1)})-F_2( \rho^{(2)},  v^{(2)}, \vartheta^{(2)}, \eta^{(2)})\right\|_{L^{p}(0,T;L^{q}(\mf))^{3}} 
\\
+\left\| F_3( \rho^{(1)},  v^{(1)}, \vartheta^{(1)}, \eta^{(1)})-F_3( \rho^{(2)},  v^{(2)}, \vartheta^{(2)}, \eta^{(2)})\right\|_{L^{p}(0,T;L^{q}(\mf))} 
\\
+\left\| G( \rho^{(1)},  v^{(1)}, \vartheta^{(1)}, \eta^{(1)})-G( \rho^{(2)},  v^{(2)}, \vartheta^{(2)}, \eta^{(2)})\right\|_{L^{p}(0,T;W^{1-1/q,q}(\partial \mf))\cap F^{(1-1/q)/2}_{p,q}(0,T;L^{q}(\partial \mf))} 
\\
+\left\| H( \rho^{(1)},  v^{(1)}, \vartheta^{(1)}, \eta^{(1)})-H( \rho^{(2)},  v^{(2)}, \vartheta^{(2)}, \eta^{(2)})\right\|_{L^{p}(0,T;L^{q}(\ms))} 
\leqslant C_R T^{\delta}\| \vect{f_{1},f_{2}, f_{3}, g, h} \|_{\mathcal{R}_{T,p,q}}
\end{multline}
for some positive constant $\delta$. Thus taking $T$ small enough, we deduce that
${{\Xi}_{T,R}}|_{\mathcal{B}_{T,R}}$ is a strict contraction and this ends the proof of the theorem.
\qed

\section{Global in time existence} \label{sec_glo}
In this section we prove \cref{th:global-main}.
\subsection{Change of variables and Linearization}\label{sec_chgvar2}
As in the first part of this work, in order to show global existence in time we use a change of variables to write 
the system \eqref{eq:main01}--\eqref{eq:force-fsi} in the fixed spatial domain 
$\mathcal{F}$. We consider the same transformation as in \cref{sec:cov-local}, that is 
$X$ is defined by \eqref{X-def}. Note that \eqref{mainsmall} for $R$ small enough yields condition \eqref{geo02}.
However, we modify \eqref{lpq01} since we 
linearize here the system around the constant steady state $\vect{\orh, 0, \ot, 0},$ with $\orh, \ot\in \mathbb{R}_+^*$:
\begin{equation}\label{lpq01-g}
  \rho(t,y)  = \widetilde \rho(t,X(t,y))- \orh , \qquad  v (t,y)   = \widetilde  v(t,X(t,y)), \quad
  \vartheta(t,y)  = \widetilde  \vartheta(t,X(t,y))- \ot,
\end{equation}
for $(t,y) \in (0,T) \times \mf.$  In particular,
\begin{equation}
\label{lpq02-g} 
\widetilde \rho(t,x)   =  \orh+ \rho(t,Y(t,x)), \quad  \widetilde  v(t,x)  = v(t,Y(t,x)),\quad
\widetilde  \vartheta(t,x)   = \ot+  \vartheta(t,Y(t,x)),
\end{equation}
for $(t,x) \in (0,T) \times \mf(\eta(t))$. 

This change of variables transforms \eqref{eq:main01}--\eqref{eq:force-fsi} into the following system for
$\vect{ \rho, v, \vartheta, \eta}$:
\begin{equation} \label{eq:global-01a} 
\left\{
\begin{array}{ll}
\partial_{t} \rho + \orh \div v = F_{1}( \rho, v, \vartheta,\eta)  & \mbox{ in } (0,\infty) \times \mf, \\
\partial_{t} v -  \dfrac{1}{\orh} \div \overline{\mathbb{T}}(\rho, v, \vartheta) = F_{2}( \rho, v, \vartheta,\eta)  & \mbox{ in } (0,\infty) \times \mf, \\
\partial_{t} \vartheta - \overline{\kappa} \Delta \vartheta = F_{3}( \rho, v, \vartheta,\eta) & \mbox{ in } (0,\infty) \times \mf, \\
\partial_{tt} \eta + \Delta_{s}^{2} \eta - \Delta_{s} \partial_{t} \eta =  - \overline{\mathbb{T}}(\rho, v, \vartheta)e_{3} \cdot e_{3} 
+  H( \rho, v, \vartheta,\eta) & \mbox{ in } (0,\infty) \times \ms,
\end{array}
\right.
\end{equation}
\begin{equation} \label{eq:global-01b} 
\left\{
\begin{array}{ll}
v = 0 & \mbox{ on } (0,\infty) \times \Gamma_{F}, \\
v = \partial_{t} \eta e_{3} & \mbox{ on } (0,\infty) \times \Gamma_{S}, \\
\dfrac{\partial \vartheta}{\partial n} = G( \rho, v, \vartheta,\eta)  & \mbox{ on } (0,\infty) \times \partial \mf,\\
\eta = \nabla_{s} \eta \cdot n_{S} = 0 & \mbox{ on } (0,\infty) \times \partial \ms,
\end{array}
\right.
\end{equation}
\begin{equation} \label{eq:global-01c} 
\left\{
\begin{array}{ll}
\eta(0,\cdot) = \eta_{1}^{0}, \quad \partial \eta(0,\cdot) = \eta_{2}^{0} & \mbox{ in } \ms, \\
\rho(0,\cdot) = \rho^{0}, \quad v(0,\cdot) = v^{0}, \qquad \vartheta(0,\cdot) = \vartheta^{0} & \mbox{ in } \mf,  
\end{array}
\right.
\end{equation}
where 
\begin{equation}
\overline{\mathbb{T}}(\rho, v, \vartheta) = 
2\mu\mathbb{D} v+ \left(\alpha \div v -R_0 \ot \rho -R_0 \orh \vartheta\right) I_3,
\end{equation}
\begin{equation}\label{kappabar}
\overline{\kappa}=\frac{\kappa}{c_{v}\orh}
\end{equation}
\begin{equation}\label{ci-G}
\rho^{0}  = \widetilde \rho^{0}\circ X^0 - \orh, \qquad  v^{0}  =  \widetilde v^{0}\circ X^0, \qquad \vartheta^{0} = \widetilde \vartheta^{0}\circ X^0 - \ot. 
\end{equation}
The nonlinear  terms in \eqref{eq:global-01a}--\eqref{eq:global-01c} can be written as 
\begin{equation} \label{F1-g}
{F}_{1}( \rho, v, \vartheta,\eta)
= -  \rho \div  v  - ( \rho+\overline\rho) \Big(\frac{1}{\delta_{X}} \mathbb{B}_{X} - I_{3} \Big) : \nabla v, 
\end{equation}
\begin{multline} \label{F2-G}
F_{2} (\rho,  v, \vartheta, \eta) = \frac{1}{\orh} \bigg[
 -\orh (\delta_X - 1) \partial_{t} v
 -\rho \delta_X  \partial_{t} v
 +\mu \div \left(\nabla v (\mathbb{A}_X-I_3)\right)
\\
+(\mu+\alpha) \div \left(\frac{1}{\delta_X}\mathbb{B}_X(\nabla v)^{\top}\mathbb{B}_X-(\nabla v)^{\top}\right)
+R_0\mathbb{B}_X \nabla (\rho \vartheta)
+R_0(\mathbb{B}_X-I_3) (\orh \nabla \vartheta +\ot \nabla \rho)\bigg]
\end{multline}
\begin{multline} \label{F3-G}
{F}_{3}( \rho,  v, \vartheta, \eta) = \frac{1}{c_v \orh}\left[
-c_v  \delta_{X} \rho \partial_{t} \vartheta
-c_v \orh(\delta_{X} - 1) \partial_{t} \vartheta 
-R_0 \Big(\rho\vartheta + \orh \vartheta + \ot \rho \Big)  \Big( \mathbb{B}_{X} : \nabla v \Big)  
\right.  \\ \left.
 + \kappa \div \Big( (\mathbb{A}_{X} - I_{3}) \nabla\vartheta \Big) 
 +  \frac{\alpha}{\delta_{X}} \left( \mathbb{B}_{X} : \nabla v\right)^{2}   
 + \frac{2\mu}{\delta_{X}} \left|\nabla v \mathbb{B}_{X}^{\top}  +  \mathbb{B}_{X}\nabla v^{\top} \right|^2
 \right],
\end{multline}
\begin{equation} \label{G-G}
{G}( \rho,  v, \vartheta, \eta)  = \left(  I_{3} - \mathbb{A}_{X} \right) \nabla \vartheta\cdot n,
\end{equation}
\begin{multline} \label{H-G}
{H}( \rho,  v, \vartheta, \eta)  = 
 - \mu\left[\frac{1}{\delta_{X}}\left(  \nabla v \mathbb{B}_{X}^{\top} +  \mathbb{B}_{X} \nabla v^{\top} \right) 
\begin{bmatrix}  -\nabla_{s} \eta \\ 1 \end{bmatrix}  - 2\mu \mathbb{D}(v) e_{3}  \right] \cdot e_{3} 
\\
-\alpha  \Big(\frac{1}{\delta_{X}} \mathbb{B}_{X} - I_{3} \Big) : \nabla v +   R_{0} \rho\vartheta,
\end{multline}
where  $\mathbb{A}_{X},$ $\mathbb{B}_{X}$ and $\delta_{X}$ are defined in \eqref{mat-Azv}. 
The hypotheses \eqref{cc1}--\eqref{cc4} on the initial conditions are transformed into 
\eqref{eta10pos}--\eqref{ci03} and  
\begin{equation}
\rho^{0} \in W^{1,q}(\mf), \quad \min_{\overline{\mf}} \rho^{0}+\orh  > 0, \label{rho0pos-G}.
\end{equation}

Using the above change of variables, \cref{th:global-main} can be reformulated as 
\begin{Theorem} \label{th:global-new}
Assume $(p,q)$ satisfies \eqref{pqpqpq}
and assume that $\orh$ and $\ot$ are two given positive constants such that \eqref{pressure-new} holds. 
Then there exist $\beta>0$ and $R > 0$ such that, for any $\vect{\rho^{0}, v^{0}, \vartheta^{0}, \eta_{1}^{0}, \eta_{2}^{0}}$ satisfying 
\eqref{geo02}, \eqref{rho0pos-G}, \eqref{eta10pos}--\eqref{ci03} and 
\begin{equation*}
\norm{\rho^{0}}_{W^{1,q}(\mf)} + \norm{v^{0}}_{B^{2(1-1/p)}_{q,p}(\mf)^{3}} +  \norm{\vartheta^{0}}_{B^{2(1-1/p)}_{q,p}(\mf)} \\
+  \norm{\eta_{1}^{0}}_{B^{2(2-1/p)}_{q,p}(\ms)} + 
 \norm{\eta_{2}^{0}}_{B^{2(1-1/p)}_{q,p}(\ms)} \leqslant R,
\end{equation*}
the system \eqref{eq:global-01a}--\eqref{H-G} admits a unique strong solution $\vect{\rho, v, \vartheta, \eta}$ in the class of functions satisfying 
\begin{gather}
\rho \in C^0_{b}([0,\infty);W^{1,q}(\mf)), \quad \nabla \rho \in {W^{1,p}_{\beta}(0,\infty;L^{q}(\mf))}, \quad \partial_t \rho \in {L^{p}_{\beta}(0,\infty;W^{1,q}(\mf))},  
\label{reg1}\\
\vartheta \in C_{b}^{0}([0,\infty);B^{2(1-1/p)}_{q,p}(\mf)), \quad \nabla \vartheta \in L^{p}_{\beta}(0,\infty ; W^{1,q}(\mf)), \quad  
	\partial_t \vartheta \in {L^{p}_{\beta}(0,\infty ; L^q(\mf))},  
\\
v \in {W^{1,2}_{p,q,\beta} ((0,\infty) ; \mf)^{3}},  \quad  \partial_{t}\eta \in W^{1,2}_{p,q,\beta}((0,\infty) ; \ms),
\\
\eta \in C^{0}_{b}([0,\infty);B^{2(2-1/p)}_{q,p}(\ms)), \quad \eta \in L^{p}_{\beta}(0,\infty;W^{4,q}(\ms)) + L^{\infty}(0,\infty;W^{4,q}(\ms)).
\label{reg4}
\end{gather}
Moreover, 
$$
\min_{[0,\infty) \times \overline{\mf}} \rho +\orh > 0, \quad \Gamma_{0} \cap \Gamma_S(\eta(t)) = \emptyset \quad (t\in [0,\infty)),
$$
and for all $t \in [0,\infty)$, $X(t,\cdot): \mf \to \mf(\eta(t))$ is a $C^{1}$-diffeomorphism.
\end{Theorem}

The proof of \cref{th:global-new} relies on the Banach fixed point theorem and on the maximal $L^{p}-L^{q}$ regularity of a linearized system over the time interval $(0,\infty).$
In order to introduce the linearized system associated with \eqref{eq:global-01a}--\eqref{H-G}, 
we introduce the following operator 
$\mathcal{T} : W^{2,q}_{0}(\ms) \to W^{2,q}(\partial \mf)^{3}$ defined by 
\begin{equation}
(\mathcal{T}\eta)(y) = \begin{dcases}
\eta(s) e_{3} & \mbox{ if } y = (s, 0) \in \Gamma_{S}, \\ 0 & \mbox{ if } y \in \Gamma_{0}.
\end{dcases}
\end{equation}
We also write $\eta_{1} = \eta$ and $ \eta_{2} = \partial_{t} \eta$ and we consider the following system where we have replaced in \eqref{eq:global-01a}--\eqref{ci-G},
the nonlinearities $F_1,F_2,F_3,G,H$ by given source terms $f_1,f_2,f_3,g,h$:

\begin{equation} \label{eq:linear-FSa} 
\left\{
\begin{array}{ll}
\partial_{t} \rho + \orh \div v = f_1  & \mbox{ in } (0,\infty) \times \mf, \\
\partial_{t} v -  \dfrac{1}{\orh} \div \overline{\mathbb{T}}(\rho, v, \vartheta) = f_2  & \mbox{ in } (0,\infty) \times \mf, \\
\partial_{t} \vartheta - \overline{\kappa} \Delta \vartheta = f_{3} & \mbox{ in } (0,\infty) \times \mf, \\
\partial_t \eta_1-\eta_2=0& \mbox{ in } (0,\infty) \times \mf, \\
\partial_{t} \eta_2 + \Delta_{s}^{2} \eta_1 - \Delta_{s} \eta_2 =  - \overline{\mathbb{T}}(\rho, v, \vartheta)e_{3} \cdot e_{3} 
+  h & \mbox{ in } (0,\infty) \times \ms,
\end{array}
\right.
\end{equation}
\begin{equation} \label{eq:linear-FSb} 
\left\{
\begin{array}{ll}
v = \mathcal{T}\eta_2 & \mbox{ on } (0,\infty) \times \partial \mf, \\
\dfrac{\partial \vartheta}{\partial n} = g & \mbox{ on } (0,\infty) \times \partial\mf,\\
\eta_1 = \nabla_{s} \eta_1 \cdot n_{S} = 0 & \mbox{ on } (0,\infty) \times \partial \ms,
\end{array}
\right.
\end{equation}
\begin{equation} \label{eq:linear-FSc} 
\left\{
\begin{array}{ll}
\eta_1(0,\cdot) = \eta_{1}^{0}, \quad \eta_2(0,\cdot) = \eta_{2}^{0} & \mbox{ in } \ms, \\
\rho(0,\cdot) = \rho^{0}, \quad v(0,\cdot) = v^{0}, \quad \vartheta(0,\cdot) = \vartheta^{0} & \mbox{ in } \mf,  
\end{array}
\right.
\end{equation}

Our aim is to show that the linearized operator associated to the above linear system is $\mr$-sectorial in a suitable function space. 

\subsection{The fluid-structure operator} \label{sec:op-fsi}
Here we introduce the operator associated to the linear system \eqref{eq:linear-FSa}--\eqref{eq:linear-FSc}.  To this aim, we first define 
\begin{equation}\label{Au}
\mathcal{D}(A_{\text v}) = \Big\{ v \in W^{2,q}(\mf)^{3} \mid v =0 \mbox{ on } \partial \mf \Big\},
\quad
A_{\text v} = \frac{\mu}{\overline\rho}\Delta + \frac{\alpha + \mu}{\overline\rho} \nabla \div,
\end{equation}
and
\begin{equation}\label{Atheta}
\mathcal{D}(A_{\vartheta}) = \left\{ \vartheta \in W^{2,q}(\mf) \mid \frac{\partial \vartheta}{\partial n} = 0\mbox{ on } \partial \mf \right\},
\quad A_{\vartheta} = \overline \kappa\Delta.
\end{equation}
From \cite[Theorem 1.4]{ST04}, $A_{\text v}$ is an isomorphism from $\mathcal{D}(A_{\text v})$ onto $L^{q}(\mf)^{3}$ for any $q\in (1,\infty)$.
Using trace properties, this allows us to introduce the operator
\begin{equation} \label{eq:Dv}
D_{\text v} \in \mathcal{L}(W^{2,q}_{0}(\ms); W^{2,q}(\mf)^{3}), 
\end{equation}
where $w=D_{\text{v}} g$ is the solution to the system 
\begin{equation} \label{eq:lift0}
\begin{dcases}
- \frac{\mu}{\orh} \Delta w  - \frac{\alpha + \mu}{\orh} \nabla(\div w)  = 0 & \mbox{ in }  \mf, \\
w = \mathcal{T}g  & \mbox{ on } \partial \mf.
\end{dcases}
\end{equation}
By a standard transposition method, the operator $D_{\text v}$ can be extended as a bounded operator from $L^{q}(\ms)$ to $L^{q}(\mf)^{3}.$ 

Using the above definitions and recalling the definitions  \eqref{eq:xs}, \eqref{op:As} of $A_{S}$ and $\mx_{S}$, we can write
\eqref{eq:linear-FSa}--\eqref{eq:linear-FSc} as follows (in the case $g=0$): 
\begin{equation}
\frac{d}{dt} \begin{bmatrix}
\rho \\ v \\ \vartheta \\ \eta_{1} \\ \eta_{2}
\end{bmatrix} 
= 
\mathcal{A}_{FS} \begin{bmatrix}
\rho \\ v \\ \vartheta \\ \eta_{1} \\ \eta_{2}
\end{bmatrix}
+\begin{bmatrix}
f_1 \\ f_2 \\ f_3 \\ 0 \\ h
\end{bmatrix}, 
\qquad 
\begin{bmatrix}
\rho \\ v \\ \vartheta \\ \eta_{1} \\ \eta_{2}
\end{bmatrix} (0) = \begin{bmatrix}
\rho^{0} \\ v^{0} \\ \vartheta^{0} \\ \eta_{1}^{0} \\ \eta_{2}^{0}
\end{bmatrix},
\end{equation}
where $\mathcal{A}_{FS} : \mathcal{D}(\mathcal{A}_{FS}) \to \mx$ is defined by 
\begin{equation} \label{eq:mx}
\mx = W^{1,q}(\mf) \times L^{q}(\mf)^{3} \times L^{q}(\mf) \times W^{2,q}_{0}(\ms) \times L^{q}(\ms),
\end{equation}
\begin{equation}
\mathcal{D}(\mathcal{A}_{FS}) = \Big\{[\rho, v,  \vartheta, \eta_{1}, \eta_{2}]^{\top} \in W^{1,q}(\mf) \times W^{2,q}(\mf)^{3} \times \mathcal{D}(A_{\vartheta}) \times \mathcal{D}(A_{S}) \mid v - D_{\textnormal{v}} \eta_{2} \in \mathcal{D}(A_{\textnormal{v}}) \Big\},
\end{equation}
and 
\begin{equation*}
\mathcal{A}_{FS} = \mathcal{A}^{0}_{FS} + \mathcal{B}_{FS},
\end{equation*}
with 
\begin{equation}
\mathcal{A}_{FS}^0 = \begin{bmatrix}
\rho \\ v \\ \vartheta \\ \eta_{1} \\ \eta_{2}
\end{bmatrix} = 
 \begin{bmatrix}
- \overline\rho \div v\\
A_{\textnormal{v}}  (v-D_{\textnormal{v}} \eta_2)\\
A_{\vartheta} \vartheta\\ 
\eta_2\\
-\Delta_{s}^{2} \eta_1+\Delta_{s}\eta_2
\end{bmatrix}
\quad
\text{and}
\quad
\mathcal{B}_{FS} 
\begin{bmatrix}
\rho \\ v \\ \vartheta \\ \eta_{1} \\ \eta_{2}
\end{bmatrix} = 
\begin{bmatrix}
0 \\ - \dfrac{R_0\ot}{\orh} \nabla \rho - R_0 \nabla \vartheta  \\ 0 \\ 0 \\ - \overline{\mathbb{T}}(\rho, v, \vartheta)e_{3} \cdot e_{3}
\end{bmatrix}.
\end{equation}

We recall that the definition of a $\mr$-sectorial operator is given in \cref{defmr}.
We now prove the following theorem : 
\begin{Theorem} \label{th:R-sec}
Let $1 < q< \infty.$ Then there exists $\gamma > 0$ such that $\mathcal{A}_{FS} - \gamma$ is an $\mr$-sectorial operator in $\mx$ of angle $ \beta > \pi/2.$
\end{Theorem}
\begin{proof}
In order to prove the theorem, we first combine  \cite[Theorem 2.5]{ShibaEno13}, \cite[Theorem 8.2]{DenkHieberPruss} and \cite[Theorem 5.1]{DenkSchnaubelt}: 
there exist $\gamma > 0$ and $ \beta > \pi/2$ such that 
the operators $A_{\textnormal{v}}-\gamma$, 
$A_{\vartheta}-\gamma$ and $A_{S}-\gamma$ are $\mr$-sectorial operators of angle $\beta$.

Second, standard calculation shows that for $\lambda\in \gamma+\Sigma_\beta$ (see \eqref{sigma}),
\begin{equation*}
\lambda (\lambda I  - \mathcal{A}^{0}_{FS})^{-1}  = \begin{bmatrix}
\Id & -\orh \div (\lambda I - A_{\textnormal{v}})^{-1} & 0 
	& \orh \div A_{\textnormal{v}} (\lambda I - A_{\textnormal{v}})^{-1} \widetilde{D_{\textnormal{v}}}  (\lambda I - A_{S})^{-1}\\
0 & \lambda(\lambda I - A_{\textnormal{v}})^{-1} & 0 & -A_{\textnormal{v}} (\lambda I - A_{\textnormal{v}})^{-1} \widetilde{D_{\textnormal{v}}} \lambda (\lambda I - A_{S})^{-1} \\ 
0 & 0 &  \lambda(\lambda I - A_{\vartheta})^{-1} & 0 \\
0 & 0 &   0 & \lambda(\lambda I - A_{S})^{-1} 
\end{bmatrix}, 
\end{equation*}
where  $\ds \widetilde{D_{\textnormal{v}}} [\eta_{1}, \eta_{2}]^{\top} = D_{\textnormal{v}} \eta_{2}.$ 
Using the properties of $\mathcal{R}$-boundedness recalled in \cref{sec_back}, we deduce that
$\mathcal{A}_{FS}^{0} - \gamma$ is $\mr$-sectorial operator in $\mx$ of angle $\beta.$ 
Note that in instance, we can write
$$
\div A_{\textnormal{v}} (\lambda I - A_{\textnormal{v}})^{-1} \widetilde{D_{\textnormal{v}}}  (\lambda I - A_{S})^{-1}
=-\div \widetilde{D_{\textnormal{v}}}  (\lambda I - A_{S})^{-1}
+\div  (\lambda I - A_{\textnormal{v}})^{-1} \widetilde{D_{\textnormal{v}}} \lambda  (\lambda I - A_{S})^{-1}
$$
and then use that 
$D_{\text v} \in \mathcal{L}(W^{2,q}_{0}(\ms); W^{2,q}(\mf)^{3})\cap \mathcal{L}(L^{q}_{0}(\ms); L^{q}(\mf)^{3}).$

Next, using trace results, for $s \in (1/q,1)$ there exists a constant $C$ such that
$$
\norm{\mathcal{B}_{FS}[\rho, v, \vartheta, \eta_{1}, \eta_{2}]^{\top} }_{\mx}  
	\leqslant C \left( \norm{\rho}_{W^{1,q}(\mf)}  + \norm{v}_{W^{1+s,q}(\mf)^{3}} + \norm{\vartheta}_{W^{1+s,q}(\mf)}   \right)
	\quad
	[\rho, v, \vartheta, \eta_{1}, \eta_{2}]^\top \in \mathcal{D}(\mathcal{A}_{FS}).
$$
Since the embedding $W^{1+s,q}(\mf) \hookrightarrow W^{2,q}(\mf)$ is compact for $s \in (1/q,1),$ 
for any $\varepsilon > 0$ there exists $C(\varepsilon) > 0$ such that 
\begin{equation} 
 \norm{\mathcal{B}_{FS}[\rho, v, \vartheta, \eta_{1}, \eta_{2}]^{\top} }_{\mx} 
 	\leqslant \varepsilon   \norm{\mathcal{A}^{0}_{FS}[\rho, v, \vartheta, \eta_{1}, \eta_{2}]^{\top} }_{\mx} 
	+ C(\varepsilon)  \norm{[\rho, v, \vartheta, \eta_{1}, \eta_{2}]^{\top} }_{\mx}.
\end{equation}
Finally using \cref{pr:perturb} we conclude the proof of the theorem. 
\end{proof}

\subsection{Exponential stability of the fluid-structure semigroup}  \label{sec:exp-fsi}
The aim of this subsection is to show that the operator $\mathcal{A}_{FS}$ generates an analytic semigroup of negative type in the following
subspace of $\mathcal{X}:$
\begin{equation} \label{eq:Xm}
\mathcal{X}_{\textnormal{m}} = \left\{ [f_{1}, f_{2}, f_{3},h_{1}, h_{2}]^{\top} \in \mx \ ; \  \int_{\mf} f_{1} \rd y \; +  \orh \int_{\ms} h_{1} \rd s = 0, \; \int_{\mf} f_{3} \rd y = 0 \right\}.
\end{equation}
We can verify that $\mx_{\textnormal{m}}$ is invariant under $\left( e^{t \mathcal{A}_{FS}}\right)_{t \geqslant 0}$. 
Therefore we can consider the restriction of $\mathcal{A}_{FS}$ to the domain $\mathcal{D}(\mathcal{A}_{FS}) \cap \mx_{\textnormal{m}}$ (\cite[Definition 2.4.1]{TW09}).
For this operator, we have the following result:
\begin{Theorem} \label{th:exp-stab}
Let $1 < q < \infty.$ The part of $\mathcal{A}_{FS}$ in $\mx_{\textnormal{m}}$ generates an exponentially stable semigroup $\left( e^{t \mathcal{A}_{FS}}\right)_{t \geqslant 0}$ on $\mx_{\textnormal{m}} :$ there exists constants $C > 0$ and $\beta_{0} > 0$ such that 
\begin{equation}
\norm{e^{t \mathcal{A}_{FS}} [\rho^{0}, v^{0}, \vartheta^{0}, \eta_{1}^{0}, \eta_{2}^{0}]^{\top}}_{\mx} \leqslant C e^{-\beta_{0} t} 
\norm{ [\rho^{0}, v^{0}, \vartheta^{0}, \eta_{1}^{0}, \eta_{2}^{0}]^{\top}}_{\mx},  \qquad (t \geqslant 0),
\end{equation}
for all $[\rho^{0}, v^{0}, \vartheta^{0}, \eta_{1}^{0}, \eta_{2}^{0}]^{\top} \in \mx_{\textnormal{m}}.$ 
\end{Theorem}
To show the above theorem, it sufficient to show that 
$\mathbb{C}^+\subset \rho({\mathcal{A}_{FS}}_{|\mathcal{D}(\mathcal{A}_{FS}) \cap \mx_{\textnormal{m}}}).$
We thus consider the following resolvent problem 
\begin{equation} \label{eq:resolvent-0}
\begin{dcases}
\lambda \rho + \orh \div v = f_{1}  & \mbox{ in } \mf, \\
\lambda v -  \dfrac{1}{\orh}\div \overline{\mathbb{T}} (\rho, v, \vartheta) = f_{2} & \mbox{ in }  \mf, \\
\lambda \vartheta - \overline{\kappa} \Delta \vartheta = f_{3} & \mbox{ in }  \mf, \\
v = \mathcal{T} \eta_{2}, \quad 
\frac{\partial \vartheta}{\partial n} = 0 & \mbox{ on }  \partial\mf, \\
\lambda \eta_{1} - \eta_{2} = h_{1} & \mbox{ in } \ms, \\
\lambda \eta_{2} + \Delta_{s}^{2} \eta_{1} - \Delta_{s}  \eta_{2} =  - \overline{\mathbb{T}}(\rho, v, \vartheta)e_{3} \cdot e_{3} +  h_{2} & \mbox{ in }  \ms, \\ 
\eta_{1} = \nabla_{s} \eta_{1} \cdot n_{S} = \eta_2=0 & \mbox{ on }  \partial \ms.
\end{dcases} 
\end{equation}

\begin{Remark}
If $\lambda =0,$ integrating the first and third equation of  \eqref{eq:resolvent-0} and using the boundary conditions of $v$ and $\vartheta$ we obtain 
\begin{equation*}
 \int_{\mf} f_{1} \rd y \; +  \orh \int_{\ms} h_{1} \rd s = 0 \mbox{ and }  \; \int_{\mf} f_{3} \rd y = 0.
\end{equation*}
Therefore, in order to study exponential stability of the semigroup it is necessary to consider the space $\mx_{\textnormal{m}}$ instead of $\mx.$
\end{Remark}

\begin{proof}
Assume $\lambda \in \mathbb{C}^{+}$ and $[f_{1}, f_{2}, f_{3} , h_{1}, h_{2}]^\top\in \mx_{\textnormal{m}}.$
We need to show that the system \eqref{eq:resolvent-0} admits a unique solution 
$[\rho, v, \vartheta, \eta_{1}, \eta_{2}]^\top\in \mathcal{D}(\mathcal{A}_{FS}) \cap \mx_{\textnormal{m}}$ together with an estimate 
\begin{equation*}
\norm{[\rho, v, \vartheta, \eta_{1}, \eta_{2}]^\top}_{\mathcal{D}(\mathcal{A}_{FS})}\leqslant C \norm{[f_{1}, f_{2}, f_{3} , h_{1}, h_{2}]^\top}_{\mx}.
\end{equation*}
The proof is divided into several parts. 

{\bf Step 1: Uniqueness.} Let us assume that $[\rho, v, \vartheta, \eta_{1}, \eta_{2}]^\top\in \mathcal{D}(\mathcal{A}_{FS})\cap \mx_{\textnormal{m}}$ solves the system \eqref{eq:resolvent-0} with $[f_{1}, f_{2}, f_{3}, h_{1}, h_{2}]^{\top} = 0$.  We notice that
\begin{equation} \label{all-L2}
\vect{\rho, v, \vartheta, \eta_{1}, \eta_{2}} \in W^{1,2}(\mf) \times W^{2,2}(\mf)^{3} \times W^{2,2}(\mf) \times W^{4,2}(\ms) \times W^{2,2}(\ms). 
\end{equation}
If $q \geqslant 2$ then it is a consequence of H\"older's inequality.  Else, $1 < q < 2$ and we take $\lambda_{0} \in \rho(\mathcal{A}_{FS})$  
to rewrite \eqref{eq:resolvent-0}  as 
\begin{equation*}
\left(\lambda_{0} - \mathcal{A}_{FS}\right) [\rho, v, \vartheta, \eta_{1}, \eta_{2}]^{\top} = \left(\lambda_{0} - \lambda \right) [\rho, v, \vartheta, \eta_{1}, \eta_{2}]^{\top}.
\end{equation*}
Since $W^{2,q}(\mf) \hookrightarrow L^{2}(\mf)$ and $W^{2,q}(\ms) \hookrightarrow L^{2}(\ms),$ we deduce \eqref{all-L2} from the the invertibility of the operator $(\lambda_{0} - \mathcal{A}_{FS}).$

Multiplying $\eqref{eq:resolvent-0}_{3}$ by $\vartheta$, we obtain after integration by parts 
\begin{equation*}
 \lambda \int_{\mf} |\vartheta|^{2} \ \rd y +  \overline{\kappa} \int_{\mf} |\nabla \vartheta|^{2} \ \rd y = 0.
\end{equation*}
Since $\Re \lambda \geqslant 0$ and $\ds \int_{\mf} \vartheta  \ \rd y= 0,$ we obtain $\vartheta = 0.$

Next, multiplying  $\eqref{eq:resolvent-0}_{2}$ by $v,$ $\eqref{eq:resolvent-0}_{6}$ by $\eta_{2},$  after integration by parts and taking the real part, we deduce 
\begin{multline*}
\frac{R_0 \ot}{\orh} (\Re \lambda)  \int_{\mf} |\rho|^{2} \ \rd y  
+\orh (\Re \lambda)  \int_{\mf} |v|^{2} \ \rd y 
+ 2\mu \int_{\mf} |\mathbb{D} v|^{2} \ \rd y + \alpha \int_{\mf} (\div v)^{2} \ \rd y \\
+ \Re\lambda \int_{\ms} |\eta_{2}|^{2} \ \rd s + \Re\lambda \int_{\ms} |\Delta_{s} \eta_{1}|^{2} + \int_{\ms} |\nabla_{s} \eta_{2}|^{2} = 0.
\end{multline*}
Since $\Re \lambda \geqslant 0,$ using $\eqref{visco-relation}$ and using the boundary conditions we obtain $v = \eta_{2} =0$ and that $\rho$ is a constant.
Using that $\vect{\rho, v, \vartheta, \eta_{1}, \eta_{2}} \in \mx_{\textnormal{m}}$ we deduce that $\eta_{1}$ solves 
\begin{equation} \label{eq:eta1-bis}
\begin{dcases}
 \Delta_{s}^{2} \eta_{1}  +  \frac{R_0\ot \orh}{|\mf|} \int_{\ms} \eta_{1} \rd s =  0 & \mbox{ in }  \ms, \\ 
\eta_{1} = \nabla_{s} \eta_{1} \cdot n_{S} = 0 & \mbox{ on }  \partial \ms. 
\end{dcases} 
\end{equation}
Multiplying the first equation of the above system by $\eta_1$ and integrating by parts, we deduce that $\eta_1=0$ and that $\rho=0$.

{\bf Step 2. Existence for $\lambda =0$.} We consider the system \eqref{eq:resolvent-0} with $\lambda = 0.$
It can be written as follows
$$
\eta_2=-h_1 \ \text{in} \ \ms,
$$
$$
- \overline{\kappa} \Delta \vartheta = f_{3}  \mbox{ in }  \mf, \quad \frac{\partial \vartheta}{\partial n} = 0  \mbox{ on }  \partial\mf, 
\quad \int_{\mf} \vartheta \ dy =0,
$$
\begin{equation} \label{eq:resolvent-0----}
\begin{dcases}
- \mu \Delta v+R_0 \ot \nabla \rho = \orh f_{2} +\frac{\alpha+\mu}{\orh}\nabla f_1-R_0\orh\nabla \vartheta & \mbox{ in }  \mf, \\
 \div v = \frac{1}{\orh} f_{1}  & \mbox{ in } \mf, \\
v = -\mathcal{T} h_{1} & \mbox{ in } \partial \mf,
\end{dcases} 
\end{equation}
\begin{equation} \label{eq:resolvent-0----2}
\begin{dcases}
 \Delta_{s}^{2} \eta_{1}  =  - \overline{\mathbb{T}}(\rho, v, \vartheta)e_{3} \cdot e_{3} - \Delta_{s}  h_1+  h_{2} & \mbox{ in }  \ms, \\ 
\eta_{1} = \nabla_{s} \eta_{1} \cdot n_{S} =0 & \mbox{ on }  \partial \ms.
\end{dcases} 
\end{equation}
\begin{equation}\label{eq:resolvent-0----3}
\int_{\mf} \rho \rd y +  \orh \int_{\ms} \eta_{1} \rd s = 0,
\end{equation}

We can solve the two first equations and obtain the existence and uniqueness of $\vartheta \in W^{2,q}(\mf)$ and $\eta_{2} \in W^{2,q}_{0}(\ms)$ 
and we have the following estimate 
\begin{equation*}
\norm{\vartheta}_{W^{2,q}(\mf)} \leqslant C  \norm{f_{3}}_{L^{q}(\mf)}, \qquad \norm{\eta_{2}}_{W^{2,q}(\ms)} = \norm{h_{1}}_{W^{2,q}(\ms)}.
\end{equation*}

Using that $[f_{1}, f_{2}, f_{3} , h_{1}, h_{2}]^\top\in \mx_{\textnormal{m}},$ we can solve \eqref{eq:resolvent-0----} (see, for instance \cite[Proposition 2.3, p.35]{Temam})
and we obtain the existence and uniqueness of $(\rho, v) \in \left(W^{1,q}(\mf)/\mathbb{R}\right) \times W^{2,q}(\mf)^{3}$ with the following estimate 
\begin{equation*}
\norm{v}_{W^{2,q}(\mf)^{3}} + \norm{\rho}_{W^{1,q}(\mf)/\mathbb{R}} \leqslant \left( \norm{f_{1}}_{W^{1,q}(\mf)} + \norm{f_{2}}_{L^{q}(\mf)^{3}} + \norm{f_{3}}_{L^{q}(\mf)} + \norm{h_{1}}_{W^{2,q}_{0}(\ms)} \right).
\end{equation*}
Then we decompose $\rho=\rho_{\textnormal{m}}+\rho_{\textnormal{avg}}$, with
$$
\rho_{\textnormal{avg}}=\frac{1}{|\mf|} \int_{\mf} \rho \ dy=-\frac{\orh}{|\mf|} \int_{\ms} \eta_{1} \rd s 
$$
and we can rewrite \eqref{eq:resolvent-0----2} as
\begin{equation} \label{eq:resolvent-0----8}
\begin{dcases}
 \Delta_{s}^{2} \eta_{1}+\frac{R_0\ot \orh}{|\mf|} \int_{\ms} \eta_{1} \rd s 
 	 =  - \overline{\mathbb{T}}(\rho_{\textnormal{m}}, v, \vartheta)e_{3} \cdot e_{3} - \Delta_{s}  h_1+  h_{2} & \mbox{ in }  \ms, \\ 
\eta_{1} = \nabla_{s} \eta_{1} \cdot n_{S} =0 & \mbox{ on }  \partial \ms.
\end{dcases} 
\end{equation}
Using the Fredholm alternative, the above system admits a unique solution $\eta_{1} \in W^{4,q}(\ms)$ and
$$
\|\eta_1\|_{W^{4,q}(\ms)} \leq C \norm{[f_{1}, f_{2}, f_{3} , h_{1}, h_{2}]^\top}_{\mx}.
$$

{\bf Step 3. Existence for $\lambda \in \mathbb{C}^{+}, \lambda \neq 0.$} 
By setting $\ds \rho = \frac{1}{\lambda} (f_{1} - \orh \div v),$ the system \eqref{eq:resolvent-0} can be rewritten as 
\begin{equation} \label{eq:resolvent-2}
\begin{dcases}
\lambda v -  \dfrac{1}{\orh}\div \widehat{\mathbb{T}}_{\lambda}  (v, \vartheta)  =  \widehat{f}_{2} & \mbox{ in }  \mf, \\
\lambda \vartheta - \overline{\kappa} \Delta \vartheta = f_{3} & \mbox{ in }  \mf, \\
v = \mathcal{T} \eta_{2}, \quad 
\frac{\partial \vartheta}{\partial n} = 0 & \mbox{ on }  \partial\mf, \\
\lambda \eta_{1} - \eta_{2} = h_{1} & \mbox{ in } \ms, \\
\lambda \eta_{2} + \Delta_{s}^{2} \eta_{1} - \Delta_{s}  \eta_{2} =  - \widehat{\mathbb{T}}_{\lambda}( v, \vartheta)e_{3} \cdot e_{3} +  \widehat{h}_{2} & \mbox{ in }  \ms, \\ 
\eta_{1} = \nabla_{s} \eta_{1} \cdot n_{S} = \eta_2=0 & \mbox{ on }  \partial \ms.
\end{dcases} 
\end{equation}
where 
\begin{gather*}
 \widehat{\mathbb{T}}_{\lambda}( v, \vartheta) 
 =  2\mu \mathbb{D} (v) + \left(\left( \alpha + \frac{R_0\ot\orh}{\lambda} \right)\div v  - R_0\orh \vartheta \right) I_{3}, \\
 \widehat{f}_{2} = f_{2} - \frac{R_0\ot}{\lambda\orh} \nabla f_{1}, \qquad \widehat{h}_{2} = h_{2} +\frac{R_0\ot}{\lambda} f_{1}|_{\ms}.
\end{gather*}
Let us set $\widehat{\mx} =  L^{q}(\mf)^{3} \times L^{q}(\mf) \times W^{2,q}_{0}(\ms) \times L^{q}(\ms).$   We define (see \eqref{Au})
\begin{equation*}
\mathcal{D}(A_{\textnormal{v}, \lambda}) = \mathcal{D}(A_{\textnormal{v}}),\quad 
A_{\textnormal{v}, \lambda} = \frac{\mu}{\orh} \Delta + \left( \frac{\alpha + \mu}{\orh} + \frac{R_{0}\orh\ot}{\lambda}\right) \nabla \div. 
\end{equation*}
In view of \cite[Theorem 1.4]{ST04} and of the Fredholm theorem, for each $\lambda$ with $\Re \lambda \geq 0,$ $ A_{\textnormal{v}, \lambda}$ is an isomorphism from $\mathcal{D}( A_{\textnormal{v}, \lambda})$ onto $L^{q}(\mf)^{3}$ for any $q \in (1, \infty).$  Let $D_{\textnormal{v}, \lambda} \in \mathcal{L}(W^{2,q}_{0}(\ms), W^{2,q}(\mf)^{3})$
defined by $D_{\textnormal{v}, \lambda} g = w,$ where $w$ is the solution to the problem 
\begin{equation*}
\begin{dcases}
-\frac{\mu}{\orh} \Delta w - \left( \frac{\alpha + \mu}{\orh} + \frac{R_{0}\orh\ot}{\lambda}\right) \nabla (\div w) = 0 & \mbox{ in } \mf, \\
w = \mathcal{T} g & \mbox{ on } \partial \mf.
\end{dcases}
\end{equation*}
We introduce the unbounded operator $\mathcal{A}_{\lambda} : \mathcal{D}(\mathcal{A}_{\lambda}) \to \widehat{\mx}$ defined by 
\begin{equation*}
\mathcal{D}(\mathcal{A}_{\lambda}) = \left\{ [v, \vartheta, \eta_{1}, \eta_{2}]^{\top} \in W^{2,q}(\mf)^{3} \times \mathcal{D}(A_{\vartheta}) \times \mathcal{D}(A_{S}) \mid v - D_{\textnormal{v}, \lambda} \eta_{2} \in \mathcal{D}(A_{\textnormal{v}, \lambda})  \right\},
\end{equation*}
and 
\begin{equation*}
\mathcal{A}_{\lambda} 
\begin{bmatrix} v \\ \vartheta \\ \eta_{1} \\ \eta_{2}  \end{bmatrix} = 
\begin{bmatrix} A_{\textnormal{v}, \lambda} (v - D_{\textnormal{v}, \lambda} \eta_{2})  - R_{0} \nabla \vartheta  \\  A_{\vartheta} \vartheta \\ \eta_{2}  \\
-\Delta_{s}^{2} \eta_{1} + \Delta_{s} \eta_{2} - \widehat{\mathbb{T}}_{\lambda}(v, \vartheta) e_{3} \cdot e_{3}
\end{bmatrix}.
\end{equation*}
With the above notations, the system \eqref{eq:resolvent-2} can be written as 
\begin{equation}\label{eq:Alambda}
(\lambda I  - \mathcal{A}_{\lambda}) [v, \vartheta, \eta_{1}, \eta_{2}]^{\top}  =  \left[\widehat{f}_{2}, f_{3}, h_{1}, \widehat{h}_{2} \right]^{\top}. 
\end{equation}
Proceeding as in the proof of \cref{th:R-sec}, one can show the existence of $\widetilde{\lambda}\in \rho(\mathcal{A}_{\lambda})$. Using that 
$\mathcal{A}_{\lambda}$ has compact resolvent and the Fredholm alternative theorem, 
the existence and uniqueness of a solution to the system \eqref{eq:resolvent-2} are equivalent. 
Let us consider a solution of \eqref{eq:resolvent-2}  with $\left[\widehat{f}_{2}, f_{3}, h_{1}, \widehat{h}_{2} \right]^{\top}=0$. As in Step 1, we can deduce that
\begin{equation} \label{all-L2---}
[v, \vartheta, \eta_{1}, \eta_{2}]^\top \in W^{2,2}(\mf)^{3} \times W^{2,2}(\mf) \times W^{4,2}(\ms) \times W^{2,2}(\ms). 
\end{equation}
Then $\vartheta=0$ and multiplying \eqref{eq:resolvent-2} by $v$ and by $\eta_2$, we deduce as in Step 1 that 
$$
[v, \vartheta, \eta_{1}, \eta_{2}]^\top =0. 
$$
This completes the proof of the proposition. 
\end{proof}

\subsection{Maximal $L^{p}$-$L^{q}$ regularity of the linear system} \label{sec:max-lin-g}
Assume 
\begin{equation}\label{tak1.4}
p,q\in (1,\infty),\quad
\frac{1}{p} + \frac{1}{2q} \neq 1,\quad
\frac{1}{p} + \frac{1}{2q} \neq \frac12.
\end{equation}
Note that \eqref{pqpqpq} implies \eqref{tak1.4}.
In order to show the maximal $L^{p}$-$L^{q}$ regularity of the system \eqref{eq:linear-FSa}--\eqref{eq:linear-FSc},
we first introduce the following decomposition: for any $f\in L^1(\mf)$,
\begin{equation}\label{tak1.8}
f = f_{\textnormal{m}} + f_{\textnormal{avg}}, \quad\text{with}\quad  \int_{\mf} f_{\textnormal{m}} \ \rd y = 0, \quad f_{\textnormal{avg}}= |\mf|^{-1} \int_{\mf} f(y) \ \rd y.
\end{equation}
We use the same decomposition and the same notation for $L^1(\partial \mf)$ and $L^{1}(\ms).$

Let us recall some standard results on the heat equation and on the linearized compressible Navier-Stokes system:
\begin{Lemma}\label{lemdeb}
There exists $\beta_1>0$ such that, for any $\beta\in (0,\beta_1)$ and for any $\eta_{2,\dagger}\in W^{2,4}_{p,q,\beta} ((0,\infty); \ms)$ with
$$
\eta_{2,\dagger}(0,\cdot)\equiv 0,
$$ 
the following linear system
\begin{equation} \label{NSC} 
\left\{
\begin{array}{ll}
\partial_{t} \rho_{\dagger} + \orh \div v_{\dagger} = 0  & \mbox{ in } (0,\infty) \times \mf, \\
\partial_{t} v_{\dagger} -  \dfrac{1}{\orh} \div \overline{\mathbb{T}}(\rho_{\dagger}, v_{\dagger}, 0) = 0  & \mbox{ in } (0,\infty) \times \mf,
\\
v_{\dagger} = \mathcal{T}\eta_{2,\dagger} & \mbox{ on } (0,\infty) \times \partial \mf, 
\\
\rho_{\dagger}(0,\cdot) = 0, \quad v_{\dagger}(0,\cdot) = 0 & \mbox{ in } \mf.
\end{array}
\right.
\end{equation}
admits a unique solution
\begin{gather}
\rho_{\dagger} = \rho_{\dagger,\textnormal{m}} + \rho_{\dagger,\textnormal{avg}},  \quad 
\rho_{\dagger,\textnormal{m}}  \in W^{1,p}_{\beta}(0,\infty;W^{1,q}(\mf)), \quad 
\partial_{t}  \rho_{\dagger,\textnormal{avg}} \in L^{p}_\beta(0,\infty),  \label{regrhod}
\\
v_{\dagger} \in  W^{1,2}_{p,q,\beta}((0,\infty) \times \mf). \label{regvd}
\end{gather}
Moreover, the following estimate holds 
\begin{multline}
\norm{\rho_{\dagger, \textnormal{m}}}_{W^{1,p}_{\beta}(0,\infty;W^{1,q}(\mf))} + \norm{\rho_{\dagger,\textnormal{avg}}}_{L^{\infty}(0, \infty)} +\norm{\partial_{t} \rho_{\dagger,\textnormal{avg}}}_{L^{p}_{\beta}(0, \infty)} \\
+ \norm{v_{\dagger}}_{ W^{1,2}_{p,q, \beta}((0, \infty) \times \mf)}  \leqslant C  \norm{\eta_{2,\dagger}}_{W^{2,4}_{p,q,\beta} ((0,\infty)\times \ms)}.
\end{multline}
\end{Lemma}

\begin{proof}
Let $\chi$ be the cut-off function defined in \eqref{chi-cut-off} and we define 
\begin{equation*}
w_{\dagger}(t, y_{1}, y_{2}, y_{3}) : = \chi(y_{1}, y_{2}, y_{3}) \eta_{2, \dagger}(t, y_{1}, y_{2})e_3 \quad (t, y) \in (0, \infty) \times \overline{\mf}.
\end{equation*}
Let us set $u_{\dagger} = v_{\dagger} - w_{\dagger}.$ Then $(\rho_{\dagger}, u_{\dagger})$ solves 
\begin{equation} \label{NSC-2} 
\begin{dcases}
\partial_{t} \rho_{\dagger} + \orh \div u_{\dagger} = f_{1, \dagger}  & \mbox{ in } (0,\infty) \times \mf, \\
\partial_{t} u_{\dagger} -  \dfrac{1}{\orh} \div \overline{\mathbb{T}}(\rho_{\dagger}, u_{\dagger}, 0) = f_{2, \dagger}  & \mbox{ in } (0,\infty) \times \mf,
\\
u_{\dagger} = 0 & \mbox{ on } (0,\infty) \times \partial \mf, 
\\
\rho_{\dagger}(0,\cdot) = 0, \quad v_{\dagger}(0,\cdot) = 0 & \mbox{ in } \mf,
\end{dcases}
\end{equation}
where 
\begin{equation*}
f_{1, \dagger} = -\orh \div w_{\dagger}, \quad f_{2, \dagger} = - \partial_{t} w_{\dagger} - \dfrac{1}{\orh} \div \overline{\mathbb{T}}(0, w_{\dagger}, 0).
\end{equation*} 
It is easy to see that 
\begin{equation*}
\norm{f_{1, \dagger}}_{L^{p}_{\beta}(0,\infty;W^{1,q}(\mf))} + \norm{f_{2}}_{L^{p}_{\beta}(0,\infty;L^{q}(\mf))} \leqslant C  \norm{\eta_{2,\dagger}}_{W^{2,4}_{p,q,\beta} ((0,\infty)\times \ms)}, 
\end{equation*}
for any $\beta > 0.$ We look for a solution to the system \eqref{NSC-2} of the form $\rho_{\dagger} = \rho_{\dagger,\textnormal{m}} + \rho_{\dagger,\textnormal{avg}},$ where $(\rho_{\dagger, \textnormal{m}}, u_{\dagger})$ solves the system \eqref{NSC-2} with $f_{1, \dagger}$  replaced by $f_{1,\dagger, \textnormal{m}}$ and $\rho_{\dagger,\textnormal{avg}} = \ds \int_{0}^{t} f_{1,\dagger, \textnormal{avg}} (s) \ \rd s.$ By \cite[Theorem 2.9]{ShibaEno13}, there exists $\beta_{1} > 0$ such that for any $\beta\in (0,\beta_1)$,  
$(f_{1,\dagger, \textnormal{m}}, f_{2, \dagger}) \in L^{p}_{\beta}(0,\infty;W^{1,q}(\mf)) \times L^{p}_{\beta}(0,\infty;L^{q}(\mf)),$  we have 
\begin{equation*}
\norm{\rho_{\dagger, \textnormal{m}}}_{W^{1,p}_{\beta}(0,\infty;W^{1,q}(\mf))} + \norm{v_{\dagger}}_{ W^{1,2}_{p,q, \beta}((0, \infty) \times \mf)}  \leqslant C  \norm{f_{1, \dagger}}_{L^{p}_{\beta}(0,\infty;W^{1,q}(\mf))} + \norm{f_{2}}_{L^{p}_{\beta}(0,\infty;L^{q}(\mf))}.
\end{equation*}
Combining the above estimates we obtain the conclusion of the lemma. 
\end{proof}

Combining Step 3 of the proof of \cref{thm-linear} and \cite[Proposition 6.4]{DenkHieberPruss07}, we deduce the following result:
\begin{Lemma}\label{lemheat}
Assume $\beta>0$. There exists $\gamma_{1} > 0$ such that  
for any 
$$
\vartheta^{0} \in B^{2(1-1/p)}_{q,p}(\mf),
\quad
f_3\in L^{p}_{\beta}(0,\infty;L^{q}(\mf)),
$$
$$
g \in F^{(1-1/q)/2}_{p,q,\beta}(0,\infty;L^{q}(\partial \mf)) \cap 
L^{p}_{\beta}(0,\infty;W^{1-1/q,q}(\partial \mf))
$$
with
$$
\frac{\partial \vartheta^0}{\partial n} = g(0,\cdot) \quad \mbox{ on } \partial \mf,
$$
the following heat equation 
\begin{equation}\label{heatgamma}
\begin{dcases}
\partial_{t} \vartheta_{\sharp} + \gamma_{1} \vartheta_{\sharp} - \overline{\kappa} \Delta \vartheta_{\sharp} =  f_{3}  & \mbox{ in } (0,\infty) \times \mf, \\
\frac{\partial \vartheta_{\sharp}}{\partial n} = g & \mbox{ on } (0, \infty) \times \partial \mf, \\
\vartheta_{\sharp} (0, \cdot) = \vartheta^{0} & \mbox{ in } \mf. 
\end{dcases}
\end{equation}
admits a unique solution $\vartheta_{\sharp} \in W^{1,2}_{p,q,\beta}((0,\infty) ; \mf)$. Moreover, we have the following estimate 
\begin{multline} \label{est:phi0}
\norm{\vartheta_{\sharp}}_{W^{1,2}_{p,q,\beta}((0,\infty) ; \mf)} \leqslant C \Big( \norm{\vartheta^{0}}_{B^{2(1-1/p)}_{q,p}(\mf)} 
+ \norm{f_{3}}_{L^{p}_{\beta}(0,\infty;L^{q}(\mf))} +  \|g \|_{F^{(1-1/q)/2}_{p,q,\beta}(0,\infty;L^{q}(\partial \mf))} 
\\ 
+ \| g \|_{L^{p}_{\beta}(0,\infty;W^{1-1/q,q}(\partial \mf))}\Big). 
\end{multline}
\end{Lemma}

We consider the subset of initial conditions 
\begin{multline}\label{Jpq}
\mathcal{J}_{p,q}
:=
\Bigg\{ \vect{\rho^{0}, v^{0}, \vartheta^{0}, \eta_{1}^{0}, \eta_{2}^{0}}  \in W^{1,q}(\mf) \times  B^{2(1-1/p)}_{q,p}(\mf)^{3} \times 
 B^{2(1-1/p)}_{q,p}(\mf) \times B^{2(2-1/p)}_{q,p}(\ms) \times B^{2(1-1/p)}_{q,p}(\ms) 
 \\
 \eta_{1}^{0} = \nabla_{s} \eta_{1}^{0} \cdot n_{S} = 0  \quad\mbox{ on } \ms, 
 \\
v^{0} = \mathcal{T} \eta_{2}^{0}  \mbox{ on } \partial \mf \quad 
	\text{and} \quad \eta_{2}^{0} = 0  \mbox{ on } \partial \ms \quad \mbox{if} \quad \frac{1}{p} + \frac{1}{2q} < 1, \\
\nabla_{s}\eta_{2}^{0} \cdot n_{S} = 0  \mbox{ on } \partial \ms \quad \mbox{if}  \quad \frac{1}{p} + \frac{1}{2q} < \frac{1}{2}
\Bigg\}
\end{multline}
with 
\begin{multline*}
\norm{\vect{\rho^{0}, v^{0}, \vartheta^{0}, \eta_{1}^{0}, \eta_{2}^{0}}}_{\mathcal{J}_{p,q}} := \norm{\rho^{0}}_{W^{1,q}(\mf)} + \norm{v^{0}}_{B^{2(1-1/p)}_{q,p}(\mf)^{3}} +  \norm{\vartheta^{0}}_{B^{2(1-1/p)}_{q,p}(\mf)} \\
+  \norm{\eta_{1}^{0}}_{B^{2(2-1/p)}_{q,p}(\ms)} + 
 \norm{\eta_{2}^{0}}_{B^{2(1-1/p)}_{q,p}(\ms)}.
\end{multline*}
We also consider the following subset for the source terms: 
\begin{multline}\label{com0.5-deuze}
\mathcal{R}_{p,q,\beta}^{cc} = \Big\{ [f_{1}, f_{2}, f_{3}, g, \widetilde{h}, \widehat{h}]^\top
\mid f_{1} \in  L^{p}_\beta(0,\infty,W^{1,q}(\mf)), f_{2} \in  L^{p}_\beta(0,\infty;L^{q}(\mf))^{3}, \\ 
f_{3} \in  L^{p}_\infty(0,\infty;L^{q}(\mf)),
g \in F^{(1-1/q)/2}_{p,q,\beta}(0,\infty;L^{q}(\partial \mf)) \cap L^{p}_\beta(0,\infty;W^{1-1/q,q}(\partial \mf)), \\ 
\widehat{h}\in L^{\infty}(0,\infty),\quad
\partial_t \widehat{h}\in L^p_{\beta}(0,\infty),\quad
\widetilde{h} \in L^{p}_\beta(0,\infty;L^{q}(\ms)),   
\\
\mbox{with}\ g(0,\cdot) = \frac{\partial \vartheta_{0}}{\partial n}  \mbox{ if } \displaystyle \frac{1}{p} + \frac{1}{2q} < \frac{1}{2} \Big\},
\end{multline}
with
\begin{multline*}
\|[f_{1}, f_{2}, f_{3}, g, \widetilde{h}, \widehat{h}]^\top\|_{\mathcal{R}_{p,q,\beta}^{cc}} 
=  
\|f_{1}\|_{L^{p}_\beta(0,\infty;W^{1,q}(\mf))} + \|f_{2}\|_{L^{p}_\beta(0,\infty;L^{q}(\mf))^{3}}   
+ \|f_{3}\|_{L^{p}_\beta(0,\infty;L^{q}(\mf))} \\ 
+ \|g\|_{F^{(1-1/q)/2}_{p,q,\beta}(0,\infty;L^{q}(\partial \mf)) \cap L^{p}_\beta(0,\infty;W^{1-1/q,q}(\partial \mf))} 
\\
+ \|\widetilde{h}\|_{L^{p}_\beta(0,\infty;L^{q}(\ms))}
+\| \widehat{h}\|_{L^{\infty}(0,\infty)}
+\| \partial_t \widehat{h}\|_{L^p_{\beta}(0,\infty)}. 
\end{multline*}

We take $\beta=\min(\beta_0,\beta_1)>0$ where $\beta_0$  is the constant in \cref{th:exp-stab}
and where $\beta_1$  is the constant in \cref{lemdeb}.
We decompose the solution of the system \eqref{eq:linear-FSa}--\eqref{eq:linear-FSc} as follows
\begin{equation}\label{deco++}
\rho=\rhob+\rho_{\diamond}+\rho_{\dagger},
\quad
v=v_{\diamond}+v_{\dagger},
\quad
\vartheta=\vartheta_{\diamond}+\vartheta_{\sharp}+\vartheta_{\flat},
\quad
\eta_1=\eta_{1,\diamond}+\eta_{1,\dagger},
\quad
\eta_2=\eta_{2,\diamond}+\eta_{2,\dagger},
\end{equation}
where $\vartheta_{\sharp}$ is the solution of \eqref{heatgamma}
given by \cref{lemheat}, where
\begin{equation}\label{flatflat}
\vartheta_{\flat}(t) := \int_0^t  \gamma_{1} \vartheta_{\sharp,\textnormal{avg}}(r) \ \rd r,
\quad
\rhob(t) = \frac{1}{|\mf|} \left( \int_{\mf} \rho^{0}\ \rd y + \orh \int_{\ms} \eta_{1}^{0} \ \rd s\right)+\int_{0}^{t} f_{1, \textnormal{avg}} (r) \ \rd r.
\end{equation}
where $[\rho_{\diamond},v_{\diamond},\vartheta_{\diamond},\eta_{1,\diamond},\eta_{2,\diamond}]^\top$
is solution of the following system
\begin{equation} \label{eq:linear-FS-2a} 
\left\{
\begin{array}{ll}
\partial_{t} \rho_{\diamond} + \orh \div v_{\diamond} = f_{1,\textnormal{m}}  & \mbox{ in } (0,\infty) \times \mf, \\
\partial_{t} v_{\diamond} -  \dfrac{1}{\orh} \div \overline{\mathbb{T}}(\rho_{\diamond}, v_{\diamond}, \vartheta_{\diamond}) 
	= f_2-R_0\nabla \vartheta_{\sharp} & \mbox{ in } (0,\infty) \times \mf, \\
\partial_{t} \vartheta_{\diamond} - \overline{\kappa} \Delta \vartheta_{\diamond} = \gamma_1 \vartheta_{\sharp,\textnormal{m}} & \mbox{ in } (0,\infty) \times \mf, \\
\partial_t \eta_{1,\diamond}-\eta_{2,\diamond}=0& \mbox{ in } (0,\infty) \times \mf, \\
\partial_{t} \eta_{2,\diamond} + \Delta_{s}^{2} \eta_{1,\diamond} - \Delta_{s} \eta_{2,\diamond} =  - \overline{\mathbb{T}}(\rho_{\diamond} , v_{\diamond} , \vartheta_{\diamond} )e_{3} \cdot e_{3} 
+  \widetilde{h} +R_0 \orh {\vartheta_{\sharp}}_{|\ms}& \mbox{ in } (0,\infty) \times \ms,
\end{array}
\right.
\end{equation}
\begin{equation} \label{eq:linear-FS-2b} 
\left\{
\begin{array}{ll}
v_{\diamond} = \mathcal{T}\eta_{2,\diamond} & \mbox{ on } (0,\infty) \times \partial \mf, \\
\dfrac{\partial \vartheta_{\diamond}}{\partial n} = 0& \mbox{ on } (0,\infty) \times \partial \mf,\\
\eta_{1,\diamond} = \nabla_{s} \eta_{1,\diamond} \cdot n_{S} = 0 & \mbox{ on } (0,\infty) \times \partial \ms,
\end{array}
\right.
\end{equation}
\begin{equation} \label{eq:linear-FS-2c} 
\left\{
\begin{array}{ll}
\eta_{1,\diamond}(0,\cdot) = \eta_{1}^{0}, \quad \eta_{2,\diamond}(0,\cdot) = \eta_{2}^{0} & \mbox{ in } \ms, \\
\rho_{\diamond}(0,\cdot) = \ds \rho^{0} - \rhob(0), \quad v_{\diamond}(0,\cdot) = v^{0}, \quad \vartheta_{\diamond}(0,\cdot) = 0 & \mbox{ in } \mf.
\end{array}
\right.
\end{equation}
and where $[\rho_{\dagger},v_{\dagger},\vartheta_{\dagger},\eta_{1,\dagger},\eta_{2,\dagger}]^\top$
is solution of the following system
\begin{equation} \label{eq:linear-FS-3a} 
\left\{
\begin{array}{ll}
\partial_{t} \rho_{\dagger} + \orh \div v_{\dagger} = 0  & \mbox{ in } (0,\infty) \times \mf, \\
\partial_{t} v_{\dagger} -  \dfrac{1}{\orh} \div \overline{\mathbb{T}}(\rho_{\dagger}, v_{\dagger}, \vartheta_{\dagger}) = 0  & \mbox{ in } (0,\infty) \times \mf, \\
\partial_{t} \vartheta_{\dagger} - \overline{\kappa} \Delta \vartheta_{\dagger} = 0 & \mbox{ in } (0,\infty) \times \mf, \\
\partial_t \eta_{1,\dagger}-\eta_{2,\dagger}=0& \mbox{ in } (0,\infty) \times \mf, \\
\partial_{t} \eta_{2,\dagger} + \Delta_{s}^{2} \eta_{1,\dagger} - \Delta_{s} \eta_{2,\dagger} =  - \overline{\mathbb{T}}(\rho_{\dagger} , v_{\dagger} , \vartheta_{\dagger} )e_{3} \cdot e_{3} 
+ \widehat{h}+R_0 \ot \rhob + R_0 \orh {\vartheta_{\flat}}
& \mbox{ in } (0,\infty) \times \ms,
\end{array}
\right.
\end{equation}
\begin{equation} \label{eq:linear-FS-3b} 
\left\{
\begin{array}{ll}
v_{\dagger} = \mathcal{T}\eta_{2,\dagger} & \mbox{ on } (0,\infty) \times \partial \mf, \\
\dfrac{\partial \vartheta_{\dagger}}{\partial n} = 0& \mbox{ on } (0,\infty) \times \partial \mf,\\
\eta_{1,\dagger} = \nabla_{s} \eta_{1,\dagger} \cdot n_{S} = 0 & \mbox{ on } (0,\infty) \times \partial \ms,
\end{array}
\right.
\end{equation}
\begin{equation} \label{eq:linear-FS-3c} 
\left\{
\begin{array}{ll}
\eta_{1,\dagger}(0,\cdot) = 0, \quad \eta_{2,\dagger}(0,\cdot) = 0 & \mbox{ in } \ms, \\
\rho_{\dagger}(0,\cdot) = 0, \quad v_{\dagger}(0,\cdot) = 0, \quad \vartheta_{\dagger}(0,\cdot) = 0 & \mbox{ in } \mf.
\end{array}
\right.
\end{equation}
Let us show that the decomposition \eqref{deco++} is valid.
First, we can check that 
$$
\vartheta_{\flat}, \rhob \in C^{0}_b([0, \infty)), \quad \partial_{t} \vartheta_{\flat}, \partial_t\rhob \in L^{p}_{\beta}(0, \infty).
$$
Second, for the system \eqref{eq:linear-FS-2a}--\eqref{eq:linear-FS-2c}, we
note that from \eqref{eq:Xm} and \eqref{Jpq}
$$\Big[f_{1,\textnormal{m}},   f_{2} - R_0\nabla \vartheta_{\sharp},  \gamma_{1} \vartheta_{\sharp,\textnormal{m}}, 0,  \widetilde{h} +R_0 \orh {\vartheta_{\sharp}}_{|\ms} \Big]^{\top} \in L^{p}_{\beta}(0, \infty;\mx_{\textnormal{m}}),$$
\begin{equation*}
\Big[\rho^{0}  -  \rhob(0), v^{0}, 0, \eta_{1}^{0}, \eta_{2}^{0} \Big]^{\top} \in  \left(\mx_{\textnormal{m}}, \mathcal{D}(\mathcal{A}_{FS}) \right)_{1-1/p,p}. 
\end{equation*}
From \cref{th:R-sec} and \cref{th:exp-stab} we know that $\mathcal{A}_{FS}+\beta I$ is a $\mr$-sectorial operator on $\mx_{\textnormal{m}}$ and generates an analytic exponential stable semigroup on $\mx_{\textnormal{m}}.$  Therefore, by \cref{thm:max-reg-g}, the system \eqref{eq:linear-FS-2a}--\eqref{eq:linear-FS-2c} admits a unique solution 
\begin{equation}
\Big[ \rho_{\diamond},v_{\diamond},\vartheta_{\diamond},\eta_{1,\diamond},\eta_{2,\diamond} \Big]^{\top} \in 
L^{p}_{\beta}(0, \infty;\mathcal{D}(\mathcal{A}_{FS}) \cap \mx_{\textnormal{m}}) \cap W^{1,p}_{\beta}(0,\infty; \mx_{\textnormal{m}}). 
\end{equation}
Finally, let us consider the system \eqref{eq:linear-FS-3a}--\eqref{eq:linear-FS-3c}.
Note that $\vartheta_{\dagger} \equiv 0$. Moreover 
\begin{equation} 
\frac{d}{dt} \begin{bmatrix}
\partial_{t}\rho_{\dagger} \\ \partial_{t} v_{\dagger} \\ \partial_{t}\vartheta_{\dagger} \\ \partial_{t} {\eta}_{1,\dagger} \\ \partial_{t} {\eta}_{2,\dagger}
\end{bmatrix} = \mathcal{A}_{FS} \begin{bmatrix}
\partial_{t}\rho_{\dagger} \\ \partial_{t} v_{\dagger} \\ \partial_{t}\vartheta_{\dagger} \\ \partial_{t} {\eta}_{1,\dagger} \\ \partial_{t} {\eta}_{2,\dagger}
\end{bmatrix} 
+ \begin{bmatrix} 0 \\ 0 \\ 0 \\ 0 \\ \partial_t\widehat{h}+R_0 \ot f_{1, \textnormal{avg}} + R_0 \orh \gamma_{1} \vartheta_{\sharp,\textnormal{avg}} \end{bmatrix}
, \quad \begin{bmatrix}
\partial_{t}\rho_{\dagger} \\ \partial_{t} v_{\dagger} \\ \partial_{t}\vartheta_{\dagger} \\ \partial_{t} {\eta}_{1,\dagger} \\ \partial_{t} {\eta}_{2,\dagger}
\end{bmatrix} (0) = \begin{bmatrix}
0\\ 0 \\ 0 \\ 0 \\ 0
\end{bmatrix}.
\end{equation}
Using that $\partial_t\widehat{h}+R_0 \ot f_{1, \textnormal{avg}} + R_0 \orh \gamma_{1} \vartheta_{\sharp,\textnormal{avg}} \in L^{p}_{\beta}(0,\infty),$
and combining as above 
\cref{th:R-sec}, \cref{th:exp-stab} and \cref{thm:max-reg-g}, we infer that 
\begin{equation}
\Big[ \partial_{t}\rho_{\dagger}, \partial_{t} v_{\dagger},  \partial_{t}\vartheta_{\dagger},  \partial_{t} {\eta}_{1,\dagger},  \partial_{t} {\eta}_{2,\dagger} \Big]^{\top} 
\in L^{p}_{\beta}(0, \infty;\mathcal{D}(\mathcal{A}_{FS}) \cap \mx_{\textnormal{m}}) \cap W^{1,p}_{\beta}(0,\infty; \mx_{\textnormal{m}}). 
\end{equation}
In particular,  
\begin{equation}
{\eta}_{2,\dagger} \in W^{2,4}_{p,q,\beta} ((0,\infty); \ms). 
\end{equation}
Then, we use \cref{lemdeb} to deduce $(\rho_\dagger,v_\dagger)$ satisfies \eqref{regrhod}--\eqref{regvd}.

Let us also write 
\begin{equation}\label{deco+++}
\widetilde{\rho} = \rho_{\diamond}+\rho_{\dagger, \textnormal{m}}, \quad \widehat{\rho} = \rhob + \rho_{\dagger, \textnormal{avg}}, \quad \widetilde{\vartheta} =\vartheta_{\diamond}+\vartheta_{\sharp}, \quad \widehat{\vartheta} = \vartheta_{\flat},
\end{equation}
so that
\begin{equation}\label{decomp-rt}
\rho= \widetilde{\rho} + \widehat{\rho}, \quad \vartheta = \widetilde{\vartheta} + \widehat{\vartheta}.
\end{equation}

Gathering the above properties, we have obtained the following theorem:
\begin{Theorem} \label{th:lin-lplq-global}
Assume \eqref{tak1.4}. There exists $\beta>0$
such that for any 
$$
[\rho^{0}, v^{0}, \vartheta^{0}, \eta_{1}^{0}, \eta_{2}^{0}]^\top \in \mathcal{J}_{p,q},
\quad
[f_{1}, f_{2}, f_{3}, g, \widetilde{h}, \widehat{h}]^\top \in \mathcal{R}_{p, q, \beta}^{cc},
\quad
h=\widetilde{h}+\widehat{h},
$$
the system \eqref{eq:linear-FSa}--\eqref{eq:linear-FSc} 
admits a unique solution satisfying \eqref{reg1}--\eqref{reg4} and
\begin{multline}
\|\rho\|_{L^\infty(0,\infty;W^{1,q}(\mf))} 
+
\|\nabla \rho\|_{W^{1,p}_{\beta}(0,\infty;L^{q}(\mf))^{3}} 
+
\|\partial_t \rho\|_{L^{p}_{\beta}(0,\infty;W^{1,q}(\mf))} 
+\|v\|_{W^{1,2}_{p,q,\beta} ((0,\infty) ; \mf)^{3}}    
\\
+ \norm{\vartheta}_{L^{\infty}(0,\infty;B^{2(1-1/p)}_{q,p}(\mf))}
+ \|\nabla \vartheta\|_{L^{p}_{\beta}(0,\infty ; W^{1,q}(\mf))^{3}}   
+ \|\partial_t \vartheta\|_{L^{p}_{\beta}(0,\infty ; L^q(\mf))}   
\\
+ \|\eta_1\|_{L^{\infty}(0,\infty;B^{2(2-1/p)}_{q,p}(\ms))}
+ \|\eta_2\|_{W^{1,2}_{p,q,\beta}((0,\infty) ; \ms)}
 \\ 
 \leqslant C_{L} 
 \Big(  \norm{[\rho^{0}, v^{0}, \vartheta^{0}, \eta_{1}^{0}, \eta_{2}^{0}]^\top}_{\mathcal{J}_{p,q}}  
 +  \norm{[f_{1}, f_{2}, f_{3}, g,  \widetilde{h}, \widehat{h}]^\top}_{\mathcal{R}_{p, q, \beta}^{cc}} 
 \Big). 
\end{multline}
Moreover, we can decompose the solution as \eqref{deco+++}-\eqref{decomp-rt}, with 
\begin{gather}
\widetilde{\rho} \in W^{1,p}_{\beta}(0,\infty;W^{1,q}(\mf)), \; \widetilde{\vartheta} \in W^{1,2}_{p,q,\beta}((0,\infty);\mf), \; \vect{\widehat{\rho}, \widehat{\vartheta}} \in L^{\infty}(0,\infty)^{2}, \; \vect{\partial_{t}\widehat{\rho}, \partial_{t}\widehat{\vartheta}} \in L^{p}_{\beta}(0,\infty)^{2}, 
\end{gather}
and 
\begin{multline}
\norm{\widetilde{\rho}}_{W^{1,p}_{\beta}(0,\infty;W^{1,q}(\mf))} + \norm{\widetilde{\vartheta}}_{W^{1,2}_{p,q,\beta}((0,\infty);\mf)} + \norm{\vect{\widehat{\rho}, \widehat{\vartheta}}}_{L^{\infty}(0,\infty)^{2}} + \norm{\vect{\partial_{t}\widehat{\rho}, \partial_{t}\widehat{\vartheta}}}_{L^{p}_{\beta}(0,\infty)^{2}} \\
\leq C_{L} 
 \Big(  \norm{[\rho^{0}, v^{0}, \vartheta^{0}, \eta_{1}^{0}, \eta_{2}^{0}]^\top}_{\mathcal{J}_{p,q}}  
 +  \norm{[f_{1}, f_{2}, f_{3}, g, \widetilde{h}, \widehat{h}]^\top}_{\mathcal{R}_{p, q, \beta}^{cc}} 
 \Big). 
\end{multline}
\end{Theorem}

\subsection{Proof of \cref{th:global-new}} \label{sec:36}
In this subsection, we prove \cref{th:global-new} (or equivalently \cref{th:global-main}): 
we show the existence and uniqueness of global in time solutions for the system \eqref{eq:global-01a}--\eqref{H-G} 
under a smallness assumption on the initial data. 

Let us assume the hypotheses of \cref{th:global-new}, with $\beta$ given by \cref{th:lin-lplq-global}. 
Assume
$$
[\rho^{0}, v^{0}, \vartheta^{0}, \eta_{1}^{0}, \eta_{2}^{0}]^\top\in \mathcal{J}_{p,q},
$$ 
where $\mathcal{J}_{p,q}$ is defined by \eqref{Jpq}.
For $R > 0,$ we define $\mathcal{B}_{R}$ as follows 
\begin{equation}
\mathcal{B}_{R} = \Big\{ [f_{1}, f_{2}, f_{3}, g, \widetilde{h}, \widehat{h}]^\top \in \mathcal{R}_{p, q, \beta}^{cc} \ ; \  
	\norm{\vect{f_{1}, f_{2}, f_{3}, g, \widetilde{h}, \widehat{h}}}_{\mathcal{R}_{p, q, \beta}^{cc}} \leqslant R \Big\},
\end{equation}
where $\mathcal{R}_{p, q, \beta}^{cc}$ is defined by \eqref{com0.5-deuze}. 
By using \cref{lemheat} with $f_3=0$ and $g=0$, we see that there exists a constant $C>0$ independent of $R$ such that
if
\begin{equation}\label{Rsmall0}
\norm{[\rho^{0}, v^{0}, \vartheta^{0}, \eta_{1}^{0}, \eta_{2}^{0}]^\top}_{\mathcal{J}_{p,q}} \leq CR,
\end{equation}
then $\mathcal{B}_{R}$ is a nonempty closed subset of the Banach space
\begin{multline}\label{com0.5-tres}
\mathcal{R}_{p,q,\beta} = \Big\{ [f_{1}, f_{2}, f_{3}, g, \widetilde{h}, \widehat{h}]^\top
\ ; \  f_{1} \in  L^{p}_\beta(0,\infty,W^{1,q}(\mf)), \quad f_{2} \in  L^{p}_\beta(0,\infty;L^{q}(\mf))^{3}, \\ 
f_{3} \in  L^{p}_\infty(0,\infty;L^{q}(\mf)),
\quad
g \in F^{(1-1/q)/2}_{p,q,\beta}(0,\infty;L^{q}(\partial \mf)) \cap L^{p}_\beta(0,\infty;W^{1-1/q,q}(\partial \mf)), \\
\widehat{h}\in L^{\infty}(0,\infty),\quad
\partial_t \widehat{h}\in L^p_{\beta}(0,\infty),\quad
\widetilde{h} \in L^{p}_\beta(0,\infty;L^{q}(\ms))  \Big\}.
\end{multline}

We define the map 
\begin{equation}\label{Xi-g}
\Xi_{R} : \mathcal{B}_{R} \longrightarrow \mathcal{B}_{R}, 
	\quad [f_{1}, f_{2}, f_{3}, g, \widetilde{h}, \widehat{h}]^\top \longmapsto [F_{1}, F_{2}, F_{3}, G, \widetilde{H}, \widehat{H}]^\top, 
\end{equation}
where $[\rho, v, \vartheta, \eta,\partial_t\eta]^\top$ is the solution to the system \eqref{eq:linear-FSa}--\eqref{eq:linear-FSc} 
associated with $[f_{1}, f_{2}, f_{3}, g, h]^\top$ and $[\rho^{0}, v^{0}, \vartheta^{0}, \eta_{1}^{0}, \eta_{2}^{0}]^\top,$ 
(see \cref{th:lin-lplq-global}), where $F_{1}, F_{2}, F_{3}$ and $G$ are given by \eqref{F1-g}--\eqref{G-G}
and where 
\begin{multline} \label{tilde-H}
\widetilde{H} =  - \mu\left[\frac{1}{\delta_{X}}\left(  \nabla v \mathbb{B}_{X}^{\top} +  \mathbb{B}_{X} \nabla v^{\top} \right) 
\begin{bmatrix}  -\nabla_{s} \eta \\ 1 \end{bmatrix}  - 2\mu \mathbb{D}(v) e_{3}  \right] \cdot e_{3} 
-\alpha  \Big(\frac{1}{\delta_{X}} \mathbb{B}_{X} - I_{3} \Big) : \nabla v  \\ +   R_{0} \widetilde{\rho} \widetilde{\vartheta} + R_{0} \widetilde{\rho} \widehat{\vartheta} + R_{0} \widehat{\rho} \widetilde{\vartheta},
\end{multline}
and 
\begin{equation} \label{hat-H}
\widehat{H} = R_{0} \widehat{\rho} \widehat{\vartheta}.
\end{equation}
In the above definitions, we have used the decomposition
of $\rho$ and $\vartheta$ given by \eqref{deco+++}-\eqref{decomp-rt}. We can check that $H$ defined by \eqref{H-G} satisfies
\begin{equation*}
H = \widetilde{H} + \widehat{H}.
\end{equation*}

In order to prove \cref{th:global-new}, 
it is enough to show that 
the mapping $\Xi_{R}$ is well defined, $\Xi(\mathcal{B}_{R}) \subset \mathcal{B}_{R}$ 
and $\Xi_{|\mathcal{B}_{R}}$ is a strict contraction, for $R$ small enough. 

Throughout this subsection, $C$ will be a positive constant depending on $p,q$ and $\beta$ but independent of $R,$ which may change from line to line.
To simplify the computations, we assume that $R \in (0,1).$ 

Since $2 < p< \infty$ and $3 < q < \infty,$ one has (see, for instance \cite[(7), p. 196]{Triebel})
$$
B^{2(1-1/p)}_{q,p}(\mf) \hookrightarrow W^{1,q}(\mf) \hookrightarrow L^{\infty}(\mf).
$$  
Therefore, from \cref{th:lin-lplq-global}, we obtain 
\begin{multline} \label{est-g-2}
\norm{v}_{L^{\infty}_{\beta}(0,\infty;W^{1,q}(\mf))^3} + \norm{\vartheta}_{L^{\infty}(0,\infty;W^{1,q}(\mf))} 
+ \norm{\nabla \vartheta}_{L^{p}_{\beta}(0,\infty;L^{\infty}(\mf))^{3}} 
\\
+\norm{\eta}_{L^{\infty}(0,\infty;W^{3,q}(\ms))} + \norm{\eta}_{L^{\infty}(0,\infty;C^{2}(\overline{\ms}))} \leqslant C R.
\end{multline} 
From the definition of $X^{0}$ from \eqref{geo6} and from \eqref{Rsmall0} we deduce that
\begin{equation*}
\norm{\nabla X^{0}  - I_{3}}_{W^{2,q}(\mf)^9} \leqslant C R. 
\end{equation*}
Using the above estimate and the definition of $X$ (see \eqref{Jacobi}) it follows that 
\begin{equation}\label{10:54}
\norm{\nabla X  - I_{3}}_{L^{\infty}(0,\infty;W^{1,q}(\mf))^9} \leqslant \norm{\nabla X^{0}  - I_{3}}_{W^{2,q}(\mf)^9}  
	+ C\norm{\nabla v}_{L^{p}_\beta(0,\infty;W^{1,q}(\mf))^9} \leqslant C R. 
\end{equation}
In particular, by choosing $R$ sufficiently small, we have 
\begin{equation*}
\norm{\nabla X  - I_{3}}_{L^{\infty}((0,\infty) \times \mf)^{9}} \leqslant \frac{1}{2}. 
\end{equation*}
Thus $X$ is a $C^{1}$-diffeomorphism for $R$ small enough. Moreover, by combining the above estimates  with \eqref{mat-Azv}
and using that $\partial_t X=v$, we also deduce
\begin{equation}\label{noli03-g}
\left\| \mathbb{B}_X- I_{3} \right\|_{L^{\infty}(0,\infty;W^{1,q}(\mf))^{9}}  \leqslant C R, 
\quad 
\norm{\partial_{t} \mathbb{B}_{X}}_{L^{p}(0,\infty;W^{1,q}(\mf))^{9}} \leqslant CR, 
\quad 
\left\| \mathbb{B}_X \right\|_{L^{\infty}(0,\infty;W^{1,q}(\mf))^{9}}  \leqslant C, 
\end{equation}
\begin{equation} \label{noli03-g-2}
\left\| \delta_X- 1 \right\|_{L^{\infty}(0,\infty;W^{1,q}(\mf))}\leqslant CR,
\quad  
\norm{\partial_{t} \delta_{X}}_{L^{p}(0,\infty;W^{1,q}(\mf))} \leqslant CR,  
\quad 
\left\| \delta_X \right\|_{L^{\infty}(0,\infty;W^{1,q}(\mf))}  \leqslant C.
\end{equation}

Consequently, for $R$ small enough 
\begin{equation}\label{noli04-bis}
\delta_X\geq \frac{1}{2} \quad \mbox{ for all } (t, y) \in (0,\infty) \times \mf.
\end{equation}
We thus deduce
\begin{equation}\label{noli07-g}
\left\| \frac 1{\delta_X}- 1 \right\|_{L^{\infty}(0,\infty;W^{1,q}(\mf))} \leqslant CR,
\quad 
\left\| \partial_{t} \left(\frac 1{\delta_X} \right) \right\|_{L^{p}(0,\infty;W^{1,q}(\mf))} \leqslant C R, 
\quad 
\left\| \frac 1{\delta_X} \right\|_{L^{\infty}(0,\infty;W^{1,q}(\mf))} \leqslant C.
\end{equation}
Using the above estimates and \eqref{mat-Azv}, we also obtain
\begin{gather}
\left\| \mathbb{A}_{X}- I_{3}\right\|_{L^{\infty}(0,\infty;W^{1,q}(\mf))^{9}} \leqslant CR,
\quad
\left\| \mathbb{A}_{X}\right\|_{L^{\infty}(0,\infty;W^{1,q}(\mf))^{9}} \leqslant C,  \label{noli06-g} \\ 
\norm{\partial_{t} \mathbb{A}_{X}}_{L^{p}(0,\infty;W^{1,q}(\mf))^{9}} \leqslant CR, \quad \norm{\mathbb{A}_{X} - I_{3}}_{C^{1/p'}([0,\infty);W^{1,q}(\mf))^{9}} \leqslant CR.
\end{gather}
For more details about the proof of the above estimates, we refer to \cite[Lemma 3.19]{HMTT}. 

From \eqref{G-G} and \eqref{ci03}, we notice that
\begin{equation} 
{G}( \rho,  v, \vartheta, \eta)_{|t=0}  = \left(  I_{3} - \mathbb{A}^0\right) \nabla \vartheta^0\cdot n=\frac{\partial \vartheta^0}{\partial n} \quad 
\text{on} \ \partial \mf.
\end{equation}
Using the above estimates we deduce that $F_{1}, F_{2}, F_{3}, G$
 and $\widetilde H$, $\widehat H$ defined by \eqref{F1-g}--\eqref{G-G} and \eqref{tilde-H}--\eqref{hat-H} satisfy the estimate 
 \begin{equation} \label{est-g-4}
 \norm{[F_{1}, F_{2}, F_{3}, G,   \widetilde{H}, \widehat{H}]^\top}_{\mathcal{R}_{p, q, \beta}^{cc}} \leqslant C R^{2}.
 \end{equation}
To details on the proof of \eqref{est-g-4} can be found in \cite[Proposition 3.20]{HMTT}. 
This shows that $\Xi(\mathcal{B}_{R}) \subset \mathcal{B}_{R}$ for $R$ small enough. 

To show that ${{\Xi}_{R}}|_{\mathcal{B}_{R}}$ is a strict contraction, we proceed similarly: we consider 
$$
\vect{f_{1}^{(i)}, f_{2}^{(i)}, f_{3}^{(i)}, g^{(i)},   \widetilde{h}^{(i)}, \widehat{h}^{(i)}} \in \mathcal{B}_{R}, \quad i=1,2
$$
and we denote by $[\rho^{(i)}, v^{(i)}, \vartheta^{(i)}, \eta^{(i)},\partial_t\eta^{(i)}]^\top$ the solutions to the system \eqref{eq:linear-FSa}--\eqref{eq:linear-FSa} 
associated with $[f_{1}^{(i)}, f_{2}^{(i)}, f_{3}^{(i)}, g^{(i)},   \widetilde{h}^{(i)}, \widehat{h}^{(i)}]^\top$ 
and $[\rho^{0}, v^{0}, \vartheta^{0}, \eta_{1}^{0}, \eta_{2}^{0}]^\top$ 
(see \cref{th:lin-lplq-global}).
We can thus define
$$
\vect{F_{1}^{(i)}, F_{2}^{(i)}, F_{3}^{(i)}, G^{(i)}, \widetilde{H}^{(i)}, \widehat{H}^{(i)}}:=
	\Xi_{R}\left(\vect{f_{1}^{(i)}, f_{2}^{(i)}, f_{3}^{(i)}, g^{(i)},   \widetilde{h}^{(i)}, \widehat{h}^{(i)}}\right).
$$

We also write
$$
[f_{1}, f_{2}, f_{3}, g, \widetilde{h}, \widehat{h}]^\top
=
[f_{1}^{(1)}, f_{2}^{(1)}, f_{3}^{(1)}, g^{(1)},    \widetilde{h}^{(1)}, \widehat{h}^{(1)}]^\top
-
[f_{1}^{(2)}, f_{2}^{(2)}, f_{3}^{(2)}, g^{(2)},    \widetilde{h}^{(2)}, \widehat{h}^{(2)}]^\top,
$$
$$
[\rho, v, \vartheta, \eta]^\top
=
[\rho^{(1)}, v^{(1)}, \vartheta^{(1)}, \eta^{(1)}]^\top
-
[\rho^{(2)}, v^{(2)}, \vartheta^{(2)}, \eta^{(2)}]^\top,
$$
$$
[\widetilde \rho, \widetilde \vartheta]^{\top} = [\widetilde \rho^{(1)}, \widetilde \vartheta^{(1)}]^{\top} - [\widetilde \rho^{(2)}, \widetilde \vartheta^{(2)}]^{\top}, \quad 
[\widehat \rho, \widehat \vartheta]^{\top} = [\widehat \rho^{(1)}, \widehat \vartheta^{(1)}]^{\top} - [\widehat \rho^{(2)}, \widehat \vartheta^{(2)}]^{\top}.
$$

Therefore, from \cref{th:lin-lplq-global}, we obtain 
\begin{multline} \label{est-g-2-diff}
\norm{v}_{L^{\infty}_{\beta}(0,\infty;W^{1,q}(\mf))^3} + \norm{\vartheta}_{L^{\infty}(0,\infty;W^{1,q}(\mf))} 
+ \norm{\nabla \vartheta}_{L^{p}_{\beta}(0,\infty;L^{\infty}(\mf))^{3}} 
\\
+\norm{\eta}_{L^{\infty}(0,\infty;W^{3,q}(\ms))} + \norm{\eta}_{L^{\infty}(0,\infty;C^{2}(\overline{\ms}))} 
	\leqslant C \norm{[f_{1}, f_{2}, f_{3}, g, \widetilde{h}, \widehat{h}]^\top}_{\mathcal{R}_{p, q, \beta}^{cc}}
\end{multline} 
and 
\begin{multline}
\norm{\widetilde{\rho}}_{W^{1,p}_{\beta}(0,\infty;W^{1,q}(\mf))} + \norm{\widetilde{\vartheta}}_{W^{1,2}_{p,q,\beta}((0,\infty);\mf)} + \norm{\vect{\widehat{\rho}, \widehat{\vartheta}}}_{L^{\infty}(0,\infty)^{2}} + \norm{\vect{\partial_{t}\widehat{\rho}, \partial_{t}\widehat{\vartheta}}}_{L^{p}_{\beta}(0,\infty)^{2}} \\
\leq C
 \Big( \norm{[f_{1}, f_{2}, f_{3}, g, \widetilde{h}, \widehat{h}]^\top}_{\mathcal{R}_{p, q, \beta}^{cc}} 
 \Big). 
\end{multline}

In particular, from \eqref{Jacobi}, 
\begin{equation}
\norm{\nabla X^{(1)}-\nabla X^{(2)}}_{L^{\infty}(0,\infty;W^{1,q}(\mf))^9} 
	\leqslant C \norm{[f_{1}, f_{2}, f_{3}, g, \widetilde{h}, \widehat{h}]^\top}_{\mathcal{R}_{p, q, \beta}^{cc}}.
\end{equation}
By combining the above estimates  with \eqref{mat-Azv} and with \eqref{10:54}, we deduce
\begin{multline}\label{noli03-g-diff}
\left\| \mathbb{B}_{X^{(1)}}- \mathbb{B}_{X^{(2)}} \right\|_{L^{\infty}(0,\infty;W^{1,q}(\mf))^9}  
+\norm{\partial_{t} \mathbb{B}_{X^{(1)}}- \partial_{t}\mathbb{B}_{X^{(2)}}}_{L^{p}(0,\infty;W^{1,q}(\mf))^9} 
\\
+\left\| \delta_{X^{(1)}}- \delta_{X^{(2)}} \right\|_{L^{\infty}(0,\infty;W^{1,q}(\mf))}
+\norm{\partial_{t} \delta_{X^{(1)}}- \partial_{t}\delta_{X^{(2)}}}_{L^{p}(0,\infty;W^{1,q}(\mf))}
\\
+\left\| \frac{1}{\delta_{X^{(1)}}}- \frac{1}{\delta_{X^{(2)}}} \right\|_{L^{\infty}(0,\infty;W^{1,q}(\mf))}
+\norm{\partial_{t} \left(\frac{1}{\delta_{X^{(1)}}}\right)- \partial_{t}\left(\frac{1}{\delta_{X^{(2)}}}\right)}_{L^{p}(0,\infty;W^{1,q}(\mf))}
\\
+\left\| \mathbb{A}_{X^{(1)}}- \mathbb{A}_{X^{(2)}}\right\|_{L^{\infty}(0,\infty;W^{1,q}(\mf))^{9}} 
+\norm{\partial_{t} \mathbb{A}_{X^{(1)}}- \partial_{t}\mathbb{A}_{X^{(2)}}}_{L^{p}(0,\infty;W^{1,q}(\mf))^{9}}
\\
+\norm{\partial_{t} \mathbb{A}_{X^{(1)}}- \partial_{t}\mathbb{A}_{X^{(2)}}}_{C^{1/p'}([0,\infty);W^{1,q}(\mf))^{9}}
\\
\leq CR \norm{[f_{1}, f_{2}, f_{3}, g, \widetilde{h}, \widehat{h}]^\top}_{\mathcal{R}_{p, q, \beta}^{cc}}.
\end{multline}
Using the above estimates we deduce that $F_{1}, F_{2}, F_{3}, G$
 and $\widetilde H$, $\widehat H$ defined by \eqref{F1-g}--\eqref{G-G}, \eqref{tilde-H}, \eqref{hat-H} satisfy the estimate 
\begin{multline}\label{noli89-glob}
\left\| F_1^{(1)}-F_1^{(2)}\right\|_{L^{p}_\beta(0,\infty;W^{1,q}(\mf))}
+\left\| F_2^{(1)}-F_2^{(2)}\right\|_{L^{p}_\beta(0,\infty;L^{q}(\mf))^{3}}   
+\left\| F_3^{(1)}-F_3^{(2)}\right\|_{L^{p}_\beta(0,\infty;L^{q}(\mf))}
\\
+\left\|
G_1^{(1)}-G_1^{(2)}
\right\|_{F^{(1-1/q)/2}_{p,q,\beta}(0,\infty;L^{q}(\partial \mf)) \cap L^{p}_\beta(0,\infty;W^{1-1/q,q}(\partial \mf))} 
+\left\| \widetilde H^{(1)}-\widetilde H^{(2)}
\right\|_{L^{p}_\beta(0,\infty;L^{q}(\ms))}
\\
+\left\| 
\widehat H^{(1)}-\widehat H^{(2)}
\right\|_{L^{\infty}(0,\infty)}
+\left\| \partial_t \widehat H^{(1)}-\partial_t \widehat H^{(2)}
\right\|_{L^p_{\beta}(0,\infty)}
\\
\leqslant C R \norm{[f_{1}, f_{2}, f_{3}, g, {h}]^\top}_{\mathcal{R}_{p, q, \beta}^{cc}}.
\end{multline}
This shows that $\Xi_{|\mathcal{B}_{R}}$ is a strict contraction, for $R$ small enough. This completes the proof of \cref{th:global-new} and \cref{th:global-main}.
\qed

\bibliographystyle{siam}
\bibliography{ref-cns-Lp}

\begin{thebibliography}{10}

\bibitem{AF}
{\sc R.~A. Adams and J.~J.~F. Fournier}, {\em Sobolev spaces}, vol.~140 of Pure
  and Applied Mathematics (Amsterdam), Elsevier/Academic Press, Amsterdam,
  second~ed., 2003.

\bibitem{AmannODE}
{\sc H.~Amann}, {\em Ordinary differential equations}, vol.~13 of De Gruyter
  Studies in Mathematics, Walter de Gruyter \& Co., Berlin, 1990.
\newblock An introduction to nonlinear analysis, Translated from the German by
  Gerhard Metzen.

\bibitem{Amann}
\leavevmode\vrule height 2pt depth -1.6pt width 23pt, {\em Linear and
  quasilinear parabolic problems. {V}ol. {I}}, vol.~89 of Monographs in
  Mathematics, Birkh{\"a}user Boston, Inc., Boston, MA, 1995.
\newblock Abstract linear theory.

\bibitem{AGW}
{\sc G.~Avalos, P.~G. Geredeli, and J.~T. Webster}, {\em Semigroup
  well-posedness of a linearized, compressible fluid with an elastic boundary},
  Discrete Contin. Dyn. Syst. Ser. B, 23 (2018), pp.~1267--1295.

\bibitem{BT19}
{\sc M.~Badra and T.~Takahashi}, {\em Gevrey {R}egularity for a {S}ystem
  {C}oupling the {N}avier--{S}tokes {S}ystem with a {B}eam {E}quation}, SIAM J.
  Math. Anal., 51 (2019), pp.~4776--4814.

\bibitem{BT19-bis}
{\sc M.~Badra and T.~Takahashi}, {\em {Gevrey regularity for a system coupling
  the Navier-Stokes system with a beam: the non-flat case}},
  https://hal.archives-ouvertes.fr/hal-02303258,  (2019).

\bibitem{bio-book}
{\sc T.~Bodn\'{a}r, G.~P. Galdi, and v.~Ne\v{c}asov\'{a}}, eds., {\em
  Fluid-structure interaction and biomedical applications}, Advances in
  Mathematical Fluid Mechanics, Birkh\"{a}user/Springer, Basel, 2014.

\bibitem{Bravin}
{\sc M.~Bravin}, {\em Energy equality and uniqueness of weak solutions of a
  ``viscous incompressible fluid + rigid body'' system with {N}avier
  slip-with-friction conditions in a 2{D} bounded domain}, J. Math. Fluid
  Mech., 21 (2019), pp.~Art. 23, 31.

\bibitem{BS18}
{\sc D.~Breit and S.~Schwarzacher}, {\em Compressible fluids interacting with a
  linear-elastic shell}, Arch. Ration. Mech. Anal., 228 (2018), pp.~495--562.

\bibitem{IC}
{\sc I.~Chueshov}, {\em Dynamics of a nonlinear elastic plate interacting with
  a linearized compressible viscous fluid}, Nonlinear Anal., 95 (2014),
  pp.~650--665.

\bibitem{DenkHieberPruss}
{\sc R.~Denk, M.~Hieber, and J.~Pr{{\"u}}ss}, {\em {$\mathscr R$}-boundedness,
  {F}ourier multipliers and problems of elliptic and parabolic type}, Mem.
  Amer. Math. Soc., 166 (2003), pp.~viii+114.

\bibitem{DenkHieberPruss07}
\leavevmode\vrule height 2pt depth -1.6pt width 23pt, {\em Optimal
  {$L^p$}-{$L^q$}-estimates for parabolic boundary value problems with
  inhomogeneous data}, Math. Z., 257 (2007), pp.~193--224.

\bibitem{DS19}
{\sc R.~{Denk} and J.~{Saal}}, {\em {$L^p$-theory for a fluid-structure
  interaction model}}, arXiv e-prints,  (2019), p.~arXiv:1909.09344.

\bibitem{DenkSchnaubelt}
{\sc R.~Denk and R.~Schnaubelt}, {\em A structurally damped plate equation with
  {D}irichlet-{N}eumann boundary conditions}, J. Differential Equations, 259
  (2015), pp.~1323--1353.

\bibitem{Dor91}
{\sc G.~Dore}, {\em {$L^p$} regularity for abstract differential equations}, in
  Functional analysis and related topics, 1991 ({K}yoto), vol.~1540 of Lecture
  Notes in Math., Springer, Berlin, 1993, pp.~25--38.

\bibitem{ShibaEno13}
{\sc Y.~Enomoto and Y.~Shibata}, {\em On the {$\mathscr R$}-sectoriality and
  the initial boundary value problem for the viscous compressible fluid flow},
  Funkcial. Ekvac., 56 (2013), pp.~441--505.

\bibitem{EHL}
{\sc S.~Ervedoza, M.~Hillairet, and C.~Lacave}, {\em Long-time behavior for the
  two-dimensional motion of a disk in a viscous fluid}, Comm. Math. Phys., 329
  (2014), pp.~325--382.

\bibitem{EMT}
{\sc S.~Ervedoza, D.~Maity, and M.~Tucsnak}, {\em {Large time behaviour for the
  motion of a solid in a viscous incompressible fluid}},
  https://hal.archives-ouvertes.fr/hal-02545798,  (2020).

\bibitem{FO}
{\sc F.~Flori and P.~Orenga}, {\em Fluid-structure interaction: analysis of a
  3-{D} compressible model}, Ann. Inst. H. Poincar\'{e} Anal. Non Lin\'{e}aire,
  17 (2000), pp.~753--777.

\bibitem{GGH}
{\sc M.~Geissert, K.~G\"{o}tze, and M.~Hieber}, {\em {$L^p$}-theory for strong
  solutions to fluid-rigid body interaction in {N}ewtonian and generalized
  {N}ewtonian fluids}, Trans. Amer. Math. Soc., 365 (2013), pp.~1393--1439.

\bibitem{GlassSueur}
{\sc O.~Glass and F.~Sueur}, {\em Uniqueness results for weak solutions of
  two-dimensional fluid-solid systems}, Arch. Ration. Mech. Anal., 218 (2015),
  pp.~907--944.

\bibitem{GHL}
{\sc C.~Grandmont, M.~Hillairet, and J.~Lequeurre}, {\em Existence of local
  strong solutions to fluid--beam and fluid--rod interaction systems}, in
  Annales de l'Institut Henri Poincar{\'e} C, Analyse non lin{\'e}aire,
  vol.~36, Elsevier, 2019, pp.~1105--1149.

\bibitem{biomedical}
{\sc C.~Grandmont, M.~Luk\'{a}\v{c}ov\'{a}-Medvid'ov\'{a}, and
  v.~Ne\v{c}asov\'{a}}, {\em Mathematical and numerical analysis of some {FSI}
  problems}, in Fluid-structure interaction and biomedical applications, Adv.
  Math. Fluid Mech., Birkh\"{a}user/Springer, Basel, 2014, pp.~1--77.

\bibitem{HMTT}
{\sc B.~H. Haak, D.~Maity, T.~Takahashi, and M.~Tucsnak}, {\em Mathematical
  analysis of the motion of a rigid body in a compressible
  {N}avier-{S}tokes-{F}ourier fluid}, Math. Nachr., 292 (2019), pp.~1972--2017.

\bibitem{HiebMur}
{\sc M.~Hieber and M.~Murata}, {\em The {$L^p$}-approach to the fluid-rigid
  body interaction problem for compressible fluids}, Evol. Equ. Control Theory,
  4 (2015), pp.~69--87.

\bibitem{KunstmannWeis:Pisa2001}
{\sc P.~C. Kunstmann and L.~Weis}, {\em Perturbation theorems for maximal
  {$L\sb p-$}-regularity}, Ann. Scuola Norm. Sup. Pisa Cl. Sci. (4), 30 (2001),
  pp.~415--435.

\bibitem{KunstmannWeis:Levico}
{\sc P.~C. Kunstmann and L.~Weis}, {\em Maximal {$L\sb p$}-regularity for
  parabolic equations, {F}ourier multiplier theorems and {$H\sp
  \infty$}-functional calculus}, in Functional analytic methods for evolution
  equations, vol.~1855 of Lecture Notes in Math., Springer, Berlin, 2004,
  pp.~65--311.

\bibitem{LacTak}
{\sc C.~Lacave and T.~Takahashi}, {\em Small moving rigid body into a viscous
  incompressible fluid}, Arch. Ration. Mech. Anal., 223 (2017), pp.~1307--1335.

\bibitem{JL}
{\sc J.~Lequeurre}, {\em Existence of strong solutions to a fluid-structure
  system}, SIAM J. Math. Anal., 43 (2011), pp.~389--410.

\bibitem{MT20}
{\sc D.~Maity and T.~Takahashi}, {\em {Lp theory for the interaction between
  the incompressible Navier-Stokes system and a damped beam}},
  https://hal.archives-ouvertes.fr/hal-02294097,  (2020).

\bibitem{Tucsnak17a}
{\sc D.~Maity and M.~Tucsnak}, {\em A maximal regularity approach to the
  analysis of some particulate flows}, in Particles in flows, Adv. Math. Fluid
  Mech., Birkh\"auser/Springer, Cham, 2017, pp.~1--75.

\bibitem{MT18}
{\sc D.~Maity and M.~Tucsnak}, {\em {$L^p$}-{$L^q$} maximal regularity for some
  operators associated with linearized incompressible fluid-rigid body problems
  {S}elected {R}ecent {R}esults}, in Mathematical analysis in fluid mechanics
  selected recent results, vol.~710 of Contemp. Math., Amer. Math. Soc.,
  Providence, RI, 2018, pp.~175--201.

\bibitem{SM18}
{\sc S.~{Mitra}}, {\em {Local existence of Strong solutions for a
  fluid-structure interaction model}}, arXiv e-prints,  (2018),
  p.~arXiv:1808.06716.

\bibitem{Prussbook}
{\sc J.~Pr\"{u}ss and G.~Simonett}, {\em Moving interfaces and quasilinear
  parabolic evolution equations}, vol.~105 of Monographs in Mathematics,
  Birkh\"{a}user/Springer, [Cham], 2016.

\bibitem{ST04}
{\sc Y.~Shibata and K.~Tanaka}, {\em On a resolvent problem for the linearized
  system from the dynamical system describing the compressible viscous fluid
  motion}, Math. Methods Appl. Sci., 27 (2004), pp.~1579--1606.

\bibitem{Temam}
{\sc R.~Temam}, {\em Navier-{S}tokes equations}, vol.~2 of Studies in
  Mathematics and its Applications, North-Holland Publishing Co., Amsterdam-New
  York, revised~ed., 1979.
\newblock Theory and numerical analysis, With an appendix by F. Thomasset.

\bibitem{Tri95}
{\sc H.~Triebel}, {\em Interpolation theory, function spaces, differential
  operators}, Johann Ambrosius Barth, Heidelberg, second~ed., 1995.

\bibitem{Triebel}
\leavevmode\vrule height 2pt depth -1.6pt width 23pt, {\em Theory of function
  spaces}, Modern Birkh\"auser Classics, Birkh\"auser/Springer Basel AG, Basel,
  2010.
\newblock Reprint of 1983 edition [MR0730762], Also published in 1983 by
  Birkh\"auser Verlag [MR0781540].

\bibitem{TW09}
{\sc M.~Tucsnak and G.~Weiss}, {\em Observation and control for operator
  semigroups}, Birkh\"{a}user Advanced Texts: Basler Lehrb\"{u}cher.
  [Birkh\"{a}user Advanced Texts: Basel Textbooks], Birkh\"{a}user Verlag,
  Basel, 2009.

\bibitem{Weis01}
{\sc L.~Weis}, {\em Operator-valued {F}ourier multiplier theorems and maximal
  {$L_p$}-regularity}, Math. Ann., 319 (2001), pp.~735--758.

\bibitem{Wloka}
{\sc J.~Wloka}, {\em Partial differential equations}, Cambridge University
  Press, Cambridge, 1987.
\newblock Translated from the German by C. B. Thomas and M. J. Thomas.

\end{thebibliography}

\end{document}